\documentclass[a4paper,twoside,12pt,reqno]{amsart}       
\usepackage{latexsym}
\usepackage{amssymb,amsfonts,amsmath,stmaryrd,bbm}
\usepackage{gensymb}

\usepackage[T1]{fontenc} 

\usepackage{comment}

\numberwithin{equation}{section}

\newtheorem{theorem}{Theorem}  
\newtheorem{proposition}[theorem]{Proposition}
\newtheorem{lemma}[theorem]{Lemma}
\newtheorem{corollary}[theorem]{Corollary}

\theoremstyle{definition}
\newtheorem{definition}[theorem]{Definition}

\theoremstyle{remark}
\newtheorem*{remark}{Remark}
\newtheorem*{remarks}{Remarks}

\newtheorem{example}{Example}

\def\al{\alpha}
\def\be{\beta}
\def\de{\delta}

\def\om{\omega}

\def\rh{\rho}

\def\ta{\tau}

\def\Na{\bigtriangledown}

\def\N{{\mathbb N}}
\def\Q{{\mathbb Q}}

\def\Fix{\operatorname{Fix}}

\def\fl#1{\left\lfloor#1\right\rfloor}
\def\cl#1{\left\lceil#1\right\rceil}

\def\Na{\bigtriangledown}
\def\nequiv{\hbox{${}\not\equiv{}$}}
\def\Cat{\operatorname{Cat}}

\def\qee{$\hfill{\Box}$}

\usepackage{color,stmaryrd,graphicx}
\usepackage{colortbl,soul}
\usepackage[plainpages=false,pdfpagelabels,colorlinks=true,citecolor=blue,hypertexnames=false]{hyperref}

\newcommand{\beq}{\begin{equation}}
\newcommand{\eeq}{\end{equation}}

  \newcommand{\qbinom}{\genfrac{[}{]}{0pt}{}}
\DeclareMathOperator{\id}{id}

\def\emm#1,{{\em #1}}
\def\bn{\mathbf n}

  \DeclareMathOperator{\vv}{v}
  \DeclareMathOperator{\ff}{f}
  \DeclareMathOperator{\ee}{e}

  \DeclareMathOperator{\T}{T}
  \DeclareMathOperator{\TM}{TM}
  \DeclareMathOperator{\BT}{BT}
  \DeclareMathOperator{\NCM}{NCM}

  \newcommand{\e}{\mathsf e}
  \newcommand{\m}{\mathfrak m}

  \def\vareps{\varepsilon}

\catcode`\@=11
\def\section{\@startsection{section}{1}%
  \z@{.7\linespacing\@plus\linespacing}{.5\linespacing}%
  {\normalfont\large\scshape\bfseries\centering}}

\def\subsection{\@startsection{subsection}{2}%
  \z@{.5\linespacing\@plus\linespacing}{.5\linespacing}%
  {\normalfont\bfseries\scshape}}

\def\subsubsection{\@startsection{subsubsection}{3}%
  \z@{.5\linespacing\@plus\linespacing}{-.5em}
  {\normalfont\bfseries}}
\catcode`\@=12

\begin{document}
\title[Cyclic sieving phenomena for trees and tree-rooted maps]
      {Cyclic sieving phenomena for\\ rooted plane trees and
        tree-rooted planar maps}
\author[M. Bousquet-M\'elou]{M. Bousquet-M\'elou}

\address{CNRS, LaBRI, Universit\'e de Bordeaux, 351, cours de la
  Lib\'eration, F-33405 Talence Cedex, France.
WWW: {\tt https://www.labri.fr/perso/bousquet/}.}

\author[C. Krattenthaler]{C. Krattenthaler}

\address{Fakult\"at f\"ur Mathematik, Universit\"at Wien,
Oskar-Morgenstern-Platz~1, A-1090 Vienna, Austria.
WWW: {\tt http://www.mat.univie.ac.at/\lower0.5ex\hbox{\~{}}kratt}.}

\thanks{The first author was partially supported by the ANR project
  CartesEtPlus
(ANR-23-CE48-0018).
The second author was partially supported by the Austrian
Science Foundation FWF, grant 10.55776/F1002, in the framework
of the Special Research Program ``Discrete Random Structures:
Enumeration and Scaling Limits''} 

\subjclass[2020]{Primary 05A15; Secondary 05A10}
 
\keywords{Rooted plane trees, ordered trees, tree-rooted planar maps,
non-crossing partitions, non-crossing perfect matchings, cyclic sieving
phenomenon, Catalan numbers, Narayana numbers}

\begin{abstract}
We prove cyclic sieving phenomena satisfied by corner-rooted plane trees
(alias ordered trees). The sets of rooted plane trees
that we consider are: (1) all trees with $n$~nodes;
(2) all trees with $n$~nodes and $k$~leaves;
(3) all trees with a given degree distribution of the nodes.
Moreover, we consider four different cyclic group actions:
(1) the root is moved to the next corner along a tour of the tree;
(2) only trees in which the root is at a leaf
are considered, and the action moves the root to the next
leaf; (3) only trees in which the root is at a non-leaf
are considered, and the action moves the root to the next 
non-leaf corner;
(4) only trees in which the root is at a node of degree~$\de$ are
considered, for a fixed $\de$, and the action moves the root to the next
corner of this type.

We prove a cyclic sieving phenomenon for each meaningful
combination of these sets and actions.
As a bonus,
we also establish corresponding cyclic sieving phenomena
for tree-rooted planar maps.
\end{abstract}
\maketitle

\section{Introduction}
\label{sec:intro}

The {\it cyclic sieving phenomenon} was introduced by Reiner, Stanton
and White~\cite{ReSWAA} as the property of a set of combinatorial
objects on which a cyclic group of order~$m$ acts that the number of
objects that are invariant under a given power of a generator of the
group is given by the evaluation of a fixed univariate polynomial at
the corresponding power of a primitive $m$-th root of unity. (See
Section~\ref{sec:siev} for the precise definition.) Reiner, Stanton
and White already presented numerous instances of the cyclic sieving
phenomenon in their original article~\cite{ReSWAA}. Since
then, the discovery of cyclic sieving phenomena has almost become an
industry, with numerous more having been discovered and proved since
the appearance of~\cite{ReSWAA}. A survey of cyclic sieving phenomena
published until 2011 can be found in~\cite{SagaCS}. It goes without
saying that even more have been identified since then.

Inspired by these overwhelming developments, the second author, during
his talk at the closing conference of the ANR project Cortipom in
Le Croisic in June 2025 in which he presented further cyclic sieving
phenomena~\cite{KrStAA},
put forward a new principle of combinatorial mathematics:

\bigskip
\vbox{
\centerline{\it Every family of combinatorial objects satisfies}
\centerline{\it the cyclic sieving phenomenon!}}
\bigskip

If this is so, said the first author who was in the audience,
then surely rooted plane trees satisfy the cyclic sieving
phenomenon. The particular action of a cyclic group that she
thought of was to move the root of the tree to
the next \emm corner,
(see Figure~\ref{fig:example} for an example and Section~\ref{sec:rotations} for precise definitions).
Being a fan of maps, she had in mind a possible application to cyclic sieving for tree-rooted maps.
As for evidence, she had already
done some computations for small examples. 

Surprisingly, a first literature search led to nothing, which let
us believe that no investigations had
been undertaken in this direction. As explained  in Section~\ref{sec:connections}, the suspected  cyclic sieving result for trees \emm did, already exist in fact, but only in disguise, stated in terms of different objects\footnote{We also found more recently a preprint where it \emm is, stated in terms of trees~\cite{heitsch},
  albeit with a different proof than in our article.}.

Indeed, the cyclic sieving phenomenon for rooted plane trees
is not very difficult to prove, see Theorem~\ref{thm:ord} in
Section~\ref{sec:ordinary}.
In fact, there exists even a refinement in which one fixes the number
of leaves, see Theorem~\ref{thm:ord_leaves}, which is already more demanding to
prove. In particular, to find the cyclic sieving polynomial was not
obvious since the ``$q$-ification'' of the factor~2 in the counting
formula~\eqref{eq:T(n,k)} is not canonical (however, we realised later that this result too did already exist in a different language,
cf.\ Subsection~\ref{sec:matchings}).

From here on, we became more adventurous (or one may also say:
``greedy''). What if we modify the action so that the only allowed
roots are leaves, and the group action moves the root from one leaf to
the next? Or what happens if we do the analogous construction
with non-leaves? What if we refine further and consider rooted plane trees
in which the complete degree distribution of the nodes is prescribed?
And finally, how about taking one step further and considering tree-rooted maps, as was the original motivation of the first author? 

To give it away: the above principle did not let us down.
In all instances and all (meaningful) combinations of the ingredients
(objects and actions) we identified cyclic sieving phenomena,
which we state and prove in this paper.

\medskip
\noindent{\bf Overview of results and outline of the paper.}
After a review of the relevant tree counting formulae in the next
section, and a reminder of the definition of the cyclic sieving
phenomenon in Section~\ref{sec:siev},
in Section~\ref{sec:rotations} we establish some preliminary results on the structure of trees having a rotational symmetry. In Section~\ref{sec:ordinary} we
present the aforementioned cyclic sieving results for the rotation
action that the first author had originally in mind, see
Theorems~\ref{thm:ord} and~\ref{thm:ord_leaves}.
Section~\ref{sec:external} is then devoted to the action of ``external
rotation'', by which we mean that the root is moved from one leaf to
the next. The main result in that section is Theorem~\ref{thm:ext}.
The action that we consider in Section~\ref{sec:internal} is what we call
``internal rotation'': here the root is moved from a non-leaf corner to the
next. The corresponding cyclic sieving result is presented in
Theorem~\ref{thm:int}.

The subsequent three sections concern cyclic sieving phenomena
for trees in which the degrees of the nodes are prescribed.
Section~\ref{sec:ord_deg} presents a corresponding cyclic sieving phenomenon
for the original rotation, see Theorem~\ref{thm:ord_deg}.
In Section~\ref{sec:delta} we investigate a generalisation of the external rotation that we call ``$\de$-rotation''. Here, the trees that we consider
  have a  root of degree~$\de$, and
$\de$-rotation moves the root to the next corner of degree~$\de$. The
corresponding cyclic sieving phenomenon, with prescribed vertex degrees, is presented in Theorem~\ref{thm:delta}.
Finally in Section~\ref{sec:int_deg} the action is that of internal rotation, and the corresponding cyclic sieving result is the subject of  Theorem~\ref{thm:int_deg}.

Each of the Sections~\ref{sec:ordinary}--\ref{sec:int_deg} is built according to the same scheme:
the cyclic sieving result of that
particular section is presented and proved in a theorem. The proof of
polynomiality and non-negativity of the coefficients of the cyclic
sieving polynomial is usually delayed to the first lemma following
the theorem. The evaluations of the cyclic sieving polynomial at
the relevant roots of unity are
done separately in the second lemma
following the theorem. This restricts the proofs of the theorems to the enumeration of trees fixed by a prescribed rotation. Each of the three aspects of our proofs --- enumeration, polynomiality, evaluation --- always follows the same pattern, independently of the class of trees and of the rotation under consideration. We naturally provide more details in the early proofs than in the late ones.

In the
subsequent section, Section~\ref{sec:connections}, we take a temporary break from cyclic sieving proofs, and describe connections between some of our results and earlier ones, stated in terms of non-crossing matchings, non-crossing partitions, or dissections of a polygon. For instance, 
our first result, namely Theorem~\ref{thm:ord}, and its refinement recording the number of leaves (Theorem~\ref{thm:ord_leaves})
have actually appeared earlier in terms of
non-crossing perfect matchings. Conversely, we explain how some of our results translate in these alternative languages.
At the end of that section (see Subsection~\ref{sec:part}),
we explain how
the cyclic sieving phenomena for non-crossing matchings and non-crossing
partitions can be embedded in a much more general cyclic sieving theorem
for generalised non-crossing partitions associated with reflection groups.

Finally,
in Section~\ref{sec:maps}, we return to cyclic sieving proofs,
this time for tree-rooted planar maps; see
Theorems~\ref{thm:TMij}, \ref{thm:TMn}, and~\ref{thm:TMd}.
Tree-rooted maps can indeed be
seen as amalgams of (slightly generalised) rooted plane trees and non-crossing perfect matchings.
As we show, cyclic sieving phenomena result from this structure.

\section{Tree enumeration}
\label{sec:enum}

We consider {\it rooted plane trees}
(also called {\it ordered trees}) with $n$ edges, and thus $n+1$ vertices (or \emm nodes,); see Figure~\ref{fig:rooted-tree}.
The precise meaning of ``rooted'' is that there is a {\it pointer} to
one of the {\it corners} of the tree, as is commonly done for rooted
(combinatorial) maps (see e.g.\ \cite{CaChAA}).
More precisely, we call \emm corner, a wedge between two consecutive edges incident to the same vertex. In the figures, the root corner is indicated by an arrow in the corresponding wedge; see again  Figure~\ref{fig:rooted-tree}.
It is {\it not\/} sufficient to just mark a node
of a tree. The node incident to the root corner is called the \emm root node,. The edge that follows the root corner, in counterclockwise order around the root node, is called the \emm root edge,. 
  The \emm degree, of a node is its graph-theoretic degree. A node of degree~$1$ is called a \emm leaf,. The degree of a corner is the degree of the incident node. All trees in this paper are plane rooted trees, and they will often be called just ``trees'' for short.

\begin{figure}[htb]
  \centering
  \includegraphics{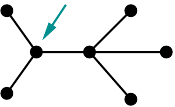}
  \caption{A rooted tree with $n=6$ edges, hence $7$ vertices and $12$ corners.}
  \label{fig:rooted-tree}
\end{figure}
 
We denote the set of all rooted plane trees with $n$~edges by~$\T(n)$. Such trees have $n+1$~nodes and $2n$~corners.
It is well-known (see e.g.\ \cite[Thm.~1.5.1(iii)]{StanBZ})
that the cardinality of~$\T(n)$ is given by the $n$-th {\it
  Catalan number}~$C_n$; more precisely, we have
\begin{equation} \label{eq:T(n)} 
|\T(n)|=C_{n} :=\frac {1} {n+1}\binom {2n}{n}.
\end{equation}

\medskip
Next we consider subsets of $\T(n)$ in which the number of leaves
is specified. We begin with
the set of all rooted plane trees with $n$~edges
and $k$~leaves, where {\it the root is one of the leaves}.
We denote the set of all these trees by~$\T_l(n,k)$.
(Here, the index ``l'' stands for ``leaf-rooted''.)

By deleting the root leaf and the edge emanating from it, and
moving the root corner to the other endpoint of this edge,
we obtain the set of rooted plane trees
with $n-1$~edges (thus $n$~nodes) and $k-1$~non-root leaves. 
Hence, by~\cite[Ex.~III.13]{flajolet-sedgewick},
the cardinality of
$\T_l(n,k)$ is given by a {\it Narayana number};
more precisely, we have\footnote{\label{smalltree-l}Strictly speaking, the expression on the right-hand side of~\eqref{eq:Tl(n,k)}  is undefined for $n=1$, corresponding to the one-edge tree.
  Consequently, the expression has to be interpreted as $1$ if $k=2$ and as $0$ otherwise.}
\begin{equation} \label{eq:Tl(n,k)} 
|\T_l(n,k)|=\frac 1 n \binom{n-2}{k-2}\binom{n}{k-1}=\frac {1} {n-1}\binom {n-1}{k-2}\binom {n-1}{k-1}.
\end{equation}

\medskip
A further subset of $\T(n)$ which we consider in our article
is the set of all rooted plane trees with $n$~edges
and $k$~leaves, where {\it the root is one of the non-leaves}.
We denote the set of all these trees by~$\T_i(n,k)$.
(Here, the index ``i'' stands for
``internally-rooted''.)

In order to determine the number of these trees, we recall that a rooted tree with $n$~edges (and $k$~leaves) has $2n$~corners. Exactly $k$ of
the corners are located at a leaf, and the other $2n-k$ at a non-leaf. A double rooting argument, where we count
trees rooted at a non-leaf \emm and, having a marked leaf
in two different ways,
yields
  \[
 k \,  |\T_i(n,k)|= (2n-k) \,    |\T_l(n,k)|,
  \]
  so that\footnote{Again, the expression on the right-hand side of~\eqref{eq:Ti(n,k)} is undefined for $n=1$. Given that all vertices of the one-edge tree are leaves, this expression has to be interpreted as $0$ in this case.}
  \beq \label{eq:Ti(n,k)}
    |\T_i(n,k)|= \frac{2n-k}{k} \,    |\T_l(n,k)|
    =\frac {2n-k} {n(n-1)}
\binom {n-1}{k-2}\binom {n}{k}.
  \eeq

A similar argument gives the number of rooted plane trees with $n$~edges and $k$~leaves without any
  restriction on the root degree. We denote this subset of~$\T(n)$ by $\T(n,k)$. Indeed, by marking
one of the $k$~leaves in a tree of $\T(n,k)$, and
by re-rooting the tree at this leaf, we obtain\footnote{\label{smalltree}With a reasoning that is analogous to the one in Footnote~\ref{smalltree-l},
  if $n=1$, the expression on the right-hand side of~\eqref{eq:T(n,k)}  has to be interpreted as $1$ if $k=2$ and as $0$ otherwise.}
  \beq\label{eq:T(n,k)} 
    |\T(n,k)|= \frac{2n}{k}\, |\T_l(n,k)|= \frac {2} {n-1}
\binom {n-1}{k-2}\binom {n}k.
  \eeq

\medskip
We now turn to families of trees with prescribed degree distribution. Let $\bn=(n_1, n_2, \ldots)$ be a sequence
of non-negative integers having finitely many non-zero entries. Let $\T(\bn)$ be the set of rooted trees having exactly $n_i$ nodes of degree~$i$, for $i\ge 1$. Such trees have $\sum_i n_i$ nodes, hence $n:= -1+ \sum_i n_i$ edges.
On the other hand, the edge number is also
$\sum_i in_i/2$, which means that such trees can only exist if
\beq\label{tree-cond}
\sum_{i\ge 1} (i-2)n_i =-2.
\eeq
Note also that $n_i$ necessarily vanishes for $i> n$. As above, we can rely on classical results to count leaf-rooted trees of degree distribution~$\bn$, and then count other types of trees by a double rooting argument. Indeed,
it was proved by Tutte in 1964~\cite[Eq.~(8)]{TutteTrees} that the number of leaf-rooted trees with degree distribution~$\bn$ is
\beq\label{eq:Tl(n,d)}
  |\T_l(\bn)|= n_1 \frac{(n-1)!}{\prod_i n_i!}= \frac 1 {n} \binom{n}{ n_1-1, n_2, n_3, \ldots}. 
\eeq
From this, we derive the number of trees with degree distribution~$\bn$ rooted at a non-leaf, as we did in~\eqref{eq:Ti(n,k)}.
Namely, we have
\beq\label{eq:Ti(n,d)}
  |\T_i(\bn)| = \frac{2n-n_1}{n_1} \ |\T_l(\bn)|=
  (2n-n_1) \frac{(n-1)!}{\prod_i n_i!}= \frac{2n-n_1}{(n+1)n}\binom{n+1}{n_1, n_2, \ldots}.
\eeq
Analogously, the total number of rooted trees with degree distribution $\bn$ is then
\beq\label{eq:T(n,d)} 
  |\T(\bn)| = \frac{2n}{n_1} \ |\T_l(\bn)|=
  2\,\frac{n!}{\prod_i n_i!}
  = \frac 2 {n+1} \binom{n+1}{n_1, n_2, \ldots}.
\eeq

Finally, we will also consider rooted trees in which the root node has a  prescribed degree, say~$\delta$.
We denote the set of such trees having degree distribution $\bn$ by $\T_\de(\bn)$.
In these trees, the number of corners that are incident to a node of degree~$\delta$ is $\delta n_\delta$. Hence the number of such trees is
\beq\label{eq:Tdelta(n,d)} 
  |\T_\delta(\bn)| = \frac{\delta n_\delta}{n_1}\ | \T_l(\bn)|=\delta n_\delta\, \frac{(n-1)!}{\prod_i n_i!}= \frac{\delta}{n} \binom{n}{n_1, \ldots, n_{\delta-1}, n_\delta-1, n_{\delta+1}, \ldots},
\eeq
generalising~\eqref{eq:Tl(n,d)}, which is the case $\de=1$.

\section{The cyclic sieving phenomenon}
\label{sec:siev}
Here we provide the definition of the {\it cyclic sieving phenomenon}.
There are actually three equivalent ways to define this
phenomenon (cf.\ \cite[Prop.~2.1]{ReSWAA}). However, the one which
gives the name is the following.

\begin{definition}[\sc Cyclic sieving phenomenon]
\label{def:siev}
Given a set $S$ of combinatorial
objects, an action on $S$ of a cyclic group $G$ of order~$m$
with generator~$g$, and a polynomial~$P(q)$ in~$q$
with non-negative integer coefficients, we say
that the triple $(S,G,P)$ {\it exhibits the cyclic sieving
phenomenon}, if the number of elements of~$S$ fixed by~$g^e$ equals
$P(\om^e)$ for all primitive $m$-th roots of unity $\om$ and all integers~$e$. 
\end{definition}

We shall denote the number of elements of~$S$ that are
fixed by~$g^e$ by $\Fix_{g^e}(S)$. Furthermore, we shall refer
to the polynomial $P(q)$ in Definition~\ref{def:siev} as the {\it
  cyclic sieving polynomial} of the cyclic sieving phenomenon.

A simple lemma states that, when proving a cyclic sieving phenomenon,  one can restrict $\om$ and $e$ to fewer values.
Note that there are two different letters ``$e$'' below, in two different fonts: the first one, $e$, is the number of times
one applies the generator~$g$, while the second, $\e$, is Euler's constant in the exponential.
  
  \begin{lemma}\label{lem:divisors}
  The triple $(S,G,P)$, with $G=\langle g\rangle$ of order~$m$,  exhibits the cyclic sieving
phenomenon if and only if the number of elements of $S$ fixed by~$g^e$ equals
$P(\om_0^e)$, with $\om_0:= \e^{2\pi i/m}$, for $e=0$ and for
divisors~$e$ of~$m$ with $e \in\llbracket 1, m-1\rrbracket$. Here
and subsequently, the symbol $\llbracket a,b\rrbracket$ stands for
$[a,b]\cap\mathbb Z$,
with $\mathbb Z$ denoting the set of integers.
\end{lemma}

\begin{proof}
  Our first argument is translation: an element is fixed by~$g^e$ if and only if it is fixed by~$g^{e+m}$.
However, we have $\om^{e+m}=\om^e$. So we can restrict our attention to $e\in \llbracket0,m-1\rrbracket$.

  Let us now explain why it suffices to prove the condition of Definition~\ref{def:siev} for $\om=\om_0$.  All primitive $m$-th roots of unity are Galois conjugates of one another in the extension of~$\Q$ by $\om_0$. Hence if a polynomial with rational coefficients vanishes at~$\om_0$, it vanishes at all primitive $m$-th roots of unity $\om$. Applying this to the  polynomial $x \mapsto P(x^e)- \Fix_{g^e}(S)$, for $e\in \llbracket 0, m-1\rrbracket$,
we obtain the result.

Finally, we explain why it suffices to focus on $0$ and divisors of~$m$. Let
$e \in \llbracket 0, m-1\rrbracket$,
and let $d=\gcd(e,m)$. There exist integers~$j$ and $k$ such that $d=je+km$. If an element is fixed by~$g^e$, it is also fixed by~$g^{-e}$, and of course by $g^m=g^{-m}=\id$. So it is fixed by~$g^d$. Conversely, if it is fixed by~$g^d$ it is fixed by~$g^e$, since $e$ is a multiple of~$d$.
(The fact that the same elements are fixed by $g^e$ and $g^d$ appears already in~\cite[Lem.~2.7]{Alexandersson}).
Assume that $P(\e^{2\pi i d/m})= \Fix_{g^d}(S)$. We have  to prove that $P(\e^{2\pi i e/m})= P(\e^{2\pi i d/m})$. Write $e=de'$ and $m=dm'$. Now~$e'$ and $m'$ are coprime, so that $\e^{\frac{2\pi i}{m'} e'}$ and $\e^{\frac{2\pi i}{m'}}$ are both primitive $m'$-th roots of unity. Thus, by the Galois argument used in the second part of the proof,
we have
  \[
    P(\e^{\frac{2\pi i}{m} e})= P(\e^{\frac{2\pi i}{m'} e'})= P(\e^{\frac{2\pi i}{m'}})=
    P(\e^{\frac{2\pi i}{m} d}).
  \]
  This concludes the proof.
    \end{proof}

    \begin{example}\label{ex:sieving-R}
Anticipating slightly
the next section, let $S$ be the set of rooted plane trees with $3$ edges, on which a rotation~$R$ of order~$m=6$ acts by moving the root corner around the tree, as illustrated in Figure~\ref{fig:example}. Then, for $e=0$ or $e$ a divisor of~$6$, we find
    \[
      \Fix_{R^e}(S)   =
      \begin{cases}
        5, & \text{if } e=0, \\
        0, & \text{if } e=1, \\
        2, & \text{if } e=2, \\
        3, & \text{if } e=3.
      \end{cases}
    \]
    We observe that a cyclic sieving phenomenon holds with
    \[
      P(q)= q^{6}+q^{4}+q^{3}+q^{2}+1.
    \]
    Indeed, for $\om=\e^{\pi i/3}$ we check that
    \[
      P(1)=5, \quad P(\om)=0, \quad P(\om^2)=2 \quad \text{and} \quad P(\om^3)=P(-1)=3.
    \]
        This will be generalised to any tree size in Theorem~\ref{thm:ord}. \qee
    \end{example}

 \begin{figure}[htb]
    \centering
    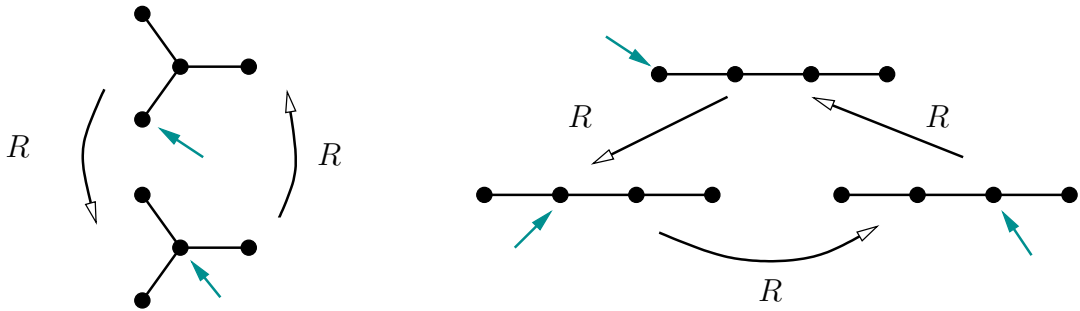
    \caption{Rotation on rooted trees with $3$ edges.}
    \label{fig:example}
  \end{figure}

 Let $\Phi_d(q)$ denote the $d$-th cyclotomic polynomial in $q$.
 Observe that in the above example, $P(q)$ is the product of the cyclotomic polynomials $\Phi_5(q)$ and $\Phi_6(q)$, so that it can be expressed rationally with factors $(1-q^i)$ only. More generally, all our cyclic sieving polynomials will be expressed in terms of the classical $q$-analogues of integers, factorials and binomial coefficients, defined as follows:
\begin{align}
[M]_q&:=\frac {1-q^M} {1-q}, \nonumber\\
[M]_q!&:=[1]_q\,[2]_q\cdots[M]_q\ \text{ and } \ [0]_q!:=1, \label{q-analogues}\\
\begin{bmatrix} M\\N\end{bmatrix}_q&:=
  \frac {[M]_q!} {[N]_q!\,[M-N]_q!}, \ \text{ for } 0\le N\le M.  \nonumber
\end{align}
By construction,  the above quantities equal $M$, $M!$, and  $\binom{M}{N}$, respectively, when $q\to1$.

\medskip
We finish this section with two simple tools that will be used repeatedly in our proofs. The first one helps evaluating our polynomials at  roots of unity, while the second is used in non-negativity proofs.

\medskip
Let $d\ge 2$ and let $\rho$ be a primitive $d$-th root of unity.
    Let $\alpha$, $\beta$ be two  positive integers. Then:
\begin{itemize}
\item $1-q^\alpha$ vanishes at $\rho$ if and only if $\alpha$ is a multiple of~$d$; that is,
\beq\label{qint0}
\lim_{q\to \rho} \  (1-q^\alpha) = 0 \quad \text{if and only if}\quad
\alpha \equiv0\ (\text{mod }d),
        \eeq
         and in this case $\rho$ is a simple root of this polynomial,
\item if $\alpha\equiv\beta\pmod d$, then 
\begin{equation} \label{eq:SM}
  \lim_{q\to \rho} \frac {[\alpha]_q}{[\beta]_q}=
  \lim_{q\to \rho} \frac {1-q^\al} {1-q^\be}=
\begin{cases}
\frac \alpha \beta,&\text{if }
\alpha\equiv\beta\equiv0\ (\text{mod } d),\\
1,&\text{otherwise}.
\end{cases}	
\end{equation}
\end{itemize}
We leave the proof of these two facts to the reader;
see also~\cite[Lem.~2.4]{SagaCS}.
The first point implies in particular that $\rho$ is a root of multiplicity $\lfloor \alpha/d\rfloor$ of $[\alpha]_q!$.

\medskip
Our second tool is a slight generalisation of~\cite[Prop.~10.1(iii)]{ReSWAA}; see also~\cite[Thm.~2]{AndrCB}. Recall that a polynomial
$P(q)=\sum _{i=0} ^{m}p_iq^i$ of degree~$m$ is called {\it reciprocal\/} if
$P(q)=q^mP(1/q)$. It is called  {\it unimodal\/} if there is an
integer $r$ with $0\le r\le m$ and $0\le p_0\le\dots\le p_r\ge\dots\ge
p_m\ge0$.  It is \emm non-negative, if its coefficients are non-negative.
\begin{lemma}\label{lem:unimodal} 
Let $n \ge 1$.  Let $P_0(q)$ and $P(q)$ be two non-negative
reciprocal polynomials, with $P_0(q)$ of degree less than $n$. Assume moreover that $P(q)$ is unimodal.
  If
  \[
    Q(q):=     \frac{P_0(q)P(q)}{[n]}
  \]
  is a polynomial, then it is reciprocal and  non-negative.
\end{lemma}
\begin{proof}
  It is clear that $Q(q)$ is reciprocal, so let us focus on non-negativity. Let $m$ be the degree of $P(q)=\sum _{i=0} ^{m}p_iq^i$.  By assumption on $P(q)$, the sequence $0, p_0, p_1, \ldots, p_{\cl{m/2}}$ is non-decreasing. In other terms, the coefficient of $q^i$ in $(1-q)P(q)$ is non-negative for $0\le i\le \cl{m/2}$. Hence the same holds for the first $\cl{m/2}+1$ coefficients
(including the constant coefficient!) of
  \[
    Q(q)= \frac{P_0(q)}{1-q^n} (1-q) P(q),
  \]
  because  the \emm  series, $\frac{P_0(q)}{1-q^n}$ has non-negative coefficients. By assumption, $Q(q)$ is a polynomial. Given that $P_0(q)$ has degree less than $n$, the degree of $Q(q)$ is at most $m$. Since it is reciprocal, and the coefficient of $q^i$ is non-negative for  $i\le \cl{m/2}$, all other coefficients are non-negative as well.
\end{proof}

  \section{Tree rotations}
  \label{sec:rotations}

In this paper, we consider three different
rotations (plus a generalisation of one of them) that act on rooted trees by changing  the root corner. We define them in this section, and characterise trees that are fixed by a given power of one of these rotations.

Our main rotation, denoted $R$ and called \emm ordinary rotation, in what follows, acts on rooted trees by moving the root corner to the next corner in counterclockwise order around the tree. This is the rotation illustrated in Figure~\ref{fig:example}. Since a rooted tree with $n$~edges has $2n$~corners, we obtain
an action of the cyclic  group of order~$2n$ on~$\T(n)$.

In the somewhat degenerate case where $n=1$, there is a unique rooted tree in~$\T(n)$, which is of course fixed by all powers of the rotation~$R$. We could then consider that the natural cyclic group acting on $\T(n)$ has order~$1$ only. Nevertheless, even when $n=1$, we study the action of the cyclic group of order~$2n$ on~$\T(n)$.

The \emm external rotation, $R_l$ acts on trees rooted at a leaf by moving the root corner to the next leaf, still in counterclockwise order. This yields an action of the cyclic group of order~$k$ on trees with $k$~leaves.

The \emm internal rotation, $R_i$ acts on trees rooted at a non-leaf by moving the root corner to the next corner incident to a non-leaf, in counterclockwise order. This yields an action of the cyclic group of order~$2n-k$  on trees with $n$~edges and $k$~leaves.

We will also consider a generalisation of the external rotation, for any integer $\delta\ge 1$. The \emm $\delta$-rotation,~$R_\delta$ acts on trees rooted at a vertex of degree~$\delta$ by moving the root to the next corner
of degree~$\delta$. This yields an action of the cyclic group of order~$\delta n_\delta$ on  trees with $n_\delta$~nodes of degree~$\delta$.

\medskip
Recall from Lemma~\ref{lem:divisors} that, in order to establish a cyclic sieving phenomenon for the triple $(S, G, P)$, with $G=\langle g \rangle$, we only have to count elements of~$S$ fixed by~$g^e$ for $e=0$ and for $e$ a divisor of~$|G|$. 
We now want to characterise trees of $\T(n)$ invariant by $R^e$, for $e$ a divisor of~$2n$. Let us first recall the definition of the
\emm centre,
of a tree~$\tau$: if $\tau$ is reduced to a node or to an edge then the centre is $\tau$
itself.
Otherwise, it is
  determined recursively  by deleting all leaves of~$\tau$, and
by iterating this pruning until a tree with one or two nodes is
  obtained. For instance, pruning the tree of Figure~\ref{fig:rooted-tree} shows in one step that its centre is an edge.

  We will also need some transformations $\Phi$ and $\Psi_d$, for $d\ge 2$, that extract a piece of (some) rooted trees~$\tau$. 
  If the centre of~$\tau$ is an edge, we define  $\Phi(\tau)$ as the tree obtained by cutting the central edge $c$ in its middle, retaining the piece that contains the root corner, and adding a marked leaf at the end of the half-edge coming from $c$ (Figure~\ref{fig:Phidef}). Note that $\Phi$ is only defined on edge-centred trees. For $d\ge 2$, and $\tau$ a rooted tree whose centre is a node~$c$ of degree~$\ell$ divisible by~$d$, we define $\Psi_d(\tau)$ as follows: let
  $\tau_1, \ldots, \tau_\ell$ be the trees incident to~$c$,  in counterclockwise order, with $\tau_1$ containing the root \emm edge,.
Then $\Psi_d(\tau)$ is obtained from $\tau$ by only retaining the trees $\tau_{1}, \ldots, \tau_{\ell/d}$ and marking the node $c$ (Figure~\ref{fig:Psidef}).
 
\begin{figure}[htb]
    \centering
    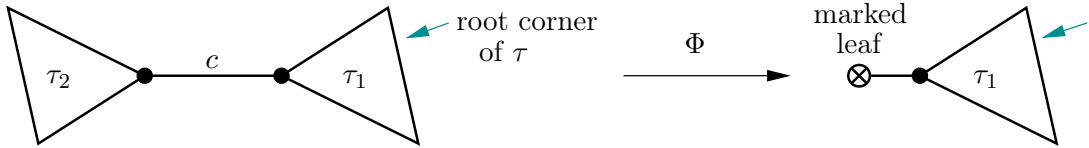
    \caption{The transformation $\Phi$, acting on an edge-centred tree $\tau$.}
    \label{fig:Phidef}
  \end{figure}

\begin{figure}[htb]
    \centering
      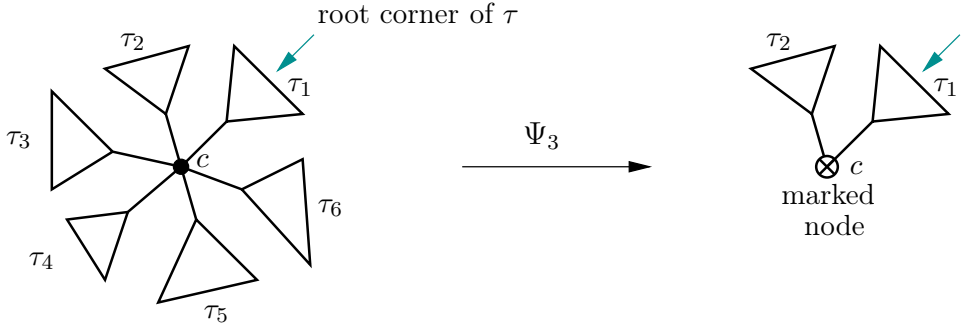
    \caption{The transformation $\Psi_d$, with $d=3$, acting on a node-centred tree $\tau$ of central degree $\ell=6$.}
    \label{fig:Psidef}
  \end{figure}

\begin{proposition}\label{prop:structure}
Let $n\ge 1$, and let
$e$ be a positive divisor of~$2n$ with $e<2n$.
  Write $d=2n/e$. Note that $e\le n$, so that $d\ge 2$.

 If $e$ is odd and $e=n$ {\em(}so that $n$ is odd and $d=2${\em)}, all trees of $\T(n)$ fixed by~$R^e$ have a central edge, and the above transformation $\Phi$ induces  a bijection between these trees and    trees of $\T\!\left(\frac {n+1} 2\right)$ having a marked non-root leaf.

    If $e$ is odd and $e<n$, then no tree of $\T(n)$ is fixed by~$R^e$.

    If $e$ is even {\em(}equivalently, $d\mid n${\em)}, then  all trees of~$\T(n)$ fixed by~$R^e$ have a central
node, the degree of which is divisible by~$d$,
     and the above transformation $\Psi_d$ induces a bijection between these trees
   and trees  of $\T\!\left( \frac e 2\right)= \T\!\left(\frac n d\right)$ having a marked node.
 \end{proposition}
 
  \begin{proof}
    We first observe that $\tau \in \T(n)$ is fixed by~$R^e$ if and only if any tree of the form $R^a(\tau)$ is fixed by~$R^e$.

\medskip
Let  $\tau \in \T(n)$ be an edge-centred tree, as on the left of Figure~\ref{fig:Phidef}, and assume that it is fixed by~$R^e$ for
a positive divisor~$e$ of~$2n$.
We use the notation $\tau_1$, $\tau_2$ of this figure. Then the tree $\tau'$ obtained by re-rooting $\tau$ at the corner that follows $\tau_1$ (in counterclockwise order) is also fixed by~$R^e$. Since $e<2n$, it must be that $R^e$ moves the root corner of~$\tau'$ to the corner that follows $\tau_2$. This implies in turn that~$\tau_1$ and $\tau_2$ are  copies of the same tree, that $n$ is odd, and that $e=n$, so that $R^e$ performs a half-turn rotation. The tree~$\Phi(\tau)$ then has  $\frac {n+1} 2$ edges. Conversely, any tree $\tau^\times$ of~$\T\!\left(\frac {n+1} 2\right)$ with a marked leaf allows us to reconstruct an edge-centred tree $\tau$ with $n$ edges that is fixed by~$R^{n}$: we glue two copies of~$\tau^\times$ at their marked leaf, and then erase this leaf and one of the root corners.
    
\medskip
Now let  $\tau \in \T(n)$ be a node-centred tree, with a centre $c$ of degree~$\ell$, as  on the left of  Figure~\ref{fig:Psidef}, and assume that it is fixed by~$R^e$. We use the notation $\tau_1, \ldots, \tau_\ell$ of this figure. Then the tree $\tau'$ obtained by re-rooting $\tau$ at the corner that is incident to $c$ and precedes $\tau_1$ is also fixed by~$R^e$. Hence~$R^e$ moves the root corner of~$\tau'$ to another  corner incident to~$c$, and in particular~$e$ is even. Say that~$R^e$ moves the root corner of~$\tau'$ to the corner that follows $\tau_{i}$. 
Clearly, the $(2n/e)$-fold application of the transformation~$R^e$
to~$\tau'$ 
    sends the root back to its original position. Writing $d=2n/e$ as in the statement of the  proposition, we see that this implies that $d\cdot i=\ell$ and that the trees $\tau_a$ and $\tau_{a+i}$ are copies of one another, for all $a$ (taken modulo~$\ell$).  Note that the tree $\Psi_d(\tau)$ then consists of $\tau_1, \ldots, \tau_i$ and has $n/d=e/2$ edges (and~$c$ as a marked node, by definition of~$\Psi_d$).
    Conversely, starting from a tree $\tau^\times$ in $\T\!\left(\frac e 2\right)$ with a marked node~$c$,
    we reconstruct a tree $\tau$ fixed by~$R^e$ as follows: we label
the subtrees of~$\tau'$ attached to $c$ by $\tau_1, \ldots, \tau_i$,
with~$\tau_1$ containing the root edge. We then attach
$d$~copies of the $i$-tuple $\tau_1, \ldots, \tau_i$ to a central node,
in this order. At the end, we only retain one of the root corners.

\medskip
To summarise:
for a positive divisor~$e$ of~$n$,
an edge-centred tree with $n$ edges can only be invariant by $R^e$ if $e=n$ and $n$ is odd. A node-centred tree with $n$ edges can  only be invariant by $R^e$ if $e$ is even.
The proposition now follows by putting all the above observations together.
      \end{proof}

The following proposition relates trees fixed by a power of $R_l$, $R_i$ or $R_\delta$ to trees fixed by a power of~$R$.

\begin{proposition}\label{prop:RRR}
Let $\tau$ be a rooted tree with $n$ edges and $k$ leaves, rooted at a leaf,
and let $e$ be a positive divisor of\/~$k$ with $e<k$.
Write $d:=k/e$. Then
  \[
R_l^e(\tau)= \tau \quad \text{if and only if}\quad  R^f(\tau)=\tau,
  \]
with $f=2n/d$ being an integer.
In particular, such trees can only exist
if $e=k/2$ with $k$ and $n+1$ even, or if $k/e$ divides~$n$.

\medskip
Let $\tau$ be a rooted tree with $n$ edges and $k$ leaves, rooted at a non-leaf,
and let $e$ be a positive divisor of $2n-k$ with $e<2n-k$.
Write $d=(2n-k)/e$. Then
  \[
R_i^e(\tau)= \tau \quad \text{if and only if}\quad  R^f(\tau)=\tau,
  \]
with $f=2n/d$ being an integer. In particular, such trees can only exist
if $e=(2n-k)/2$ with $k$ and $n+1$ even, or if $(2n-k)/e$  divides~$n$.

\medskip
Let $\delta\ge 1$. Let $\tau$ be a rooted tree with $n$ nodes and $n_\delta$ nodes of degree~$\delta$, rooted at a node of degree~$\delta$. Let $e$ be a positive divisor of $\de n_\de$ with $e< \de n_\de$.
Write $d=\delta n_\delta/e$. Then
   \[
R_\delta^e(\tau)= \tau \quad \text{if and only if}\quad  R^f(\tau)=\tau,
  \]
  with $f=2n/d$ being an integer. In particular, such trees can only exist
if $e=\delta n_\delta/2$ with $\delta n_\delta$ and $n+1$ even, or if $\delta n_\delta/e$  divides~$n$.
\end{proposition}

\begin{proof}
We only prove  the first statement, since the other proofs are similar.
Let $\ell_0, \ell_1, \ldots, \ell_{k-1}$ be the $k$~leaves of~$\tau$,
starting from the root leaf $\ell_0$,  in counterclockwise order around $\tau$
(cf.\ Figure~\ref{fig:star}).
By definition, $R_l^e$ moves the root of~$\tau$ to~$\ell_e$. Given that $de=k$,
the $d$-fold application of the transformation $R_l^e$ to~$\tau$
puts the root back in its original place. Let $f$ be the length of the path that joins $\ell_0$ to $\ell_e$ when walking around~$\tau$ in counterclockwise order (we consider edges to have unit length). Then  $R^f$ moves $\ell_0$ to $\ell_e$, so that $R^f(\tau)=R_l^e(\tau)=\tau$. 
The $d$-fold repetition of the operation~$R^f$, starting from~$\tau$,
brings the root back to its original place, so $f=2n/d$.

  Conversely, a tree $\tau$ of $\T_l(n,k)$ that is fixed by~$R^{2n/d}$ is fixed by~$R_l^e$ with $e=k/d$, since both transformations move the root to the same position.

  The conditions on $e$ come from the fact that trees of $\T(n)$ fixed by~$R^f$ can only exist if $f=n$ with $n$ odd or if $f$ is even, by Proposition~\ref{prop:structure}. That is, with $d=2n/f=k/e$, if $d=2$ with $n$ odd, or if $d$ divides~$n$.
\end{proof}

\begin{figure}[htb]
    \centering
      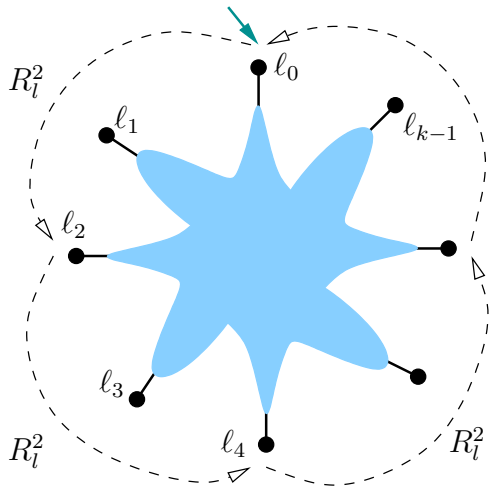
    \caption{A tree of $\T_l(n,k)$ fixed by $R_l^e$, here with $k=8$, $e=2$ and $d=4$.} 
    \label{fig:star}
   \end{figure}

\section{Ordinary tree rotation}
\label{sec:ordinary}

In this section,
we consider the ordinary rotation~$R$, which moves the root corner of a rooted tree to the next corner in counterclockwise order
(see Figure~\ref{fig:example}).
This rotation yields an action of the  cyclic group
of order~$2n$ on trees of~$\T(n)$.

We remind the reader of the definitions of $q$-integers, $q$-factorials, and
 $q$-binomial coefficients given in~\eqref{q-analogues}.

\subsection{Trees with prescribed number of nodes}

\begin{theorem} \label{thm:ord}
Let $n$ be a positive integer. The cyclic group of order~$2n$ acts on~$\T(n)$ by the rotation~$R$, and  the triple
\[\left(\T(n),\langle R\rangle,
    \frac {1}{[n+1]_q}\begin{bmatrix}
    2n\\n\end{bmatrix}_q\right)
\]
exhibits the cyclic sieving phenomenon.
\end{theorem}

\begin{example}
For $n=3$,  the above theorem gives the polynomial $P(q)= q^{6}+q^{4}+q^{3}+q^{2}+1$. This cyclic sieving phenomenon was already checked in Example~\ref{ex:sieving-R}. \qee
\end{example}

\begin{proof}[Proof of Theorem~\ref{thm:ord}]
The fact that $\frac {1}
{[n+1]_q}\left[\begin{smallmatrix} 2n\\n\end{smallmatrix}\right]_q$
is a polynomial in~$q$ with non-negative coefficients goes back to
MacMahon, who showed that it is the generating function for --- what
he calls --- \emm lattice permutations,, objects that are in trivial
bijection with Dyck paths, and which are enumerated with respect to
--- what MacMahon calls --- the \emm greater index,, which was later renamed
the ``major index'' by Foata. We refer the interested reader to
\cite[p.~256]{FuHoAA}.

\medskip
Now let $\om$ be a primitive $2n$-th root of unity.
We have to prove 
\begin{equation} \label{eq:2}
\Fix_{R^e}(\T(n))=
\frac {1}
{[n+1]_q}\begin{bmatrix} 2n\\n\end{bmatrix}_q
\Bigg\vert_{q=\om^ e}
\end{equation}
for all integers $e$.
By Lemma~\ref{lem:divisors},
it suffices to prove this for values of $e\in \llbracket 0, 2n-1\rrbracket$ equal to $0$ or dividing $2n$. Lemma~\ref{lem:1} below  gives the corresponding values of the right-hand side of~\eqref{eq:2}.

\medskip
We begin with $e=0$.  All trees in $\T(n)$ are invariant under  $R^0=\text{id}$, and we have $|\T(n)|=\frac {1} {n+1}\binom
{2n}{n}$  by~\eqref{eq:T(n)}. As observed above when defining $q$-integers, $q$-factorials and $q$-binomials, this is indeed the value of the right-hand side of~\eqref{eq:2}  at $q=1$.

\medskip
  We now assume that $0<e<2n$, with $e$ a divisor of~$2n$. We rely on Proposition~\ref{prop:structure}, which describes the structure of trees fixed by~$R^e$.

\medskip
  If $e$ is odd and equals $n$, there are as many trees in $\T(n)$ fixed by~$R^e$ as trees of $\T\!\left(\frac {n+1} 2\right)$ with a marked non-root leaf. The latter trees can be constructed as follows: we start with a tree of $\T\!\left(\frac {n-1} 2\right)$, insert  at the root corner a new edge ending with a  marked leaf (the resulting tree has $\frac {n+1} 2$ edges and $n+1$~corners), forget this root corner, and create a new root corner in one of the $n$~corners of the tree that are not incident to the marked leaf. According to~\eqref{eq:T(n)},
  this gives
   \[
     \Fix_{R^e}(\T(n))= n \times\left|\T\left(\frac{n-1}2\right)\right|=
     n \times \frac {1} {\frac{n+1}2}\binom {2\frac {n-1} {2}}{\frac {n-1} {2}}.
  \]
 This is equivalent to the expression in the second alternative
  of~\eqref{eq:1}, and, hence, proves~\eqref{eq:2} in this case.

  For odd values of $e$ distinct from $n$, no tree of $\T(n)$ is fixed by~$R^e$. This proves~\eqref{eq:2}
by the last alternative of~\eqref{eq:1} in that case.

Finally,
if $e$ is even and positive,
Proposition~\ref{prop:structure} tells us that there are as many trees in~$\T(n)$ fixed by~$R^e$ as trees of~$\T(e/2)$ having a marked node. According to~\eqref{eq:T(n)}, with~$n$ replaced by~$e/2$, we conclude that
  \[
    \Fix_{R^e}(\T(n))= \left( 1+ \tfrac e 2\right)  \times \frac {1} {1+\tfrac e 2}\binom {2\frac {e} {2}}{\frac {e} {2}} =\binom{e}{\frac e 2}.
  \]
 This coincides with the expression in the third alternative
  of~\eqref{eq:1}, and, hence, proves \eqref{eq:2} also in this case.

  The proof of the theorem is now complete (assuming Lemma~\ref{lem:1} below).
\end{proof}

\begin{lemma} \label{lem:1}
Let $\om$ be a primitive $2n$-th root of unity, and let $e$ be a
non-negative integer $<2n$
equal to~$0$ or dividing~$2n$.
Then we have
\begin{equation} \label{eq:1} 
\frac {1} {[n+1]_q}\begin{bmatrix} 2n\\n\end{bmatrix}_q
\Bigg\vert_{q=\om^e}
=
\begin{cases}
  \frac {1} {n+1}\binom {2n}{n},&\text{if }e=0,
  \\[.5em]
 \binom {n}{\frac{n+1}2},
&\text{if }e=n\text{ and $n$ is odd},\\[.5em]
 \binom {e}{\frac e 2},
&\text{if $e$ is even and positive},\\
0,&\text{otherwise.}
\end{cases}
\end{equation}
\end{lemma}

\begin{remark}
If we write
$d=2n/e$, the second case can be described by ``$d=2$ with $n$ odd'', and the third one by ``$d\mid n$''. This is closer to the descriptions used later in the numerous variants of this first cyclic sieving phenomenon (see for instance Lemmas~\ref{lem:3} or~\ref{lem:9}). 
\end{remark}

\begin{proof}[Proof of Lemma~\ref{lem:1}]
The result is obvious for $e=0$, so let us assume $e>0$. We will use
the basic facts~\eqref{qint0} and~\eqref{eq:SM} to evaluate the above $q$-Catalan polynomial at (some) roots of unity.

 Let $e$ be a divisor of~$2n$ with $e\in\llbracket 1, 2n-1\rrbracket$,
so that in fact $e\in\llbracket 1, n\rrbracket$. Then $\rho:= \om^e$ is a primitive $d$-th root of unity for $d=2n/e$, with $d\ge 2$.  The value of the $q$-Catalan polynomial at $\rho$ is non-zero if and only if there is an equal number of factors of the form $[\al]_q$ with $d \mid \al$ in the numerator $[2n]_q!$ and in the denominator $[n]_q![n+1]_q!$; that is, if and only if
  \[
    \left\lfloor \frac {2n}d\right\rfloor -  \left\lfloor \frac n d \right\rfloor -   \left\lfloor \frac {n+1} d \right\rfloor
    =
e - \left\lfloor \frac e 2 \right\rfloor-   \left\lfloor  \frac e 2+ \frac {1} d \right\rfloor
    =0.
  \]
   One readily checks that this holds if and only if $e$ is even or if $e$ is odd and equals $n$. This already proves the fourth case of the lemma.

 Let us now examine more closely the first two cases. If $e=n$ is odd (so that $\rh=-1$), we write the $q$-Catalan polynomial as
  \beq\label{grouped-2}
   \frac {1} {[n+1]_q}\begin{bmatrix} 2n\\n\end{bmatrix}_q
=\frac {1-q} {1-q^n} \cdot \frac{1-q^{2n}}{1-q^{n-1}}\cdot
\frac{1-q^{2n-1}}{1-q^{n-2}}  \cdots \frac{1-q^{n+2}}{1-q}.
\eeq
Now we apply~\eqref{eq:SM} to each of the above ratios and get
  \[
    \lim_{q\to \rho}    \frac {1} {[n+1]_q}\begin{bmatrix} 2n\\n\end{bmatrix}_q
    =\frac{2n}{n-1}\cdot \frac{2n-2}{n-3}  \cdots \frac{n+3}{2}.
  \]
  The second case of the lemma follows.

Finally, let $e$ be even and divide~$2n$.
Then, with $d=2n/e$, the number~$d$ divides~$n$.
We rewrite the $q$-Catalan polynomial as
  \beq\label{grouped-d}
    \frac {1} {[n+1]_q}\begin{bmatrix} 2n\\n\end{bmatrix}_q=
\frac{1-q^{2n}}{1-q^{n}}\cdot \frac{1-q^{2n-1}}{1-q^{n-1}}
\cdots \frac{1-q^{n+2}}{1-q^2}.
\eeq
     We now apply~\eqref{eq:SM} to each of the above ratios and obtain
     \begin{align*}
      \lim_{q\to \rho}    \frac {1} {[n+1]_q}\begin{bmatrix} 2n\\n\end{bmatrix}_q
       &=\frac{2n}{n} \cdot \frac{2n-d}{n-d}\cdots \frac {n+d}{d}
       \\
       &= \frac{e}{e/2} \cdot \frac{e-1}{e/2-1} \cdots \frac {e/2+1}{1},
     \end{align*}
     which proves the third case of the lemma.
   \end{proof}

     \begin{remark}
       We have chosen to detail a typical proof of the evaluation of a polynomial involving $q$-factorials at roots of unity, but the $q$-Lucas theorem often offers an interesting shortcut for such evaluations~\cite[Thm.~2.2]{sagan-congruences}. This theorem states that if $n=dn_1 + n_0$ and $k=dk_1+k_0$ with $0\le n_0, k_0<d$, then
       \[
         \qbinom{n}{k}_q = \binom{n_1}{k_1} \qbinom{n_0}{k_0}_q \mod \Phi_d(q),
         \]
         where, as before, $\Phi_d(q)$ denotes the $d$-th cyclotomic polynomial in $q$. Recall that the roots of this polynomial are the primitive $d$-th roots of unity. Completed with the fact that $[n]_q=[n_0]_q$ mod $\Phi_d$, this gives, for instance when $d=2$ and $n$ is odd:
         \begin{align*}
  \frac {1} {[n+1]_q}\begin{bmatrix} 2n\\n\end{bmatrix}_q
  \Bigg\vert_{q=-1}
  &=  \frac {1} {[2n+1]_q}\begin{bmatrix} 2n+1\\n\end{bmatrix}_q
  \Bigg\vert_{q=-1}
 \\ & =\frac 1 {[1]_q} \binom{n}{\frac{n-1}2} \begin{bmatrix} 1\\1\end{bmatrix}_q
  \Bigg\vert_{q=-1} = \binom{n}{\frac{n-1}2}.
        \end{align*}
        \end{remark}

\subsection{Trees with prescribed number of nodes and leaves}
Next we turn to the subset $\T(n,k)$ of~$\T(n)$, consisting of
all rooted plane trees with~$n$~edges and $k$~leaves.
Figure~\ref{fig:example} illustrates the cases $n=3$, $k=3$ and $n=3, k=2$.

\begin{theorem} \label{thm:ord_leaves}
Let $n$ and $k$ be positive integers. The cyclic group of order~$2n$ acts on $\T(n,k)$ by the rotation~$R$, and the triple
\[\left(\T(n,k),\langle R\rangle,\frac {1+q^{n}}
{[n-1]_q}
\begin{bmatrix} n-1\\k-2\end{bmatrix}_q
\begin{bmatrix} n\\k\end{bmatrix}_q\right)
\]
exhibits the cyclic sieving phenomenon.
The cyclic sieving polynomial must be interpreted as $1$ if $n=1$ and $k=2$;
cf.\ Footnote~\ref{smalltree}.
\end{theorem}

\begin{example}
The above theorem gives the polynomial $1+q^3$ for $\T(3,3)$ and $1+q^2+q^4$ for $\T(3,2)$. One can indeed check the cyclic sieving phenomenon with Figure~\ref{fig:example}.  \qee
  \end{example}

\begin{remark}
  We explain in Subsection~\ref{sec:matchings} how Theorem~\ref{thm:ord} (dealing with $\T(n)$) can be derived from the above
  theorem by summation over~$k$.
The argument is borrowed from~\cite{Alexandersson}, where an equivalent cyclic sieving phenomenon is described in terms of non-crossing perfect matchings.
\end{remark}

\begin{proof}[Proof of Theorem~\ref{thm:ord_leaves}]
Polynomiality and non-negativity of coefficients of the cyclic
sieving polynomial are proved in Lemma~\ref{lem:6P} below.

Let $\om$ be a primitive $2n$-th root of unity.
We have to prove 
\begin{equation} \label{eq:12}
\Fix_{R^e}(\T(n,k))=
\frac {1+q^{n}} {[n-1]_q}
\begin{bmatrix} n-1\\k-2\end{bmatrix}_q
\begin{bmatrix} n\\k\end{bmatrix}_q
\Bigg\vert_{q=\om^ e}
\end{equation}
for all integers $e$. 
By Lemma~\ref{lem:divisors},
it  suffices to prove this for values of $e\in \llbracket 0, 2n-1\rrbracket$ equal to $0$ or dividing $2n$. Lemma~\ref{lem:8} below  gives the corresponding values of the right-hand side of~\eqref{eq:12}.

\medskip
We begin with $e=0$.  All trees in $\T(n,k)$ are invariant under  $R^0=\text{id}$. 
By~\eqref{eq:T(n,k)} we know the cardinality
of~$\T(n,k)$, which coincides  with the value at $q=1$ of the right-hand side of~\eqref{eq:12}.

\medskip
We now assume that $0<e<2n$, with $e$ a divisor of $2n$. Again, we rely on Proposition~\ref{prop:structure}, which describes the structure of trees fixed by~$R^e$. We revisit the proof of Theorem~\ref{thm:ord} by keeping track of the number of leaves.

\medskip
If $e$ is odd and equals $n$, the map $\Phi$ sends
trees of~$\T(n,k)$ fixed by~$R^e$ bijectively
to trees of $\T\!\left(\frac {n+1} 2,1+ \frac k 2\right)$ with a marked non-root leaf. In particular, $k$ must be even. According to~\eqref{eq:T(n,k)}, with $n$ replaced by~$\frac {n+1} {2}$
and~$k$ replaced by~$\frac {k} {2}+1$, there are
   \[
  \left( \tfrac k 2+1\right) \frac 2 {\frac {n+1} 2 -1} \binom{\frac {n+1} 2 -1}{\frac k 2 -1}  \binom{\frac {n+1} 2 }{\frac k 2 +1} 
  \]
ways of choosing a tree of $\T\!\left(\frac {n+1} 2,1+ \frac k 2\right)$ with a marked leaf.
We must however subtract the number of those trees where the
marked leaf is the root. By~\eqref{eq:Tl(n,k)}
with $n$ replaced by~$\frac {n+1} {2}$
and~$k$ replaced by~$\frac {k}
{2}+1$, this number equals
\[
\frac {1} {\frac {n+1} {2}-1}
\binom {\frac {n+1} {2}-1}{\frac {k} {2}-1}
\binom {\frac {n+1} {2}-1}{\frac {k} {2}}.
\]
The difference of these two numbers simplifies to
\[
\frac {2\frac {n+1} {2}-1} {\frac {n+1} {2}-1}
\binom {\frac {n+1} {2}-1}{\frac {k} {2}-1}
\binom {\frac {n+1} {2}-1}{\frac {k} {2}},
\]
which is equivalent to the expression in the second alternative
of~\eqref{eq:11}, and, hence, proves \eqref{eq:12} in this case.

\medskip 
Finally, if $e$ is even with $e=2n/d$, the map $\Psi_d$  sends
trees of $\T(n,k)$ fixed by~$R^e$ bijectively 
to trees of $\T(e/2)$ having
  \begin{itemize}
  \item either $\frac kd$ leaves and a marked node that is not a leaf,
    \item or $1+\frac kd$ leaves, one of which is  marked.
  \end{itemize}
 The second case occurs for trees of $\T(n,k)$ (fixed by~$R^e$) in which the centre is a node of degree exactly~$d$. Observe that in all cases  $d$ must divide $k$. According to~\eqref{eq:T(n,k)}, there are
\[
\left(\tfrac {e} {2}+1-\tfrac {k} {d}\right)
\cdot
\frac {2} {\frac {e} {2}-1}
\binom {\frac {e} {2}-1}{\frac {k} {d}-2}
\binom {\frac {e} {2}}{\frac {k} {d}}
\]
trees of the first type, and
\[
\left(\tfrac {k} {d}+1\right)\cdot
\frac {2} {\frac {e} {2}-1}
\binom {\frac {e} {2}-1}{\frac {k} {d}-1}
\binom {\frac {e} {2}}{\frac {k} {d}+1}
\]
trees of the second type.
The sum of these two numbers is equal to the expression in
the third alternative 
of~\eqref{eq:11}, and is thus proving~\eqref{eq:12} also in this case.

\medskip
For all other triples $(n,k,e)$,  no tree of $\T(n,k)$ is fixed by~$R^e$. This proves \eqref{eq:12} in the last alternative of~\eqref{eq:11}.

\medskip
  The proof of the theorem is now complete (assuming Lemmas~\ref{lem:6P} and~\ref{lem:8} below).
\end{proof}

\begin{lemma} \label{lem:6P}
Let $n$ and $k$ be positive integers. Then the expression
\begin{equation} \label{eq:N(n,k)3} 
\frac {1+q^{n}} {[n-1]_q}
\begin{bmatrix} n-1\\k-2\end{bmatrix}_q
\begin{bmatrix} n\\k\end{bmatrix}_q
\end{equation}
is a polynomial in $q$ with non-negative coefficients.
\end{lemma}

\begin{proof}
We start by recalling the well-known fact that
\[
[n]_q= \frac{q^n-1}{q-1} =\prod _{d\ge 2,\  d\mid n} ^{}\Phi_d(q),
\]
where $\Phi_d(q)$ denotes the $d$-th cyclotomic polynomial in $q$.
Consequently,
\begin{equation} \label{eq:1+q^{n-1}3} 
\frac {1+q^{n}} {[n-1]_q}
\begin{bmatrix} n-1\\k-2\end{bmatrix}_q
\begin{bmatrix} n\\k\end{bmatrix}_q
=\frac {[2n]_q} {[n]_q\,[n-1]_q}
\begin{bmatrix} n-1\\k-2\end{bmatrix}_q
\begin{bmatrix} n\\k\end{bmatrix}_q
=\prod _{d\ge 2} \Phi_d(q)^{e_d}, 
\end{equation}
with
\begin{multline} \label{eq:e_d} 
e_d=\chi\big(d\mid 2n\big)-\chi\big(d\mid n\big)
+\fl{\frac {n-2} {d}}+\fl{\frac {n} {d}}\\
-\fl{\frac {k-2} {d}}-\fl{\frac {k} {d}}
-\fl{\frac {n-k+1} {d}}-\fl{\frac {n-k} {d}},
\end{multline}
where $\chi(\mathcal A)=1$ if $\mathcal A$ is true and $\chi(\mathcal
A)=0$ otherwise.
Note that $e_d$ is the multiplicity of a primitive $d$-th root of unity as a root of~\eqref{eq:1+q^{n-1}3} (where a negative multiplicity would indicate a pole). We want to prove that $e_d \ge 0$ for all $d\ge 2$.

From now on, let $d\ge2$.
Abusing notation,
let us write $N=\{(n-2)/d\}$
and $K=\{(k-2)/d\}$, where $\{\al\}:=\al-\fl{\al}$ denotes
the fractional part of~$\al$. Note that $dN$ and $dK$ are integers. We use $N$ and $K$ to rewrite all terms of~\eqref{eq:e_d} with the
integer parts (``floors'') of $(n-2)/d$ and $(k-2)/d$, plus $N$ and $K$. For instance, we have
  \[
    \fl{\frac {n-k+1} {d}} = \fl{\frac{n-2} d - \frac {k-2} d + \frac 1 d}
    =\fl{\frac{n-2} d} - \fl{\frac {k-2} d} + \fl{N-K+ \frac 1 d}.
    \]
  We thus may rewrite Equation~\eqref{eq:e_d} as 
  \begin{equation}\label{eq:e_d2}
e_d
=\chi\big(d\mid 2n\big)-\chi\big(d\mid n\big)
+\fl{N+\frac {2} {d}}
-\fl{K+\frac {2} {d}}
-\fl{N-K+\frac {1} {d}}-\fl{N-K}.
\end{equation}

For the case where $d=2$, this reduces to
\begin{align*} 
\notag
e_2&=1-\chi\big(2\mid n\big)
+\fl{N+1}-\fl{K+1}
-\fl{N-K+\frac {1} {2}}-\fl{N-K}\\
&=1-\chi\big(2\mid n\big)
-\fl{N-K+\frac {1} {2}}
-\fl{N-K},
\end{align*}
because $\fl{N}=\fl{K}=0$.
Note that $N$ and $K$ can only take the values $0$ and $1/2$ in
the current case, depending on whether $n$ and $k$ are even or odd. One checks that the exponent~$e_2$ is non-negative in all four cases. More precisely, $e_2=0$ if~$k$ is even, and $e_2=1$ otherwise.

Now let $d>2$. By~\eqref{eq:e_d2}, the exponent $e_d$ may be bounded
from below as follows: 
\begin{equation} \label{eq:e_d3} 
e_d\ge\fl{N+\frac {2} {d}}
-\fl{K+\frac {2} {d}}
-\fl{N-K+\frac {1} {d}}-\fl{N-K}.
\end{equation}
Given that $N-K<1$, the only terms that can be negative are
$-\fl{K+\frac {2} {d}}$ and $-\fl{N-K+\frac {1} {d}}$, and if they are
then they are equal to~$-1$. 
If $N<K$, then the last term on the right-hand side
equals~$+1$. This is already
sufficient to guarantee non-negativity of the right-hand side
of~\eqref{eq:e_d3} since under the condition $N<K$ the summand
$-\fl{N-K+\frac {1} {d}}$ cannot be negative.
Consequently $e_d$ is non-negative in this case.
We may therefore assume from now on that $N\ge K$. 

We have $-\fl{N-K+\frac {1} {d}}=-1$ if and only if $N=1-\frac {1}
{d}$ and $K=0$.
The inequality~\eqref{eq:e_d3} then simplifies to
$e_d\ge1-0-1-0=0$, which indeed implies non-negativity of~$e_d$.

On the other hand, we have $-\fl{K+\frac {2} {d}}=-1$ if and only if
$K=1-\frac {2} {d}$ or $K=1-\frac {1} {d}$. However, by our assumption
$N\ge K$, this forces $\fl{N+\frac {2} {d}}$ to equal~$+1$.
Consequently, the
inequality~\eqref{eq:e_d3} simplifies to
$e_d\ge1-1-0-0=0$, again implying non-negativity of~$e_d$.

This proves the polynomiality part of the lemma.

\medskip
We will now show that the coefficients of the polynomial
in~\eqref{eq:N(n,k)3} are non-negative.
We rewrite this polynomial as
  \[
    \frac {1+q^{n}} {[n-1]_q}\begin{bmatrix} n-1\\k-2\end{bmatrix}_q
  \begin{bmatrix} n\\k\end{bmatrix}_q  =
    \frac {1+q^{n}} {[n+1]_q}
\begin{bmatrix} n-2\\k-2\end{bmatrix}_q
\begin{bmatrix} n+1\\k\end{bmatrix}_q.
\]
 It is well-known that $q$-binomial coefficients are
reciprocal and unimodal; see \cite[Ex.~7.75.d]{StanBI}. Moreover, the product of reciprocal and
unimodal polynomials is also reciprocal and unimodal
(cf.\ \cite[Thm.~3.9]{AndrAF}). Hence the product
\[
\begin{bmatrix} n-2\\k-2\end{bmatrix}_q
\begin{bmatrix} n+1\\k\end{bmatrix}_q
\]
is a reciprocal and unimodal polynomial in~$q$.
We conclude using Lemma~\ref{lem:unimodal} with $n$ replaced by $n+1$, $P_0(q)=1+q^n$, and $P(q)$ being equal to the above product of $q$-binomial coefficients.
\end{proof}

\begin{lemma} \label{lem:8}
Let $\om$ be a primitive $2n$-th root of unity, and let $e$ be a
non-negative integer $<2n$  equal to~$0$ or dividing~$2n$.
In the latter case, write $d:=2n/e$, so that $d\ge 2$.
Then we have
\begin{multline} \label{eq:11} 
\frac {1+q^{n}} {[n-1]_q}
\begin{bmatrix} n-1\\k-2\end{bmatrix}_q
\begin{bmatrix} n\\k\end{bmatrix}_q
\Bigg\vert_{q=\om^e}\\
=
\begin{cases}
\frac {2} {n-1}\binom {n-1}{k-2}\binom {n}{k},&\text{if }e=0,\\[.5em]
\frac {n} {\frac {n+1} {2}-1}
\left(\begin{smallmatrix} {\frac {n+1} {2}-1}\\{\frac {k} {2}-1}
\end{smallmatrix}\right)
\left(\begin{smallmatrix} {\frac {n+1} {2}-1}\\{\frac {k} {2}}
\end{smallmatrix}\right),
&\text{if }
d=2 \text{ with  $n+1$ and $k$ even},\\[.5em]
2
\left(\begin{smallmatrix} {\frac {e} {2}-1}\\{\frac {k} { d}-1}
\end{smallmatrix}\right)
\left(\begin{smallmatrix} {\frac {e} {2}}\\{\frac {k} { d}}
  \end{smallmatrix}\right),
&\text{if $d\mid n$
  and }  d\mid k,\\
0,&\text{otherwise.}
\end{cases}
\end{multline}
\end{lemma}

\begin{proof}
The case $e=0$ being obvious, we assume that
$e$ is a positive divisor of~$2n$ with $e<2n$.
Then $\om^e$ is a primitive $d$-th root of unity, with $d=2n/e$. Its minimal polynomial is the cyclotomic polynomial $\Phi_d(q)$, as defined in the proof of Lemma~\ref{lem:6P} just above.

  We will first determine for which values of $d$ the specialisation on the left-hand side of~\eqref{eq:11} is non-zero. This will settle the fourth case in~\eqref{eq:11}. We return to the proof of Lemma~\ref{lem:6P}, and use the same notation $N, K, e_d$. The specialisation of~\eqref{eq:11} at a primitive $d$-th root of unity is non-zero if and only if the exponent $e_d$ is zero. In the proof of Lemma~\ref{lem:6P} we proved that $e_d\ge 0$ for all $d$, now we revisit this proof to decide when $e_d=0$.

  For  $d=2$, we already established that $e_d=0$ if and only if $k$ is even. This corresponds to the second case of~\eqref{eq:11} when $n$ is odd, and to the third case of~\eqref{eq:11} (with $d=2$, $e=n$) when $n$ is even.

For $d>2$,
a comparison of~\eqref{eq:e_d2} and~\eqref{eq:e_d3} tells us that $e_d=0$ if and only if $d \mid n$ (that is, $e$ is even, or equivalently $N=\{ (n-2)/d \}=1-\frac 2d$) and the right-hand side of~\eqref{eq:e_d3} vanishes. So let us focus on
the right-hand side of~\eqref{eq:e_d3}.
If $N<K$, that is, $K=1-\frac 1d$, then this expression is $1 - 1 -0 - (-1)=1 >0$, so $e_d>0$. If $K\le N=1-\frac 2d$, then this expression is
  \[
    1  - \fl {K+ \frac 2 d} -0 -0 .
  \]
  It equals $0$ if and only if $K=N=1-\frac 2d$, or equivalently $d \mid k$ (since $K= \{ (k-2)/d \}$). This corresponds to the third case of~\eqref{eq:11} with $d>2$.

  \medskip
  It remains to evaluate the left-hand side of~\eqref{eq:11} in the second and third case. This is an easy task based on the simple fact~\eqref{eq:SM}, as in the proof of Lemma~\ref{lem:1}. The key idea, again, is to form ratios of terms of the form $1-q^a$ where the exponents of~$q$ differ by a multiple of~$d$, as in~\eqref{grouped-2} and~\eqref{grouped-d}.
\end{proof}

\section{External tree rotation}
\label{sec:external}

In this section, we consider the set $\T_l(n,k)$
of all rooted plane trees with $n$~edges
and~$k$~leaves, where the root node is one of the leaves.
The cardinality of $\T_l(n,k)$ is given in~\eqref{eq:Tl(n,k)}. The external rotation~$R_l$ defined in Section~\ref{sec:rotations} acts on these trees by moving the root to the next leaf found in counterclockwise order around the tree. 

\begin{theorem} \label{thm:ext}
Let $n$ and $k$ be positive integers. The cyclic group of order~$k$ acts on $\T_l(n,k)$ by the rotation~$R_l$, and  the triple
\[\left(\T_l(n,k),\langle R_l\rangle,\frac {1}
{[n-1]_q}
\begin{bmatrix} n-1\\k-2\end{bmatrix}_q
\begin{bmatrix} n-1\\k-1\end{bmatrix}_q\right)
\]
exhibits the cyclic sieving phenomenon.
The cyclic sieving polynomial must be interpreted as $1$ if $n=1$ and $k=2$;
cf.\ Footnote~\ref{smalltree-l}.
\end{theorem}

\begin{example}
The sets $\T_l(3,3)$ and $\T_l(3,2)$ are singletons (see Figure~\ref{fig:example}), and the above theorem gives the polynomial $P(q)=1$ in both cases. \qee
  \end{example}

\begin{proof}[Proof of Theorem \ref{thm:ext}]
Polynomiality and non-negativity of coefficients of the cyclic
sieving polynomial are proved in Lemma~\ref{lem:2P} below.

Let $\om$ be a primitive $k$-th root of unity.
We have to prove 
\begin{equation} \label{eq:4}
\Fix_{R_l^e}(\T_l(n,k))=
\frac {1} {[n-1]_q}
\begin{bmatrix} n-1\\k-2\end{bmatrix}_q
\begin{bmatrix} n-1\\k-1\end{bmatrix}_q
\Bigg\vert_{q=\om^ e}
\end{equation}
for all integers $e$.
By Lemma~\ref{lem:divisors}, it suffices to prove this for values of $e\in \llbracket 0, k-1\rrbracket$ equal to $0$ or dividing $k$.
Lemma~\ref{lem:2} below gives the corresponding  values of the right-hand side of~\eqref{eq:4}.

\medskip
We begin with $e=0$. Clearly, all trees in $\T_l(n,k)$ are invariant under
$R_l^0=\text{id}$. By~\eqref{eq:Tl(n,k)} we know the cardinality
of~$\T_l(n,k)$, which coincides with the value at $q=1$ of the right-hand side of~\eqref{eq:4}.

We now assume that $0<e<k$, with $e$ a divisor of $k$. We rely in what follows
on Proposition~\ref{prop:RRR}, which relates the rotations~$R_l$ and~$R$, and on Proposition~\ref{prop:structure}, which characterises trees fixed by a given power of~$R$. Recall in particular that the only possible values of~$e$ are $e=k/2$ with $n+1$ and $k$  even (corresponding to $f=n$ in the notation of Proposition~\ref{prop:RRR}), and those
for which $k/e$ divides~$n$ (corresponding to even values of~$f$ in the notation  of Proposition~\ref{prop:RRR}).

\medskip
If $e=k/2$, with $k$ and $n+1$ even, Proposition~\ref{prop:RRR} tells us that we just have to  count trees of $\T_l(n,k)$ fixed by~$R^{n}$. By Proposition~\ref{prop:structure}, the map~$\Phi$ sends them bijectively to trees of $\T_l(\frac {n+1} 2, 1+\frac k 2)$ having a marked non-root leaf. 
By~\eqref{eq:Tl(n,k)} with $n$ replaced by $\frac {n+1} {2}$ and $k$
replaced by $\frac {k} {2}+1$, we thus have:
\[
\Fix_{R_l^e}(\T_l(n,k))=\tfrac {k} {2}\cdot\frac {1} {\frac {n+1} {2}-1}
\binom {\frac {n+1} {2}-1}{\frac {k} {2}-1}
\binom {\frac {n+1} {2}-1}{\frac {k} {2}}.
\]
This  is equivalent to
the expression in the second alternative
of~\eqref{eq:3}. This proves \eqref{eq:4} in the present case.

\medskip

Now, if $d:=k/e$ divides $n$, let us
write $f=2n/d$. By Proposition~\ref{prop:RRR}, we need to  count trees of $\T_l(n,k)$ fixed by~$R^f$. By Proposition~\ref{prop:structure}, the map $\Psi_d$ sends them bijectively to trees of $\T_l(\frac f 2)$ having
\begin{itemize}
\item either $\frac kd=e$ leaves and a marked node that is not a leaf,
  \item or $1+\frac kd=1+e$ leaves, one of which being marked and distinct from the root.
\end{itemize}
According to~\eqref{eq:Tl(n,k)}, there are
\[
\left(1+\tfrac f 2-e\right)
\frac {1} {\frac f 2-1}
\binom {\frac f 2-1}{e-2}
\binom {\frac f 2-1}{e-1}
\]
trees of the first type, and 
\[
e\cdot
\frac {1} {\frac f 2-1}
\binom {\frac f 2-1}{e-1}
\binom {\frac f 2-1}{e}
\]
trees of the second type. The sum of these two numbers is equal to the expression in the third alternative 
of~\eqref{eq:3}, and is thus proving~\eqref{eq:4} also in this case.

For all other triples $(n,k,e)$,  no tree of $\T_l(n,k)$ is fixed by~$R_l^e$. This proves \eqref{eq:4} in the last alternative of~\eqref{eq:3}.
\end{proof}

\begin{lemma} \label{lem:2P}
Let $n$ and $k$ be positive integers. Then the expression
\begin{equation} \label{eq:N(n,k)1} 
\frac {1} {[n-1]_q}
\begin{bmatrix} n-1\\k-2\end{bmatrix}_q
\begin{bmatrix} n-1\\k-1\end{bmatrix}_q
\end{equation}
is a polynomial in $q$ with non-negative coefficients.
\end{lemma}

\begin{proof}
This follows immediately from the fact that the
expression~\eqref{eq:N(n,k)1} has --- up to a power of~$q$ --- a
combinatorial interpretation as generating function for certain
lattice paths; see~\cite[Eq.~(4.1)]{FuHoAA} 
or \cite[Eq.~(5.8) with $t=1$, $n$ replaced by $n-1$,
  and $k$ replaced by $k-1$]{KratAF}.
\end{proof}

\begin{lemma} \label{lem:2}
Let $n\ge 1$ and $k\ge 2$.
Let $\om$ be a primitive $k$-th root of unity, and let $e$ be a
non-negative integer $<k$ being equal to~$0$ or dividing~$k$.
In the latter case, write $d:=k/e$, so that $d \ge 2$.
Then we have
\begin{multline} \label{eq:3} 
\frac {1} {[n-1]_q}
\begin{bmatrix} n-1\\k-2\end{bmatrix}_q
\begin{bmatrix} n-1\\k-1\end{bmatrix}_q
\Bigg\vert_{q=\om^e}\\
=
\begin{cases}
  \frac {1} {n-1}\binom {n-1}{k-2}\binom {n-1}{k-1},&\text{if }e=0,
  \\[.5em]
\binom{ \frac{n+1} 2 -2}{ \frac k 2 -1} \binom{ \frac{n+1} 2 -1}{ \frac k 2-1},
&\text{if
  $d=2$ with $k$ and $n+1$ even},
\\[.5em]
{\binom {\frac {n}{ d }-1}{e-1}}^2,
&\text{if }  d
\mid n,\\
0,&\text{otherwise.}
\end{cases}
\end{multline}
\end{lemma}

\begin{proof}
The case $e=0$ being obvious, we assume that
$e$ is a positive divisor of~$k$ with $e<k$.
Then $\om^e$ is a primitive $d$-th root of unity. By~\eqref{qint0}, its multiplicity, as a root of the conjectured cyclic sieving polynomial, is
\[
  f_d= \fl{ \frac{n-2}d} + \fl{\frac{n-1}d} - \fl{\frac{k-2}d} -\fl{\frac {k-1}d} - \fl{\frac {n-k+1} d} - \fl{\frac{n-k}d}.
\]
We introduce
again the fractional parts  $N= \{(n-2)/d\}$ and $K=\{(k-2)/d\}$. Since by assumption $d \mid k$, we have $K=1-\frac 2d$. Then
the above equation for~$f_d$
can be rewritten as
\begin{align*}
    f_d&= \fl{N+ \frac 1 d}
  - \fl{K+\frac 1 d}
  - \fl{N-K+\frac 1 d} - \fl{N-K}
  \\
  &= \fl{N+ \frac 1 d} -0 - \fl{N - \left( 1- \frac 3 d\right)} - \fl{N- \left( 1- \frac 2 d\right)}.
\end{align*}
The left-hand side of~\eqref{eq:3} will be non-zero if and only if $f_d=0$.
If $N<K$ (which is only possible for $d>2$), then  $f_d$ is $1$ or $2$. If  $N=1-\frac 2d$, then $f_d=0$. By definition of~$N$,  the condition $N=1-\frac 2d$ means that  $d \mid n$, so that we are in the third case of~\eqref{eq:3}. Finally, if $N=1-\frac 1d$ (equivalently,
$d\mid (n-1)$),
we find that $f_d=0$ if and only if $d=2$. This corresponds to the second case of~\eqref{eq:3}. We have now isolated all values of~$d$ for which the left-hand side of~\eqref{eq:3} is non-zero. This already proves the fourth case of~\eqref{eq:3}.

It is then  an easy task to evaluate the left-hand side of~\eqref{eq:3} in the second and third case, using the simple fact~\eqref{eq:SM}. The key idea, again, is to form ratios of terms of the form $1-q^a$ where the exponents of~$q$ differ by a multiple of~$d$, as in~\eqref{grouped-2} and~\eqref{grouped-d}.
\end{proof}

\section{Internal tree rotation}
\label{sec:internal}

In this section, we consider the set $\T_i(n,k)$
of all rooted plane trees with $n$~edges
and~$k$~leaves, where the root is not a leaf.
The cardinality of $\T_i(n,k)$ is given in~\eqref{eq:Ti(n,k)}. The internal rotation~$R_i$ defined in Section~\ref{sec:rotations} acts on these trees by moving the root to the next corner not incident to a leaf found in counterclockwise order around the tree. 

\begin{theorem} \label{thm:int}
Let $n$ and $k$ be positive integers with $n\ge 2$. The cyclic group of order~$2n-k$ acts on $\T_i(n,k)$ by the internal rotation~$R_i$, and the triple
\[
  \left(\T_i(n,k),\langle R_i\rangle,\frac {[2n-k]_q}
{[n]_q\,[n-1]_q}
\begin{bmatrix} n-1\\k-2\end{bmatrix}_q
\begin{bmatrix} n\\k\end{bmatrix}_q\right)
\]
exhibits the cyclic sieving phenomenon.
\end{theorem}

\begin{example}
The set $\T_i(3,3)$ is a singleton, while  $\T_i(3,2)$ has two elements that~$R_i$ exchanges (see Figure~\ref{fig:example}). The above theorem gives the polynomials $1$ and $1+q^2$, respectively, and one checks that they indeed satisfy the cyclic sieving phenomenon. \qee
  \end{example}

\begin{proof}[Proof of Theorem \ref{thm:int}]
Polynomiality and non-negativity of coefficients of the cyclic
sieving polynomial are proved in Lemma~\ref{lem:3P} below.

Let $\om$ be a primitive $(2n-k)$-th root of unity.
We have to prove 
\begin{equation} \label{eq:6}
\Fix_{R_i^e}(\T_i(n,k))=
\frac {[2n-k]_q} {[n]_q\,[n-1]_q}
\begin{bmatrix} n-1\\k-2\end{bmatrix}_q
\begin{bmatrix} n\\k\end{bmatrix}_q
\Bigg\vert_{q=\om^ e}
\end{equation}
for all integers $e$.
It suffices  to prove this for values of $e\in \llbracket 0, 2n-1-k\rrbracket$ equal to $0$ or dividing $2n-k$. Lemma~\ref{lem:3} gives the corresponding values of the right-hand side of~\eqref{eq:6}.

\medskip
We begin with $e=0$. Clearly, all trees in $\T_i(n,k)$ are invariant under
$R_i^0=\text{id}$. By~\eqref{eq:Ti(n,k)} we know the cardinality
of~$\T_i(n,k)$, which coincides with the value at $q=1$ of the right-hand side of~\eqref{eq:6}.

We now assume $0<e<2n-k$, with $e$ a divisor of $2n-k$. We rely in what follows on Proposition~\ref{prop:RRR}, which relates the rotations~$R_i$ and~$R$, and on Proposition~\ref{prop:structure}, which characterises trees fixed by a given power of~$R$. Recall in particular that the only possible values of~$e$ are $e=n-\frac k2$, with $n+1$ and $k$ even (corresponding to $f=n$ in the notation of Proposition~\ref{prop:RRR}) and those
for which
$d:=(2n-k)/e$ divides~$n$ (corresponding to even values of~$f$). Note that, in the latter case, $d$ divides~$k$.

If $e=n-\frac k2$, with $k$ and $n+1$ even, Proposition~\ref{prop:RRR} tells us that we have to count trees of $\T_i(n,k)$ fixed by~$R^{n}$. By Proposition~\ref{prop:structure}, the map $\Phi$ sends them bijectively to trees of $\T_i(\frac {n+1} 2, 1+\frac k 2)$ having a marked leaf. 
 By~\eqref{eq:Ti(n,k)}, the number of such trees is
  \[
    \left( \tfrac k 2 +1\right) \cdot \frac{2\,\frac {n+1} 2-(\frac k 2+1)}{(\frac {n+1} 2)(\frac {n+1} 2-1)}
    \binom{\frac {n+1} 2-1}{\frac k 2-1}\binom{\frac {n+1} 2}{\frac k 2+1}.
    \]
 This coincides with
 the expression in the second alternative
 of~\eqref{eq:5}. This proves \eqref{eq:6} in the present case.

Now, if $d:=(2n-k)/e$ divides~$n$, let us
write $f:=2n/d$. By Proposition~\ref{prop:RRR}, we need to count trees of $\T_i(n,k)$ fixed by~$R^f$. By Proposition~\ref{prop:structure}, the map $\Psi_d$ sends them bijectively to trees $\tau'$ with $f/2$ edges and $\frac kd+ \vareps$ leaves, where
\begin{itemize}
\item either $\vareps=1$, and $\tau'$ is rooted at a (marked) leaf,
\item or $\vareps=1$, one leaf is marked, and $\tau'$  is  rooted at a non-leaf,
  \item or $\vareps=0$, one non-leaf is marked, and $\tau'$ is rooted at a non-leaf.
  \end{itemize}
  The first two cases occur for trees of $\T_i(n,k)$ (fixed by~$R^f$) in which the central vertex~$c$ has degree~$d$ exactly. In the first (respectively second) case, the root node is $c$ itself (respectively is different from $c$).
  
  By~\eqref{eq:Tl(n,k)}, the number of trees of the first type is
  \[
  \frac 1 {\frac f 2-1}\binom{\frac f 2-1}{\frac k d -1} \binom{\frac f 2-1}{\frac k d }.
  \]
  By~\eqref{eq:Ti(n,k)}, the number of trees of the second type is
  \[
    \left(\tfrac k d  +1\right) \cdot \frac{f-\frac k d -1}{\frac f 2 (\frac f 2-1)} \binom{\frac f 2 -1}{\frac k d -1} \binom{\frac f 2 }{\frac k d +1},
  \]
  while the number of trees of the third type is
  \[
    \left( \tfrac f 2+1-\tfrac k d \right) \cdot \frac{f-\frac k d }{\frac f 2 (\frac f 2-1)}\binom{\frac f 2 -1}{\frac k d -2}\binom{\frac f 2 }{\frac k d }.
  \]
The sum of these three numbers is equal to
the expression in the third alternative 
of~\eqref{eq:5} and hence proves \eqref{eq:6} also in this case.

\medskip
For all other triples $(n,k,e)$, no tree of $\T_i(n,k)$ is fixed by~$R_i^e$. This proves~\eqref{eq:6} in the last alternative of~\eqref{eq:5}.
\end{proof}

\begin{lemma} \label{lem:3P}
Let $n\ge 1$ and $k\ge 2$.
 Then the expression
\begin{equation} \label{eq:N(n,k)2} 
\frac {[2n-k]_q} {[n]_q\,[n-1]_q}
\begin{bmatrix} n-1\\k-2\end{bmatrix}_q
\begin{bmatrix} n\\k\end{bmatrix}_q
\end{equation}
is a polynomial in $q$ with non-negative coefficients.
\end{lemma}

\begin{proof}
We decompose the polynomial in~\eqref{eq:N(n,k)2} as
\begin{equation} \label{eq:decomp1} 
\frac {1} {[n-1]_q}
\begin{bmatrix} n-1\\k-2\end{bmatrix}_q
\begin{bmatrix} n-1\\k-1\end{bmatrix}_q
+q^k
\frac {1+q^{n-k}} {[n-1]_q}
\begin{bmatrix} n-1\\k-2\end{bmatrix}_q
\begin{bmatrix} n-1\\k\end{bmatrix}_q.
\end{equation}
We know from Lemma~\ref{lem:2P} that the first term in this
expression is a polynomial in~$q$ with non-negative coefficients.

To see that the second term is a polynomial,
we proceed as in the proof of Lemma~\ref{lem:6P}. Denoting again by  $\Phi_d(q)$ the $d$-th cyclotomic polynomial, we write
\begin{equation} \label{eq:1+q^{n-1}} 
\frac {1+q^{n-k}} {[n-1]_q}
\begin{bmatrix} n-1\\k-2\end{bmatrix}_q
\begin{bmatrix} n-1\\k\end{bmatrix}_q
=\frac {[2n-2k]_q} {[n-k]_q\,[n-1]_q}
\begin{bmatrix} n-1\\k-2\end{bmatrix}_q
\begin{bmatrix} n-1\\k\end{bmatrix}_q
=\prod _{d\ge2} \Phi_d(q)^{g_d}, 
\end{equation}
with
\begin{multline} \label{eq:g_d} 
g_d=\chi\big(d\mid (2n-2k)\big)-\chi\big(d\mid (n-k)\big)
+\fl{\frac {n-2} {d}}+\fl{\frac {n-1} {d}}\\
-\fl{\frac {k-2} {d}}-\fl{\frac {k} {d}}
-\fl{\frac {n-k+1} {d}}-\fl{\frac {n-k-1} {d}}.
\end{multline}

From now on, let $d\ge2$.
Abusing notation,
here we introduce the fractional parts  $N=\{(n-1)/d\}$
and $K=\{(k-2)/d\}$.
Then we may rewrite Equation~\eqref{eq:g_d} as
\begin{align}
\notag
  g_d
&=\chi\big(d\mid (2n-2k)\big)-\chi\big(d\mid (n-k)\big)
+\fl{N-\frac {1} {d}}\\
&\kern3cm
-\fl{K+\frac {2} {d}}
-\fl{N-K}-\fl{N-K-\frac {2} {d}}.
\label{eq:g_d2}
\end{align}

For the case where $d=2$, this reduces to
\begin{align*}
\notag
g_2&=1-\chi\big(2\mid (n-k)\big)
+\fl{N-\frac {1} {2}}-\fl{K+1}
-\fl{N-K}-\fl{N-K-1}\\
&=1-\chi\big(2\mid (n-k)\big)
+\fl{N-\frac {1} {2}}
-2\fl{N-K}.
\end{align*}
A simple analysis of the four possible values of $(N,K)$ shows that the exponent~$g_2$ is always non-negative.

Now let $d>2$. By~\eqref{eq:g_d2}, the exponent $g_d$ may be bounded
from below as follows: 
\begin{equation} \label{eq:g_d3} 
g_d\ge\fl{N-\frac {1} {d}}
-\fl{K+\frac {2} {d}}
-\fl{N-K}-\fl{N-K-\frac {2} {d}}.
\end{equation}
Here, the only terms that can be negative are
$\fl{N-\frac {1} {d}}$ and
$-\fl{K+\frac {2} {d}}$, and if they are then they are equal to~$-1$.
If $N<K$, then the last two terms on the right-hand side
together
produce at least~$+2$, 
which is sufficient to guarantee non-negativity of the right-hand side
of~\eqref{eq:g_d3}, and consequently of~$g_d$.
We may therefore assume from now on that $N\ge K$. 

We have $\fl{N-\frac {1} {d}}=-1$ if and only if $N=0$.
Since we have assumed $N\ge K$,
this implies $K=0$, and therefore the
inequality~\eqref{eq:g_d3} simplifies to
$g_d\ge-1-0-0-(-1)=0$, which indeed implies non-negativity of~$g_d$.

On the other hand, we have $-\fl{K+\frac {2} {d}}=-1$ if and only if
$K=1-\frac {2} {d}$ or $K=1-\frac {1} {d}$. However, by our assumption
$N\ge K$, this forces $-\fl{N-K-\frac {2} {d}}$ to equal~$+1$.
Consequently, the
inequality~\eqref{eq:g_d3} simplifies to
$g_d\ge0-1-0-(-1)=0$, again implying non-negativity of~$g_d$.

This proves the polynomiality of the second part of~\eqref{eq:decomp1}.
The fact that this polynomial  has non-negative coefficients follows from Lemma~\ref{lem:unimodal}, with $n$ replaced by $n-1$, $P_0(q)= 1+q^{n-k}$,
and $P(q)=\left[\begin{smallmatrix} n-1\\k-2\end{smallmatrix}\right]_q
\left[\begin{smallmatrix} n-1\\k\end{smallmatrix}\right]_q$.

This concludes the proof of the lemma.
\end{proof}

\begin{lemma} \label{lem:3}
Let $\om$ be a primitive $(2n-k)$-th root of unity, and let $e$ be a
non-negative integer $<2n-k$ being equal to~$0$ or dividing~$2n-k$.
In the latter case, write $d:=(2n-k)/e$, so that $d\ge 2$.
Then we have
\begin{multline} \label{eq:5} 
\frac {[2n-k]_q} {[n]_q\,[n-1]_q}
\begin{bmatrix} n-1\\k-2\end{bmatrix}_q
\begin{bmatrix} n\\k\end{bmatrix}_q
\Bigg\vert_{q=\om^e}\\
=
\begin{cases}
  \frac {2n-k} {n(n-1)}\binom {n-1}{k-2}\binom {n}{k},&\text{if }e=0,
  \\[.5em]
  \frac {2n-k} {n-1}
  \left(\begin{smallmatrix} \frac {n-1} 2  \\ \frac k 2 -1 \end{smallmatrix}\right)
    \left(\begin{smallmatrix} \frac {n-1} 2  \\ \frac k 2  \end{smallmatrix}\right),
&\text{if $d=2$ with $n+1$ and $k$  even,}\\[.5em]
\frac {2n-k} {n}
\left(\begin{smallmatrix}  {\frac {n} {d}-1}\\
    {\frac {k} {d}-1}\end{smallmatrix}\right)
\left(\begin{smallmatrix}  {\frac {n} {d}}\\
    {\frac {k} {d}}\end{smallmatrix}\right),
&\text{if } d
\mid n,\\
0,&\text{otherwise.}
\end{cases}
\end{multline}
\end{lemma}

\begin{proof}
The case $e=0$ being obvious, we assume that
$e$ is a positive divisor of $2n-k$ with $e<2n-k$.
Then $\om^e$ is a primitive $d$-th root of unity. By~\eqref{qint0}, its multiplicity, as a root of the conjectured cyclic sieving polynomial, is 
\[
h_d:=1+ \fl{\frac {n-2}d} + \fl{\frac{n-1}d}- \fl{\frac{k-2}d} -\fl{\frac k d}
- \fl{\frac{n-k+1} d} -\fl{\frac{n-k}d}.
\]
Introducing the fractional parts $N:=\{(n-2)/d\}$ and $K:=\{(k-2)/d\}$, this can be rewritten as
\[
  h_d=1+ \fl{N+\dfrac 1 d} - \fl{K+\frac 2 d} -\fl{N-K+\dfrac 1 d} -\fl{N-K}.
\]
The left-hand side of~\eqref{eq:5} will be non-zero if and only if $h_d = 0$.
The condition $d\mid (2n-k)$ means that
$\Delta:=2N-K+\frac 2 d$ is an integer, and, given that $N$ and $K$ are in $[0, 1-\frac 1d]$, this integer can only be $0$, $1$ or $2$.

The simplest case is when $\Delta=2$, which forces $K=0$ and $N=1-\frac 1d$. Then
\[
  h_d=1+ 1 - \fl{\frac 2 d} -1 -0=  1 - \fl{\frac 2 d}.
\]
This vanishes only for $d=2$. Combined with $K=0$ and
$N=1-\frac 1d=1-\frac 12$,
this implies that~$k$ and $n+1$ are even. This is the second case in~\eqref{eq:5}.

Now, if $\Delta=0$ or $1$, we have, by definition of~$\Delta$,
\[
2N \le K+1 - \frac 2 d \le 2 - \frac 3 d ,
\]
so that $N \le 1-\frac 2 d$.
Hence
\begin{align*}
  h_d&= 1+ 0 - \fl{K+\frac 2 d}- \fl{\Delta -N - \frac 1 d } -\fl{\Delta-N-\frac 2 d} \\
     &=1 - \fl{K+\frac 2 d}- 2 \Delta + \cl{N+\frac 1 d} + \cl{N+\frac 2 d}\\
  &= 3 -2 \Delta -\fl{K+\frac 2 d}.
\end{align*}
Therefore, $h_d$ vanishes if and only if $\Delta=1$ and either $K=1-\frac 2d$ or $K=1-\frac 1d$. In the former case,
we have $2N=K+1-\frac 2d= 2-\frac 4d$, that is, $N=1-\frac 2d$. By definition of~$N$ and~$K$, this means that $d\mid n$ and $d\mid k$; this is the third case of~\eqref{eq:5}. In the latter case,
we have $2N=2-\frac 3d$, that is, $N=1-\frac 3{2d}$, which is impossible since~$dN$ must be an integer.

We have found all cases for which the left-hand side of~\eqref{eq:5} is non-zero, and thus proved the  fourth case of~\eqref{eq:5}. It is then an easy task to evaluate
the left-hand side of~\eqref{eq:5}
 in the second and third case, using the simple fact~\eqref{eq:SM}.
\end{proof}

\section{Trees with given vertex degrees under ordinary rotation}  
\label{sec:ord_deg}

Let $\bn=(n_1,n_2,\dots)$ be a sequence of non-negative
integers satisfying~\eqref{tree-cond}.  In particular, only finitely many of the $n_i$'s are non-zero. In this section we focus on 
the set $\T(\bn)$ of rooted plane trees with degree distribution~$\bn$.
The cardinality of $\T(\bn)$ is given by~\eqref{eq:T(n,d)}, where $n+1=\sum_i n_i$ is the total number of nodes. Note that $n_m=0$ for $m> n$.
  We consider the action on $\T(\bn)$ of the cyclic group of order~$2n$ generated by the ordinary rotation~$R$.

In the statement of the theorem below and as well in the rest of the
article, we need the {\it $q$-multinomial coefficient} defined by
\[
\begin{bmatrix} M\\N_1,N_2,\dots,N_\ell\end{bmatrix}_q
=\frac {[M]_q!} {[N_1]_q!\,[N_2]_q!\cdots[N_\ell]_q!},
\]
given that $N_1+N_2+\dots+N_\ell=M$.
Of course this specialises to the usual multinomial coefficient at $q=1$. If $(N_1, N_2, \ldots)$ is a sequence of non-negative integers having a finite sum $M$, so that only finitely many of the $N_i$'s are non-zero, we
write
\[
    \begin{bmatrix} M\\N_1,N_2,\dots\end{bmatrix}_q= 
    \frac {[M]_q!} {\prod_{i\ge 1} [N_i]_q!}.
  \]

\begin{theorem} \label{thm:ord_deg}
  Let
  $\bn=(n_1,n_2,\dots)$
be a sequence of non-negative integers satisfying~\eqref{tree-cond}. Write $n+1=\sum_i n_i$. Then the cyclic group of order~$2n$ acts by the rotation~$R$ on~$\T(\bn)$, and 
 the triple
\[\left(\T(\bn),\langle R\rangle,\frac {1+q^{n}}
{[n+1]_q}
\begin{bmatrix}
  n+1\\n_1,n_2,\dots
\end{bmatrix}_q\right)
\] 
exhibits the cyclic sieving phenomenon.
\end{theorem}

\begin{example}
For $\bn=(3,0, 1, 0, \ldots)$ and $\bn=(2, 2, 0, \ldots)$, corresponding to the trees of Figure~\ref{fig:example}, we obtain the same polynomials as those given by Theorem~\ref{thm:ord_leaves}, namely $1+q^3$ and $1+q^2+q^4$, respectively. \qee
  \end{example}

\begin{proof}[Proof of Theorem \ref{thm:ord_deg}]
Polynomiality and non-negativity of coefficients of the cyclic
sieving polynomial are proved in Lemma~\ref{lem:10} below.

Let $\om$ be a primitive $2n$-th root of unity.
We have to prove 
\begin{equation} \label{eq:14}
\Fix_{R^e}(\T(\bn))=
\frac {1+q^{n}} {[n+1]_q}
\begin{bmatrix} n+1\\n_1,n_2,\dots
\end{bmatrix}_q
\Bigg\vert_{q=\om^ e}
\end{equation}
for all integers $e$.
Again,
it suffices to prove this for values of $e\in \llbracket 0, 2n-1\rrbracket$ equal to $0$ or dividing $2n$. Lemma~\ref{lem:9} below gives the corresponding values of the right-hand side of~\eqref{eq:14}.

\medskip

For $e=0$, all trees of $\T(\bn)$ are fixed by~$R^e$, and their number
is given by~\eqref{eq:T(n,d)}. This coincides with the value at $q=1$ of the right-hand side of~\eqref{eq:14}.

\medskip
We now assume that $0<e<2n$, with $e$ a divisor of~$2n$. We revisit the proof of Theorem~\ref{thm:ord}
while keeping track of vertex degrees.

If $e$ is odd and equals $n$, the map $\Phi$ sends
trees of $\T(\bn)$ fixed by~$R^e$ bijectively 
to rooted trees 
having $\frac {n+1} {2}+1$~nodes, among which
$\frac {n_1}2+1$ leaves and $\frac {n_i}2$~nodes of degree~$i$ for $i\ge2$, and 
in which one of the {\it non-root\/} leaves is marked.  Observe that each $n_i$ must be even.
By~\eqref{eq:T(n,d)} with $n$ replaced by $\frac {n+1} {2}$, $n_1$
replaced by $\frac{n_1}{  2}+1$, and $n_i$ replaced by $\frac {n_i} {2}$ for $i\ge2$, there are
\[
\left(\tfrac {n_1} {2}+1\right)\cdot
\frac {2} {\frac {n+1} {2}+1}
\binom {\frac {n+1} {2}+1}{\frac {n_1} {2}+1,\frac {n_2} {2},\ldots
}
\]
ways for choosing such a tree  and marking one of the $\frac {n_1} {2}+1$
leaves. We must however subtract the number of those trees where the
marked leaf is the root. By~\eqref{eq:Tl(n,d)} with
$n$ replaced by~$\frac {n+1} {2}$, $n_1$ replaced
by~$\frac {n_1 } {2}+1$, and $n_i$ replaced by $\frac {n_i} {2}$ for
$i\ge2$, this number equals
\[
\frac {1} {\frac {n+1} {2}}
\binom {\frac {n+1} {2}}{\frac {n_1} {2},\frac {n_2} {2},\dots
}.
\]
The difference of these two numbers simplifies to
\[
\frac {2\frac {n+1} {2}-1} {\frac {n+1} {2}}
\binom {\frac {n+1} {2}}{\frac {n_1} {2},\frac {n_2} {2},\dots
 },
\]
which is equivalent to
the expression in the second alternative
of~\eqref{eq:13}. This proves~\eqref{eq:14} in the present case.

\medskip
Now let $e$ be even with $e=2n/d$. Recall from the proof of Proposition~\ref{prop:structure} that a tree of $\T(\bn)$ fixed by~$R^e$ has a central node of degree~$\ell$ divisible by~$d$, and is obtained by gluing~$d$ copies of the same tree around the central vertex. This implies in particular that $n_\ell\equiv 1 \ ({\rm mod } \ d)$, while each $n_i$ is divisible by $d$ for $i\neq \ell$. This observation constrains the degree distribution~$\bn$ as in the third case of Lemma~\ref{lem:9} below, and
it implies that, given~$\bn$ and~$e$, there is at most one possible value for the degree~$\ell$ of the central node.
Moreover, these conditions on the~$n_i$'s, combined with~\eqref{tree-cond}, imply that $d\mid \ell$, which is therefore redundant. Analogously, these conditions imply that $d$ divides $n=-1+\sum n_i$, meaning that $e=2n/d$ is even, which is also redundant, given the conditions on the $n_i$'s.
By Proposition~\ref{prop:structure}, the map $\Psi_d$ sends
trees of $\T(\bn)$ fixed by~$R^e$ bijectively 
to rooted trees with a total of $\frac e 2+1$ nodes, including
\begin{itemize}
\item 
$\bar n_i:= \frac {n_i} {d}$~nodes of degree~$i$ for $i\ne\{\ell/d,\ell\}$,
\item  $\bar n_\ell:=\frac {n_\ell-1} {d}$ nodes of degree~$\ell$,
  \item  $\bar n_{\ell/d}:=\frac {n_{\ell/d}} {d}+1$  nodes of degree~$\ell/d$,
 one of which  (corresponding to~$c$) is marked.
\end{itemize}
According to~\eqref{eq:T(n,d)}, the number of such trees is equal to
\[
\bar n_{\ell/d}\cdot
\frac {2} {\frac {e} {2}+1}
\binom {\frac {e} {2}+1}{\bar n_1, \bar n_2, \dots}\\
=
2 \binom {\frac {e} {2}}{\frac{ n_1}d ,\dots,
  \frac{n_{\ell-1}} d,\frac {n_\ell-1} { d},\frac{ n_{\ell+1}}d,\dots}.
\]
This number coincides with the expression in
the third alternative 
of~\eqref{eq:13}, and is thus proving~\eqref{eq:14} also in this case.

\medskip
For all other values of $\bn$ and $e$,  there are no
trees fixed by $R^e$ in~$\T (\bn)$.
This completes the proof of the theorem.
\end{proof}

\begin{lemma} \label{lem:10}
Let $\bn=(n_1,n_2,\dots)$ be a sequence of non-negative
integers satisfying~\eqref{tree-cond}.
Write $n+1:=\sum_i n_i$. Then the expression
\begin{equation} \label{eq:Pol3} 
\frac {1+q^{n}} {[n+1]_q}
\begin{bmatrix}
  n+1\\n_1,n_2,\dots
\end{bmatrix}_q
\end{equation}
is a polynomial in $q$ with non-negative coefficients.
\end{lemma}

\begin{proof}
To see that the expression in \eqref{eq:Pol3} is a polynomial, we write
\begin{equation} \label{eq:1+q^{n-1}4} 
\frac {1+q^{n}} {[n+1]_q}
\begin{bmatrix}
  n+1\\n_1,n_2,\dots
\end{bmatrix}_q
= \frac {[2n]_q} {[n]_q[n+1]_q}
\begin{bmatrix}
  n+1\\n_1,n_2,\dots
 \end{bmatrix}_q
=\prod _{d\ge2} \Phi_d(q)^{i_d},
\end{equation}
using again the notation $\Phi_d(q)$ for the $d$-th cyclotomic polynomial,
where
\begin{equation} \label{eq:i_d} 
i_d=\chi\big(d\mid 2n\big)-\chi\big(d\mid n\big)
+  \fl{\frac {n} {d}}
-\sum_{i \ge 1}
\fl{\frac {n_i} {d}}.
\end{equation}
We recall that $n=-1+\sum_i n_i$. 

From now on, let $d\ge2$. 
We write
$N_i=\{n_i/d\}$ for $i\ge1$, where $\{\al\}:=\al-\fl{\al}$ denotes
the fractional part of~$\al$, as before.
Using this notation,
we may rewrite Equation~\eqref{eq:i_d} as
\beq
  i_d
=\chi\big(d\mid 2n\big)-\chi\big(d\mid n\big)
+\fl{\sum_{i\ge1}N_i-\frac {1} {d}}.
\label{eq:i_d2}
\eeq
This is non-negative except possibly if $\sum_{i\ge1}N_i=0$, in which case the third term on the right-hand side above equals $-1$.
This last condition is
equivalent to all $N_i$'s being equal to zero, which
in its turn is equivalent to all $n_i$'s being
divisible by~$d$. It then follows from~\eqref{tree-cond}
that $d=2$.
Moreover, $n=-1+\sum n_i$ is then odd. Therefore, the first term on the right-hand
side of~\eqref{eq:i_d2} equals~1, while the second is~$0$. This proves non-negativity of~$i_d$, and establishes the polynomiality part of the lemma.

\medskip
Non-negativity of the expression~\eqref{eq:Pol3} follows from
Lemma~\ref{lem:unimodal}, applied with $n$ replaced by $n+1$,
$P_0(q)=1+q^n$, and $P(q)=\left[\begin{smallmatrix} n+1\\n_1,n_2,\dots\end{smallmatrix}\right]_q$. Indeed, $q$-multinomial coefficients are
reciprocal and unimodal, since they can be written as the product of $q$-binomial coefficients.
\end{proof}

\begin{lemma} \label{lem:9}
Let $\bn=(n_1,n_2,\dots)$ be a sequence of non-negative
integers satisfying~\eqref{tree-cond}. Write $n+1=\sum_i n_i$. Let $\om$ be a primitive $2n$-th root of unity, and let $e$ be a
non-negative integer $<2n$ being equal to~$0$ or dividing~$2n$.
In the latter case, write $d:=2n/e$, so that $d\ge 2$.
Then we have
\begin{multline} \label{eq:13} 
\frac {1+q^{n}} {[n+1]_q}
\begin{bmatrix} n+1\\n_1,n_2,\dots
\end{bmatrix}_q
\Bigg\vert_{q=\om^e}\\
=
\begin{cases}
  \frac {2} {n+1}\binom {n+1}{n_1,n_2,\dots
  },&\kern-9pt
\text{if }e=0,\\[.5em]
\frac {n} {\frac {n+1} {2}}
\left(\begin{smallmatrix}
    \frac {n+1} {2}\\
    \frac {n_1} {2},\frac {n_2} {2},\dots
  \end{smallmatrix}\right),
&\kern-9pt
\text{if
  $d=2$
  and each $n_i$ is even}, \\[.5em]
2
\left(\begin{smallmatrix}
    \frac {e} {2}\\
    \frac {n_1} {d},\dots,\frac {n_{\ell-1}} {d},
  \frac {n_\ell-1} {d},\frac {n_{\ell+1}} {d},\dots
  \end{smallmatrix}\right),
&\kern-9pt
\text{if }
d\mid (n_\ell-1)
\text{ for some }\ell, 
\\[-.5em]
&\kern-9pt
\text{while   $d\mid n_i$
for $i\ne \ell$,}\\
0,&\kern-9pt
\text{otherwise.}
\end{cases}
\end{multline}
\end{lemma}
Note that, in the third case, $d$
divides~$n$ since $n=-1+\sum n_i$, which means that $e$ is even.
Moreover, as already observed, Condition~\eqref{tree-cond} implies that $d\mid \ell$ in this case.

\begin{proof}[Proof of Lemma \ref{lem:9}]
The case $e=0$ is obvious, so we assume that
$e$ is a positive divisor of~$2n$ with $e<2n$.
Let $\om$ be a primitive $2n$-th root of unity. Then $\om^e$ is a primitive $d$-th root of unity. According to~\eqref{eq:1+q^{n-1}4},
in order to obtain a
non-zero value of the specialisation on the left-hand side
of~\eqref{eq:13}, the number $i_d$ given by~\eqref{eq:i_d} or~\eqref{eq:i_d2} must be zero. Note that, by assumption, we have $d\mid 2n$, so that
\begin{equation} \label{eq:d3}
 i_d= 1-\chi\big(d\mid n  \big)
+\fl{\sum_{i\ge1}N_i-\frac {1} {d}},
\end{equation}
where, as in the proof of Lemma~\ref{lem:10}, we denote by $N_i$ the fractional part of $n_i/d$.

As discussed earlier in the proof of Lemma~\ref{lem:10}, the third summand in~\eqref{eq:d3} is negative (and then equals $-1$) if and only if  $\sum_{i\ge1}N_i=0$. Then $d=2$, all $n_i$'s are even and  $n$ is odd, so that $i_d=0$ indeed.
This is the second case in~\eqref{eq:13}.
It is then an easy task, using the simple fact~\eqref{eq:SM}, to
determine the value on the left-hand side of~\eqref{eq:13} in this case.

If, on the other hand, we have $\sum N_i>0$, then the last summand
on the right-hand side of~\eqref{eq:d3} is non-negative.
Then $i_d=0$ if and only if
$d$ divides~$n$ and 
$\sum N_i\le 1$ (recall that each $N_i$ is of the form $r/d$, for $r\in \llbracket 0, d-1\rrbracket$).
However, the fact that  $d$ divides~$n=\sum n_i -1$ means that $\sum N_i -\frac 1 d$ is an integer: combined with $\sum N_i \le 1$, this forces $\sum N_i=\frac 1 d$. In other words,
all $n_i$'s but one are divisible by~$d$, and one of them, say $n_\ell$,
equals~1 modulo~$d$.
We are thus in the third case of~\eqref{eq:13}.
It is then not too difficult to obtain
the evaluation on the left-hand side of~\eqref{eq:13} in this case, using the simple fact~\eqref{eq:SM}. 
The key is to write
\begin{multline*}
  \frac {1+q^{n}} {[n+1]_q}
\begin{bmatrix} n+1\\n_1,n_2,\dots
\end{bmatrix}_q
=
\frac{1-q^{2n}}{1-q^n}\cdot \frac {1-q} {1-q^{n_\ell}}\\
\times\frac{(1-q^n)(1-q^{n-1}) \cdots (1-q^2)(1-q)}
   {(1-q^{n_\ell-1})(1-q^{n_\ell-2}) \cdots (1-q)
 \prod_{i\neq \ell}\big( (1-q^{n_i})(1-q^{n_i-1}) \cdots (1-q)\big)}.
\end{multline*}
Since $d\mid n$, it follows from~\eqref{eq:SM} that the first ratio
on the right-hand side tends to $2$ as $q\rightarrow \om^e$ (a primitive $d$-th root of unity). Similarly, since
$d\mid (n_\ell-1)$, by~\eqref{eq:SM}
the second ratio tends to $1$. The numerator and the denominator of the third ratio are both products of $n$ terms of the form $1-q^a$, where the exponents of~$q$,
when reduced modulo~$d$, equal $0, d-1, d-2, \ldots, 1, 0, d-1, \ldots , 1 , \ldots , 0, d-1, d-2, \ldots, 1$, in this order, both in the numerator and denominator. We can then split the third ratio as a product of $n$ elementary ratios to which~\eqref{eq:SM} applies.
\end{proof}

\section{Trees with given vertex degrees under $\de$-rotation}
\label{sec:delta}

In this section we fix $\delta\ge 1$ and a sequence $\bn=(n_1, n_2, \ldots)$ satisfying~\eqref{tree-cond}, and we consider the set $\T_\delta(\bn)$ of  rooted plane trees  whose root corner has degree~$\delta$ (assuming $n_\de>0$). The cardinality of this set is given by~\eqref{eq:Tdelta(n,d)}. The $\delta$-rotation~$R_\delta$ defined in Section~\ref{sec:rotations} acts on these trees by moving the root corner to the next corner  of degree~$\delta$, in counterclockwise order.

\begin{theorem} \label{thm:delta}
Let  $\de\ge 1$ be an integer.  Let $\bn=(n_1,n_2,\dots)$
be a sequence of non-negative integers satisfying~\eqref{tree-cond}
and $n_\de>0$.
Write $n+1=\sum_i n_i$.  The cyclic group of order~$\delta n_\delta$ acts by the rotation~$R_\de$ on~$\T_\de(\bn)$, and 
the triple
\[
  \left(\T_\de (\bn),\langle R_\de \rangle,\frac {[\de n_\de]_q}
{[n_\de]_q\,[n]_q} 
\begin{bmatrix}
  n\\n_1,\dots,n_{\de-1},n_\de-1,n_{\de+1},\dots
\end{bmatrix}_q
\right)
\]
exhibits the cyclic sieving phenomenon.
\end{theorem}

\begin{example}
For $\de=1$ and $\bn=(3,0, 1, 0, \ldots)$ or $\bn=(2, 2, 0, \ldots)$, corresponding to the leaf-rooted  trees of Figure~\ref{fig:example}, we obtain the same polynomials as those given by Theorem~\ref{thm:ext}, namely~$1$ in both cases. For $\de=3$ and $\bn=(3,0, 1, 0, \ldots)$, we obtain $1$ as well, as in Theorem~\ref{thm:int}. For $\de=2$ and $\bn=(2, 2, 0, \ldots)$, we obtain $1+q^2$, again as in Theorem~\ref{thm:int}. \qee
\end{example}

\begin{proof}[Proof of Theorem \ref{thm:delta}]
Polynomiality and non-negativity of coefficients of the cyclic
sieving polynomial are proved in Lemma~\ref{lem:11} below.

\medskip
Let $\om$ be a primitive $\de n_\de$-th root of unity.
We have to prove 
\begin{equation} \label{eq:15}
\Fix_{R_\de ^e}(\T_\de (\bn))=
\frac {[\de n_\de]_q} {[n_\de]_q\,[n]_q}
\begin{bmatrix} n\\n_1,\dots,n_{\de-1},n_\de-1,n_{\de+1},\dots
\end{bmatrix}_q
\Bigg\vert_{q=\om^ e}
\end{equation}
for all integers $e$. It suffices to prove this for values of~$e$ in $\llbracket 0, \de n_\de -1\rrbracket$ that equal $0$ or divide $\de n_\de$. Lemma~\ref{lem:12} below gives the corresponding values of the right-hand side of~\eqref{eq:15}.

\medskip
We begin with $e=0$. Clearly, all trees in $\T_\de (\bn)$
are invariant under
$R_\de ^0=\text{id}$. By~\eqref{eq:Tdelta(n,d)} we know the cardinality
of~$\T_\de (\bn)$, which agrees with the value of the right-hand side of~\eqref{eq:15} at $q=1$.

\medskip
We now assume $0<e<\de n_\de$, with $e$ a divisor of $\de n_\de$.
We rely in what follows on Proposition~\ref{prop:RRR}, which relates the rotations~$R_\de$ and~$R$. Recall in particular that the only possible values of~$e$ are $e= \de n_\de/2$, with $n$ odd (corresponding to $f=n$ in Proposition~\ref{prop:RRR}) and those for which
$\de n_\de/e$ divides~$n$ (corresponding to even values of~$f$). We also use Proposition~\ref{prop:structure}, as usual, to characterise trees fixed by a given power of~$R$.

If $e=\de n_\de/2$, with $n+1$ and $\de n_\de$ even, Proposition~\ref{prop:RRR} tells us that we just need to count trees of $\T_\de(\bn)$ fixed by~$R^{n}$. By Proposition~\ref{prop:structure}, the map $\Phi$ sends them bijectively to trees rooted at a corner of degree~$\de$, having a total of $1+ \frac {n+1} 2$ nodes, among which
\begin{itemize}
  \item $\frac {n_i} 2$ nodes of degree~$i$, for $i\ge 2$,
\item $1+\frac {n_1} 2$ leaves, one of which is marked, and is not the root node.
\end{itemize}
The latter condition indicates that the case $\delta=1$ has to be studied separately. Note that each $n_i$ must be even.

Assume first that $\de \ge 2$. By~\eqref{eq:Tdelta(n,d)} with $n$ replaced by $\frac {n+1} {2}$,
$n_1$ replaced by $\frac {n_1} {2}+1$, and~$n_i$ replaced by~$\frac {n_i}
{2}$ for $i\ge2$, the number of such trees is
\[ 
 \Fix_{{R_\de ^e}(\T_\de (\bn))}=  \left(\tfrac {n_1} {2}+1\right)
\cdot\frac {\de} {\frac {n+1} {2}}\binom {\frac {n+1} {2}}
 {\frac {n_1}2+1,\frac {n_2} {2},\dots,\frac {n_{\de-1}} {2},
   \frac {n_\de} {2}-1,\frac {n_{\de+1}} {2},\dots
 }.
\]
 If $\de=1$ we may use~\eqref{eq:Tdelta(n,d)} again, with the same replacements as above, to express the relevant number of  trees as
 \[
   \frac {n_1} 2 \cdot \frac 1 {\frac {n+1} 2}
   \binom{\frac {n+1} 2}{\frac{n_1} 2, \frac{n_2} 2, \ldots}.
   \]
We happily observe that, in both cases ($\delta\ge 2$ or $\de=1$), the expression that we obtain  is equivalent to the expression in the second alternative
of~\eqref{eq:16}. This proves \eqref{eq:15} in
the present case. 

\medskip

Now, if $d:=\de n_\de/e$ divides $n$, let us write $f=2n/d$.
Proposition~\ref{prop:RRR} tells us that we have to count trees of $\T_\de(\bn)$ fixed by~$R^f$. As in the proof of Theorem~\ref{thm:ord_deg}, these trees are centred at a node of degree~$\ell$, where $n_\ell-1$ is divisible by~$d$, while for $i\neq \ell$ the number~$n_i$ is divisible by~$d$. This constrains the degree distribution~$\bn$ as in the third case of Lemma~\ref{lem:12} below. By Proposition~\ref{prop:structure}, the map $\Psi_d$ sends
trees of $\T_\de(\bn)$ fixed by~$R^f$ bijectively 
to trees with a total of $\frac f 2+1$ nodes, among which
\begin{itemize}
\item $\bar n_i:=\frac {n_i} d$ nodes of degree~$i$ for $i\not \in \{\ell/d, \ell\}$,
\item $\bar n_\ell:=\frac{n_\ell-1}d$ nodes of degree~$\ell$,
    \item $\bar n_{\ell/d}:= \frac {n_{\ell/d}}d+1$ nodes of degree~$\ell/d$, one of which is marked.
    \end{itemize}
    Moreover, these trees are rooted
    \begin{itemize}
    \item  either at a  corner of degree~$\delta$, not incident to the marked node,
              \item  or possibly at a corner incident to the marked node if $\ell=\de$. 
    \end{itemize}
We will consider
the cases $\de=\ell/d$, $\de =\ell$ and $\de \not  \in \{\ell/d, \ell\}$ separately,
but then  happily observe that they  reduce to the same final counting formula.

    Assuming first  that $\de \not  \in \{\ell/d, \ell\}$, we derive from the next-to-last expression in~\eqref{eq:Tdelta(n,d)} that
    \beq\label{eq:de_generic}
      \Fix_{R_\de ^e}(\T_\de (\bn))=\bar n_{\ell/d}     \cdot |\T_\de(\bar n_1, \bar n_2, \ldots)|=
 \bar n_{\ell/d}     \cdot
      {\de \bar n_\de} \cdot \frac {(\frac f 2-1)!} {\bar n_1!\, \bar n_2 ! \cdots}.
    \eeq
       If $\de =\ell/d$,  we forbid  the marked node to be  the root node, even though it has degree~$\delta$. Thus the above formula is modified into
    \[
      \Fix_{R_\de ^e}(\T_\de (\bn))=
     \left( \bar n_{\de} -1\right)     \cdot |\T_\de(\bar n_1, \bar n_2, \ldots)|=
 \left( \bar n_{\de} -1\right)    \cdot
      {\de \bar n_\de} \cdot \frac {(\frac f 2-1)!} {\bar n_1!\, \bar n_2 ! \cdots}.
    \]
    Finally, if  $\de =\ell$,  we allow the marked node to be the root node, even though its degree is $\ell/d=\de/d$ rather than $\de$.  That is, to the generic contribution~\eqref{eq:de_generic} corresponding to trees in which the root has degree~$\de$,  we have to add the number of trees of $\T(\bar n_1, \bar n_2, \ldots)$ rooted at a corner of degree~$\delta/d$. It then follows  from the next-to-last expression in~\eqref{eq:Tdelta(n,d)} that
    \[
      \Fix_{R_\de ^e}(\T_\de (\bn))=
 \bar n_{\ell/d}     \cdot
 {\de \bar n_\de} \cdot \frac {(\frac f 2-1)!} {\bar n_1!\, \bar n_2 ! \cdots}
 +\frac \de d\  \bar n_{\de/ d}  \cdot \frac {(\frac f 2-1)!} {\bar n_1!\, \bar n_2 ! \cdots}.
    \]

In all three cases, the formula that we obtain coincides with
the expression in the third alternative 
of~\eqref{eq:16}. This proves \eqref{eq:15} also in this case.

\medskip
For all other values of $\bn$ and $e$,  there are no
trees fixed by $R_\de^e$
in~$\T_\de (\bn)$.
This completes the proof of the theorem.
\end{proof}

\begin{lemma} \label{lem:11}
Let $\de\ge 1$, and   let $\bn=(n_1,n_2,\dots )$
  be a sequence of non-negative integers satisfying~\eqref{tree-cond}. Write $n+1=\sum_i n_i$. Then the expression
\[
\frac {[\de n_\de]_q} {[n_\de]_q\,[n]_q}
\begin{bmatrix}
  n\\n_1,\dots,n_{\de-1},n_\de-1,n_{\de+1},\dots
\end{bmatrix}_q
\]
is a polynomial in $q$ with non-negative coefficients.
\end{lemma}

\begin{proof}
We write
\begin{equation} \label{eq:1+q^{n-1}5} 
\frac {[\de n_\de]_q} {[n_\de]_q\,[n]_q}
\begin{bmatrix} n\\n_1,\dots,n_{\de-1},n_\de-1,n_{\de+1},\dots
\end{bmatrix}_q
=\prod _{d\ge2} \Phi_d(q)^{j_d},
\end{equation}
using again the notation $\Phi_d(q)$ for the $d$-th cyclotomic polynomial in
$q$,
where
\begin{equation} \label{eq:j_d} 
j_d=\chi\big(d\mid \de n_\de\big)-\chi\big(d\mid n_\de\big)
+  \fl{\frac {n-1} {d}}-\fl{\frac {n_\de-1} {d}}
-\underset {i\ne \de}{\sum_{i \ge 1}
}\fl{\frac {n_i} {d}}.
\end{equation}

Recall that $n+1=\sum_i n_i$. Fix
$d\ge2$.
We introduce the fractional parts
$N_i=\{n_i/d\}$ for $i\ne \de$ and
$N_\de=\{(n_\de-1)/d\}$.
Using this notation,
we rewrite Equation~\eqref{eq:j_d} as
\beq
j_d
=\chi\big(d\mid \de n_\de\big)-\chi\big(d\mid n_\de\big)
+\fl{
  {\, \sum_{i\ge1}}
  N_i-\frac {1} {d}}.
\label{eq:j_d2}
\eeq
This is non-negative except possibly if $\sum_{i\ge1}
N_i=0$.
By definition of the $N_i$'s, the last condition is
equivalent to all $N_i$'s being equal to zero, which
in its turn is equivalent to all $n_i$'s being
divisible by~$d$ except for $n_\de$, which
satisfies the congruence
$n_\de\equiv1$~(mod~$d$). It then follows from~\eqref{tree-cond} that
  $d$ divides~$\delta$.
Hence the right-hand side
of~\eqref{eq:j_d2} reduces to $1-0+(-1)=0$. Thus, in all cases the conclusion
is that $j_d$ is non-negative.
This proves the polynomiality part of the lemma.

\medskip
For proving non-negativity of the coefficients of~\eqref{eq:1+q^{n-1}5}, 
we apply Lemma~\ref{lem:unimodal} with
$P(q)=\left[\begin{smallmatrix}
  n\\n_1,\dots,n_{\de-1},n_\de-1,n_{\de+1},\dots
  \end{smallmatrix}\right]_q$ and $P_0(q)=[\de n_\de]_q/
{[n_\de]_q}=1+q^{n_\de}+ \cdots+ q^{(\de-1)n_\de}$.
We recall that $q$-multinomial 
coefficients are reciprocal and unimodal polynomials.
Thus, the above choice for $P(q)$ is a reciprocal and unimodal polynomial.
Furthermore, the choice for
$P_0(q)$ is a reciprocal and non-negative polynomial. Moreover, its degree is less than $n$: this is obvious if $\de=1$, and if $\de\ge 2$, it follows from~\eqref{tree-cond} that $(\de-2)n_\de \le -2+n_1 \le -2+(n+1-n_\de)$, which gives the desired inequality.
  
\medskip
This completes the proof of the lemma.
\end{proof}

\begin{lemma} \label{lem:12}
Let $\de\ge 1$, and   let $\bn=(n_1,n_2,\dots )$
  be a sequence of non-negative integers satisfying~\eqref{tree-cond}. Write $n+1=\sum_i n_i$.
Furthermore,
let $\om$ be a primitive $\de n_\de$-th root of unity, and let $e$ be a
non-negative integer $<\de n_\de$ being equal to~$0$ or
dividing~$\de n_\de$.
In the latter case, write $d:=\de n_\de/e$, so that $d \ge 2$.
Then we have
\begin{multline} \label{eq:16} 
\frac {[\de n_\de]_q} {[n_\de]_q\,[n]_q}
\begin{bmatrix}
  n\\n_1,\dots,n_{\de-1},n_\de-1,n_{\de+1},\dots
\end{bmatrix}_q
\Bigg\vert_{q=\om^e}\\
=
\begin{cases}
\frac {\de} {n}
\left(\begin{smallmatrix}
    n\\n_1,\dots,n_{\de-1},n_\de-1,n_{\de+1},\dots
  \end{smallmatrix}\right),
&\kern-3.2pt
\text{if }e=0,\\[.5em]
\de\cdot
\left(\begin{smallmatrix}
  \frac {n+1} {2}-1\\
 \frac {n_1}2,\dots,\frac {n_{\de-1}} {2},
 \frac {n_\de} {2}-1,\frac {n_{\de+1}} {2},\dots
\end{smallmatrix}\right),
&\kern-3.2pt
\text{if }
d=2 \text{ and each $n_i$ is even},\\[.5em]
\frac {\de n_\de}{n} \cdot \left(\begin{smallmatrix}
{\frac {n} {d}}\\
{\frac {n_1} {d},\dots,
  \frac {n_{\ell-1}} {d},
    \frac {n_\ell-1} {d},\frac {n_{\ell+1}} {d},
    \dots
  }
    \end{smallmatrix}\right),
&\kern-3.2pt
\text{if } d\mid (n_\ell-1) 
\text{ for some }\ell
\\[-.5em]
&\kern-3.2pt
\text{while } d\mid n_i \text{ for $i\ne \ell$},\\[.5em]
0,&\kern-3.2pt
\text{otherwise.}
\end{cases}
\end{multline}
\end{lemma}

Note that
in the third case we have $d \mid n$, so that if $d=2$ then $n$ is even; the second and third
case are thus disjoint because $n$ is odd in the second case.

\begin{proof}[Proof of Lemma \ref{lem:12}]
The case where $e=0$ is trivial. Therefore we assume from now on that $e$
is a positive divisor of~$\de n_\de$ with $e<\de n_\de$. Then $\om^e$ is a primitive $d$-th root of unity. According to~\eqref{eq:1+q^{n-1}5}, in order to obtain a non-zero value of the specialisation on the left-hand side of~\eqref{eq:16}, the number $j_d$ given by~\eqref{eq:j_d} or~\eqref{eq:j_d2} must be zero. Note that, by
assumption, we have $d\mid \de n_\de$, so that, with the notation $N_i$ used in the proof of Lemma~\ref{lem:11},
\begin{equation} \label{eq:deldel} 
j_d=1-\chi\big(d \mid  n_\de) 
+\fl{\sum_{i\ge1}N_i-\frac {1} {d}}.
\end{equation}

  As already discussed in the proof of Lemma~\ref{lem:11}, the third summand on the right-hand side is negative if and only if $\sum N_i=0$, and in this case it equals~$-1$. We have seen that $j_d=0$ in this case, that $d$ divides~$n_\de-1$ and all $n_i$'s with $i\neq \delta$.
  We are then in the third case of~\eqref{eq:16} with $\ell=\de$. The left-hand side of~\eqref{eq:16} is then  evaluated using the simple fact~\eqref{eq:SM}.

  If, on the other hand $\sum N_i >0$, then the third summand in~\eqref{eq:deldel} is non-negative. The equation $j_d=0$ holds if and only if $d\mid n_\de$ (which means that $N_\de=1-\frac 1d$) and $\sum N_i \le 1$. Given the value of $N_\de$, this means that $\sum_{i\neq \de} N_i \le 1/d$, or equivalently $\sum_{i\neq \de} N_i \in \{0, 1/d\}$.
  \begin{itemize}
  \item   If  $\sum_{i\neq \de} N_i=0$, then each $n_i$  is divisible by $d$ (including $n_\de$, as we have seen above). Then~\eqref{tree-cond}
    forces $d=2$, and we are in the second case of~\eqref{eq:16}. The left-hand side of~\eqref{eq:16} is then  evaluated using the simple fact~\eqref{eq:SM}.
  \item  If   $\sum_{i\neq \de} N_i=1/d$, then there exists $\ell\neq \de$ such that $N_\ell=1/d$ (that is,\break
$d \mid (n_\ell-1)$),
 while $d\mid N_i$ for $i\neq \ell$ (including $i=\de$). We are then in the third case of~\eqref{eq:16}, with $\ell\neq \de$. The left-hand side of~\eqref{eq:16} is then  evaluated using the simple fact~\eqref{eq:SM}.
    \end{itemize}
\end{proof}

\section{Trees with given vertex degrees under internal rotation}
\label{sec:int_deg}

In this section we fix a sequence $\bn=(n_1,n_2,\dots)$  of non-negative
integers satisfying~\eqref{tree-cond}, and consider the set $\T_i(\bn)$ of rooted plane trees with degree distribution~$\bn$ whose root corner is incident to a non-leaf. The cardinality of this set is given by~\eqref{eq:Ti(n,d)}, where $n+1=\sum_i n_i$ is the total number of nodes.
The internal rotation~$R_i$ acts on these trees by moving the root corner to the next corner not incident to a leaf, in counterclockwise order.

\begin{theorem} \label{thm:int_deg}
  Let $\bn=(n_1,n_2,\dots)$
  be a sequence of non-negative integers satisfying~\eqref{tree-cond}. Write $n+1=\sum_i n_i$ and suppose $n\ge 2$. The cyclic group of order~$2n-n_1$ acts on~$\T_i(\bn)$ by the internal rotation~$R_i$, and 
the triple
\[
  \left(\T_i(\bn),\langle R_i\rangle,\frac {[2n-n_1]_q}
{[n+1]_q\,[n]_q}
\begin{bmatrix}
  n+1\\n_1,n_2,\dots
\end{bmatrix}_q
\right)
\]
exhibits the cyclic sieving phenomenon.
\end{theorem}

\begin{example}
For $\bn=(3,0, 1, 0, \ldots)$ and $\bn=(2, 2, 0, \ldots)$, corresponding to the trees of Figure~\ref{fig:example}, we obtain the same polynomials as those given by Theorem~\ref{thm:int}, namely $1$ and $1+q^2$, respectively. \qee
\end{example}

\begin{proof}[Proof of Theorem \ref{thm:int_deg}]
Polynomiality and non-negativity of coefficients of the cyclic
sieving polynomial are proved in Lemma~\ref{lem:6} below.

\medskip
Let $\om$ be a primitive $(2n-n_1)$-th root of unity.
We have to prove 
\begin{equation} \label{eq:9}
\Fix_{R_i^e}(\T_i(\bn))=
\frac {[2n-n_1]_q} {[n+1]_q\,[n]_q}
\begin{bmatrix} n+1\\n_1,n_2,\dots
\end{bmatrix}_q
\Bigg\vert_{q=\om^ e}
\end{equation}
for all integers  $e$.
It suffices to prove this for values $e\in \llbracket 0, 2n-1-n_1\rrbracket$ equal to $0$ or dividing $2n-n_1$.
Lemma~\ref{lem:7} below gives the corresponding values of the right-hand side of~\eqref{eq:9}.
 
\medskip
When $e=0$, all trees in $\T_i(\bn)$ are fixed by 
$R_i^0=\text{id}$, and their number is given by~\eqref{eq:Ti(n,d)}. This coincides with the value at $q=1$ of the right-hand side of~\eqref{eq:9}.

\medskip
We now assume that $0<e<2n-n_1$, with $e$ a divisor of $2n-n_1$. We revisit the proof of Theorem~\ref{thm:int}
while keeping track of all vertex degrees. In particular, we use Proposition~\ref{prop:RRR} to relate the rotations~$R_i$ and~$R$. Recall that the only possible values of~$e$
are $e=n-\frac {n_1}2$, with $n$ odd (corresponding to $f=n$ in the notation of Proposition~\ref{prop:RRR}), and those for which
$(2n-n_1)/e$ divides~$n$ (corresponding to even values of~$f$).

If $e=n-\frac {n_1}2$, with $n+1$ and $n_1$ even, Proposition~\ref{prop:RRR} tells us that we simply have to
count trees of $\T_i(\bn)$ fixed by~$R^{n}$. By Proposition~\ref{prop:structure}, the map $\Phi$ sends them bijectively to trees of $\T_i(1+ \frac{ n_1} 2, \frac {n_2} 2, \frac{n_3} 2, \ldots)$ having a marked leaf. Such trees have $(n+1)/2$ edges. Hence, using~\eqref{eq:Ti(n,d)}, we conclude that
\[
  \Fix_{R_i^e}(\T_i(\bn))=\left (1+ \tfrac{ n_1} 2\right) 
  \cdot\frac {2 \frac{n+1}2 -1-\frac{n_1}2}
  {\frac {n+1} 2 (1+\frac {n+1} 2)}
  \binom{1+\frac {n+1} 2}{1+ \frac{ n_1} 2, \frac {n_2} 2, \ldots}.
\]
This expression is equivalent to the expression in the second alternative 
of~\eqref{eq:10}. This proves \eqref{eq:9} in this case.

Now, if $d:=(2n-n_1)/e$ divides~$n$, let us
write $f:=2n/d$. By Proposition~\ref{prop:RRR}, we need to count trees of $\T_i(\bn)$ fixed by~$R^f$. As in the proof of Theorem~\ref{thm:ord_deg}, these trees are centred at a node of degree~$\ell$ (divisible by~$d$), where $n_\ell-1$ is divisible by~$d$, while for $i\neq \ell$ the number~$n_i$ is divisible by $d$. This constrains the degree distribution~$\bn$ as in the third case of Lemma~\ref{lem:7} below. By Proposition~\ref{prop:RRR}, the map $\Psi_d$ sends
trees of $\T_i(\bn)$ fixed by~$R^f$ bijectively 
to trees with a total of $\frac f 2 +1$ nodes, among which
\begin{itemize}
\item  $\bar n_i:=\frac {n_i} d$ nodes of degree~$i$ for $i\not \in \{\ell/d, \ell\}$,
\item $\bar n_\ell:=\frac{n_\ell-1}d$ nodes of degree~$\ell$,
    \item $\bar n_{\ell/d}:= \frac {n_{\ell/d}}d+1$ nodes of degree~$\ell/d$, one of which is marked.
    \end{itemize}
    Moreover, these trees are rooted
    \begin{itemize}
    \item either at a corner incident to a non-leaf,
      \item or possibly at the marked node when it is a leaf, that is, when $\ell=d$.
      \end{itemize}

      Assuming first that $\ell\neq d$, we conclude from~\eqref{eq:Ti(n,d)} that
      \[
        \Fix_{R_i^e}(\T_i(\bn))=\bar n_{\ell/d} \cdot
        \frac{f-\bar n_1}{\frac f 2 (1+\frac f 2)} \binom{1+\frac f 2}{\bar n_1, \bar n_2, \ldots}.
      \]
      If $\ell=d$, we must add to the above term the number of leaf-rooted trees having degree distribution $(\bar n_1, \bar n_2, \ldots)$. Using~\eqref{eq:Tl(n,d)}, we  conclude that  
      \begin{align*}
          \Fix_{R_i^e}(\T_i(\bn))&=
         \bar n_{1} \cdot
         \frac{f-\bar n_1}{\frac f 2 (1+\frac f 2)} \binom{1+\frac f 2}{\bar n_1, \bar n_2, \ldots}
         +\frac 1 {\frac f 2} \binom{\frac f 2}
         {\bar n_1 -1, \bar n_2, \ldots}
 \\       & = \frac {f-\bar n_1+1} {\frac f 2} \binom{\frac f 2}
         {\bar n_1 -1, \bar n_2, \ldots}.
      \end{align*}
We happily observe that, in both cases ($\ell\neq d$ or $\ell=d$), the expression that we obtain is equivalent to the expression in the third alternative of~\eqref{eq:10}. This proves \eqref{eq:9} also in this case.

\medskip
For all other values of $e$ and $\bn$, there are no trees fixed by $R_i^e$
in~$\T_i(\bn)$.
This completes the proof of the theorem.
\end{proof}

\begin{lemma} \label{lem:6}
Let  $\bn=(n_1,n_2,\dots)$ be a sequence of non-negative integers
satisfying~\eqref{tree-cond}. Write $n+1=\sum_i n_i$. Then the expression
\begin{equation} \label{eq:Pol2} 
\frac {[2n-n_1]_q} {[n+1]_q\,[n]_q}
\begin{bmatrix}
  n+1\\n_1,n_2,\dots
\end{bmatrix}_q
\end{equation}
is a polynomial in $q$ with non-negative coefficients.
\end{lemma}

\begin{proof}
Taking inspiration from the decomposition~\eqref{eq:decomp1}, we
decompose~\eqref{eq:Pol2} as
\[
\frac {1} {[n]_q}
\begin{bmatrix}
  n\\n_1-1,n_2,\dots
\end{bmatrix}_q
+q^{n_1}
\frac {1+q^{n-n_1}} {[n+1-n_1]_q}\begin{bmatrix} n-1\\n_1\end{bmatrix}_q
\begin{bmatrix}
  n+1-n_1\\n_2,n_3,\dots
\end{bmatrix}_q.
\]
By Lemma~\ref{lem:11} with $\de=1$,
we know that the first term in this
expression is a polynomial in~$q$ with non-negative coefficients.

To see that the second term is a polynomial, we write 
\begin{equation} \label{eq:1+q^{n-1}2} 
\frac {1+q^{n-n_1}} {[n+1-n_1]_q}\begin{bmatrix} n-1\\n_1\end{bmatrix}_q
\begin{bmatrix}
  n+1-n_1\\n_2,n_3,\dots
\end{bmatrix}_q
= [2n-2n_1]_q\frac {[n-1]_q!}{\prod _{i\ge 1} [n_i]_q!}
=\prod _{d\ge 2} \Phi_d(q)^{k_d},
\end{equation}
using again the notation $\Phi_d(q)$ for the $d$-th cyclotomic polynomial in
$q$,
where
\begin{equation} \label{eq:k_d} 
k_d=\chi\big(d\mid (2n-2n_1)\big)
+  \fl{\frac {n-1} {d}}
-\sum_{i\ge 1}
\fl{\frac {n_i} {d}}.
\end{equation}

From now on, let $d\ge2$. We introduce the fractional parts
$N_i=\{n_i/d\}$ for $i\ge1$.
Using this notation,
we may rewrite Equation~\eqref{eq:k_d} as
\beq
k_d
=\chi\big(d\mid (2n-2n_1)\big)
+\fl{\sum_{i\ge1}N_i-\frac {2} {d}}.
\label{eq:k_d2}
\eeq
This is non-negative except possibly if $\sum_{i\ge1}N_i=0$ or
$\sum_{i\ge1}N_i=\frac {1} {d}$.
In that case the second term on the right-hand side is $-1$.

If $\sum_{i\ge1}N_i=0$, this is
equivalent to all $N_i$'s being equal to zero, which
in its turn is equivalent to all $n_i$'s being
divisible by~$d$. It then follows from~\eqref{tree-cond} that
$d=2$, so that $k_d=0$.

If $\sum_{i\ge1}N_i=\frac {1} {d}$, this is
equivalent to all $N_i$'s being equal to zero except for one that
equals~$\frac {1} {d}$, say $N_\ell=\frac {1} {d}$ for a specific~$\ell$.
In its turn, this is equivalent to all $n_i$'s being
divisible by~$d$ except for~$n_\ell$ which satisfies
$n_\ell\equiv1$~(mod~$d$). In particular, $d$~divides $n=-1+\sum n_i$.
Hence, if $\ell\ge2$, then $d$~divides $n-n_1$, so that the first term
on the right-hand side of~\eqref{eq:k_d2} equals~$+1$ which
guarantees non-negativity of the 
right-hand side.
However, we cannot have  $\ell=1$, that is, $n_1\equiv 1$~(mod~$d$), because~\eqref{tree-cond} would then give $-1\equiv -2$~(mod~$d$), a contradiction since $d\ge 2$.

This proves the polynomiality part of the lemma. 

\medskip
Finally, the fact that the polynomial in~\eqref{eq:1+q^{n-1}2} has
non-negative coefficients follows from an application of
Lemma~\ref{lem:unimodal}, with $n$ replaced by $n+1-n_1$,
$P_0(q)=1+q^{n-n_1}$, and $P(q)=\left[\begin{smallmatrix} n-1\\n_1\end{smallmatrix}\right]_q
\left[\begin{smallmatrix}
  n+1-n_1\\n_2,n_3,\dots
\end{smallmatrix}\right]_q$. The assumptions of this lemma hold indeed because, as already mentioned, $q$-binomial and $q$-multinomial coefficients are reciprocal and unimodal, and the product of such polynomials remains reciprocal and unimodal.

\medskip
This completes the proof of the lemma.
\end{proof}

\begin{lemma} \label{lem:7}
Let  $\bn=(n_1,n_2,\dots)$ be a sequence of non-negative integers
satisfying~\eqref{tree-cond}. Write $n+1=\sum_i n_i$. Furthermore,
let $\om$ be a primitive $(2n-n_1)$-th root of unity, and let $e$ be a
non-negative integer $<2n-n_1$ being equal to~$0$ or dividing~$2n-n_1$. 
In the latter case,
write $d= (2n-n_1)/e$, so that $d \ge 2$.
Then we have
\begin{multline} \label{eq:10} 
\frac {[2n-n_1]_q} {[n+1]_q\,[n]_q}
\begin{bmatrix}
  n+1\\n_1,n_2,\dots
\end{bmatrix}_q
\Bigg\vert_{q=\om^e}\\
=
\begin{cases}
\frac {2n-n_1} {n(n+1)}
\left(\begin{smallmatrix}
    n+1\\n_1,n_2,\dots
  \end{smallmatrix}\right),
&\kern-3.2pt
\text{if }e=0,\\[.5em]
\frac {2n-n_1} {n+1}
\left(\begin{smallmatrix}
    \frac {n+1} 2\\\frac{n_1}2,\frac{n_2}2,\dots
  \end{smallmatrix}\right),
&\kern-3.2pt
\text{if $d=2$ and each $n_i$ is even},
\\[.5em]
\frac {2n-n_1} {n}
\left(\begin{smallmatrix}
\frac {n} {d}\\{\frac {n_1}{d}},
\dots,\frac {n_{\ell-1}}{d},\frac {n_\ell-1}{d},\frac {n_{\ell+1}}{d},
  \dots
\end{smallmatrix}\right),
&\kern-3.2pt
\text{if }
d\mid (n_\ell-1)
\text{ for some }\ell,
\\[-.5em]
&\kern-3.2pt
\text{while  } d\mid n_i
\text{ for $i\ne \ell$},\\[.5em]
0,&\kern-3.2pt
\text{otherwise.}
\end{cases}
\end{multline}
\end{lemma}
Note that
in the third case we have $d \mid n$, so that if $d=2$ then $n$ is even; the second and third
case are thus disjoint because $n$ is odd in the second case.

\begin{proof}[Proof of Lemma \ref{lem:7}]
The case where $e=0$ is trivial. Therefore we assume from now on that $e$
is a positive divisor of~$2n-n_1$ with $e<2n-n_1$. Then $\om^e$ is a primitive $d$-th root of unity.
According to~\eqref{qint0}, in order to obtain something non-zero upon the specialisation
$q=\om^e$ on the left-hand side of~\eqref{eq:10},
there must be an equal number of factors of the form
$1-q^\al$ with $d\mid \al$
in the numerator and in the
denominator. In other words, we must have
\begin{equation} \label{eq:2n-2-d1/e} 
1+\fl{\frac {n-1} {d}}
-\sum_{{i\ge 1}}
\fl{\frac {n_i} {d}}=0.
\end{equation}
As in the previous proof,
we introduce the fractional parts $N_i=\{\frac {n_i}{d}\}$ for $i\ge1$.
With this notation,
Equation~\eqref{eq:2n-2-d1/e} turns into
\[
1+\fl{\sum_{i\ge1}N_i-\frac {2} {d}}=0.
\]
This equation holds if and only if  $\sum N_i \in \{0, 1/d\}$.

If $\sum N_i=0$, then each $N_i$ is $0$, so that $d$ divides each~$n_i$. Condition~\eqref{tree-cond} then forces $d=2$, and we are in the second case of~\eqref{eq:10}. One then evaluates the left-hand side of~\eqref{eq:10} in this case using the simple fact~\eqref{eq:SM}.

If $\sum N_i=1/d$, then each $N_i$ is $0$ except one, say $N_\ell$, which equals $1/d$. This means  $d$ divides $n_\ell-1$ and each $n_i$
with $i\ne\ell$.
We are thus in the third  case of~\eqref{eq:10}. One then evaluates the left-hand side of~\eqref{eq:10} in this case using the simple fact~\eqref{eq:SM}. It is useful to recall that, due to~\eqref{tree-cond}, $d$~divides~$\ell$ in this case, so that in particular $\ell\neq 1$.
\end{proof}

\section{Connections with earlier work}
\label{sec:connections}
 
So far, we have stated all our cyclic sieving results  in the uniform language  of rooted plane trees.
However, it is well known that these trees are in bijection with many other families of Catalan objects. When we consider a class of trees ($\T(n,k)$, $\T(\bn)$, \dots), a bijection to other objects,   and a rotation on trees ($R$, $R_l$, $R_i$, or $R_\de$) that translates nicely through the bijection,  we obtain a cyclic sieving phenomenon for the image of the class of trees under the bijection. Some of these results are  new, and some are not.
We briefly discuss some examples in this section. We describe the bijections we use rather informally, with the help of
figures. We hope that this will be sufficient for the readers. 

\subsection{Non-crossing perfect matchings}
\label{sec:matchings}
We start with the classical bijection between non-crossing (perfect) matchings of~$2n$ points placed on a circle and rooted plane trees with $n$ edges, illustrated in Figure~\ref{fig:NCM-tree}. Then the (ordinary) rotation~$R$ on trees, of order~$2n$, rotates the matching by $(2\pi)/(2n)$ in clockwise order, that is, replaces the point $i$ by $i+1$, modulo~$2n$.

 \begin{figure}[htb]
    \centering
    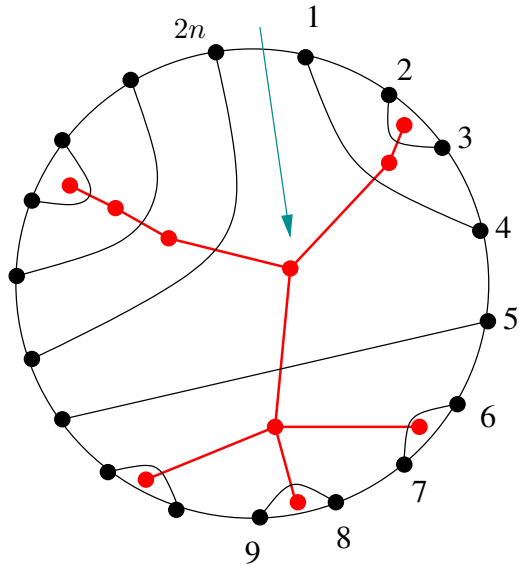  
    \caption{A non-crossing matching on $2\times 9$ elements  and the corresponding rooted plane tree with $n=9$ edges. The number of short edges in the matching is $k=5$.} 
    \label{fig:NCM-tree}
  \end{figure}

  Our first theorem, Theorem~\ref{thm:ord}, thus translates into a cyclic sieving phenomenon for non-crossing perfect matchings, for this natural rotation. This result is not new: it is stated in~\cite[Thm.~8.1]{SagaCS}, and follows for instance by combining a bijection between non-crossing matchings and two-row rectangular Young tableaux with a cyclic sieving phenomenon for these tableaux under \emm promotion,; see Theorems 1.4 and 1.5 in~\cite{Petersen-web} as well as~\cite[Thm.~1.3]{rhoades}.
   Let us rephrase this perfect matching version of Theorem~\ref{thm:ord}  in a way that will be convenient in Section~\ref{sec:maps}.
  We denote by $\NCM(n)$ the set of non-crossing  matchings of $2n$~points.

  \begin{corollary}\label{cor:matchings}
 Let $n$ be a positive integer. Furthermore, let $\om$ be a primitive $2n$-th root of unity, and let $e$ be a non-negative integer less than $2n$, equal to $0$ or dividing  $2n$. Let $d:=2n/e$, with $d=\infty$ if $e=0$. The number of non-crossing matchings on $2n$~points left invariant by a rotation of $(2\pi e)/(2n)=2\pi/d$ is given by
    \begin{align*}
        \Fix_{2\pi/d} (\NCM(n))&=   \frac 1{[n+1]_q}  \qbinom{2n}{n} \Bigg\vert_ {q=\om^e}=   \frac 1{[n+1]_q}  \qbinom{2n}{n} \Bigg\vert_ {q=\e^{2\pi i/d}}
      \\
      &=
        \begin{cases}\frac {1} {n+1}\binom {2n}{n},&\text{if }e=0,\\\binom {n}{\lceil n/2 \rceil},&\text{if } d=2 \text{ with $n$ odd},\\
           \binom {2n/d}{n/d},&\text{if $d \mid n$},\\0,&\text{otherwise.}
      \end{cases}
    \end{align*}
  \end{corollary}

Through the bijection of Figure~\ref{fig:NCM-tree}, the
number of leaves of the tree, $k$ say, corresponds to the number of \emm short edges, of the matching, that is, edges of the form $(i,i+1)$, modulo~$2n$. This means that our second theorem, Theorem~\ref{thm:ord_leaves}, is equivalent to the refinement of the cyclic sieving phenomenon for non-crossing matchings obtained in~\cite[Thm.~4.8]{Alexandersson}. The equivalence is not immediately obvious because of an additional factor $q^{k(k-2)}$ in the cyclic sieving polynomial of the latter theorem: however, we note that this factor takes the value $1$ at $\om^e$ for each value of~$e$ for which at least one tree of  $\T(n,k)$ is fixed by~$R^e$ (cf.\ Lemma~\ref{lem:8}). One advantage of this additional factor is that now the sum of the refined cyclic sieving polynomials coincides with the $q$-Catalan polynomial~\cite[Prop.~4.5]{Alexandersson}:
    \[
      \sum_k  q^{k(k-2)} \cdot \frac {1+q^{n}}{[n-1]_q}
\begin{bmatrix} n-1\\k-2\end{bmatrix}_q
\begin{bmatrix} n\\k\end{bmatrix}_q=
  \frac {1}{[n+1]_q}\begin{bmatrix}
    2n\\n\end{bmatrix}_q.
\]
Using this summation
(which is a special case of Heine's $q$-analogue of Gau{\ss}' hypergeometric
$_2F_1[1]$-summation; see~\cite[Eq.~(1.5.2)]{GaRaAF}),
we could thus have derived Theorem~\ref{thm:ord} from Theorem~\ref{thm:ord_leaves}.

\medskip
Conversely, we can  translate some of  our results in terms of non-crossing matchings. One example is Theorem~\ref{thm:ext}, which translates into a nice cyclic sieving phenomenon  for non-crossing matchings containing the pair $(2n,1)$.

\subsection{Non-crossing partitions}
\label{sec:noncr}

We now consider another classical bijection, this time between non-crossing partitions of $n$~points on a circle~\cite{KrewAC} and rooted plane trees with $n$ edges (Figure~\ref{fig:NCP-tree}, left). It is not immediately obvious to understand how the ordinary rotation~$R$ on trees, of order~$2n$, translates in terms of partitions. However, the double rotation~$R^2$, of order~$n$, simply becomes the rotation by $2\pi/n$ on non-crossing partitions; that is, $i$~is replaced by~$i+1$ modulo~$n$. 

  \begin{figure}[htb]
    \centering
    \includegraphics[width=140mm]{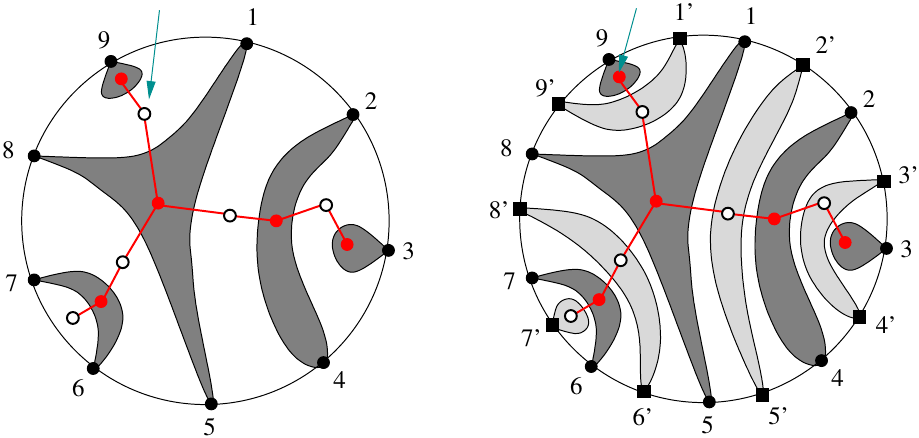}
    \caption{\emm Left,: A non-crossing partition on $n=9$ elements  and the corresponding rooted plane tree with $n$ edges.
    \emm Right,: Taking the Kreweras complement (shown in light grey) of the partition amounts to the ordinary rotation~$R$ on the tree.}
    \label{fig:NCP-tree}
  \end{figure}

  Hence our first theorem, Theorem~\ref{thm:ord}, gives by restriction to even values of~$e$ a (known) cyclic sieving phenomenon for rotation on non-crossing partitions of $n$~points, with cyclic sieving polynomial $\qbinom{2n}{n}_q/[n+1]_q$; see~\cite[Thm.~10.1]{SagaCS}, where this result is derived from a refined cyclic sieving phenomenon where the number of
  blocks of the partition is also fixed~\cite[Thm.~7.2]{ReSWAA}. See~\cite{bessis-reiner}
and Subsection~\ref{sec:part} for generalisations to non-crossing partitions associated with well-generated complex reflection groups (the above classical case corresponding to the symmetric group). 

However, it is also possible to describe
the action of the rotation~$R$ in terms of non-crossing partitions.
Indeed, it corresponds to taking the so-called \emm Kreweras complement, of the partition. This complement is obtained by inserting $n$ new vertices $1', 2', \ldots, n'$ between the original ones and organising them in blocks that correspond to the empty areas of the original partition (see Figure~\ref{fig:NCP-tree}, right).
Thus, our first theorem, Theorem~\ref{thm:ord}, also translates into a
(known) cyclic sieving phenomenon for complementation of non-crossing partitions, proved in~\cite[Thm.~1.5]{armstrong-uniform} for arbitrary finite Weyl groups, and in~\cite{KratCG,KrMuAD} for well-generated complex reflection groups;
see Subsection~\ref{sec:part}. According to Bessis and Reiner, who had conjectured this result~\cite[Conj.~6.5]{bessis-reiner} in the more general setting proved in~\cite{KratCG,KrMuAD}, the classical case was proved earlier by D.~White in a personal communication. Later
it was also established independently by C.~Heitsch~\cite{heitsch}. 
  
  \subsection{Triangulations and dissections of the disk}
  
     Another classical family in the Catalan zoo for which cyclic sieving results have been established is the set of triangulations of a polygon. There exists a simple bijection between leaf-rooted \emm ternary trees, with $k$~leaves and $k-2$ inner nodes of degree~$3$ --- that is, in our notation, trees of $\T_l(\bn)$, with $\bn=(k, 0, k-2, 0, \ldots)$ --- and triangulations of a $k$-gon; see Figure~\ref{fig:triang-tree}. 
     
 \begin{figure}[htb]
    \centering
    \includegraphics[width=60mm]{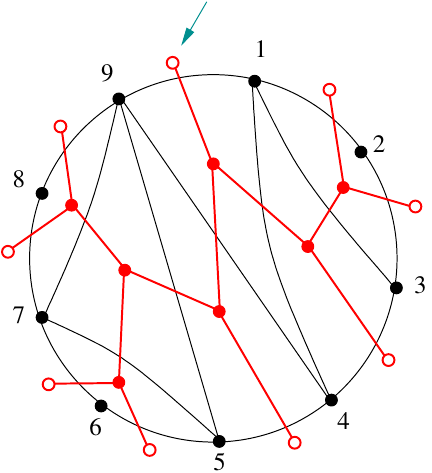}
    \caption{A triangulation of the $k$-gon, with $k=9$,  and the corresponding leaf-rooted ternary tree with $9$~leaves and $7$~inner nodes.}
    \label{fig:triang-tree}
  \end{figure}
 
 Through this bijection, the external rotation~$R_l$, of order~$k$ when considered on trees with $k$~leaves, becomes the rotation of triangulations by $2\pi/k$, which sends each point~$i$
to $i+1$. Hence, a special case of Theorem~\ref{thm:delta} (for $\delta=1$, and $\bn$ as above) gives a cyclic sieving phenomenon for triangulations of the disk, and this phenomenon is not new~\cite[Thm.~12.2]{SagaCS}.

More generally, let us call
any collection of non-intersecting diagonals $(i,j)$ with $|i-j|>1$
a \emm dissection, of the $k$-gon.
Equivalently, a dissection can be seen as an \emm outerplanar map, with faces of degree at least $3$ (see Figure~\ref{fig:dissection}). A simple bijection sends these dissections to leaf-rooted trees having no vertex of degree~$2$. The degree of a face is the degree of the node that lies in it.  Again, the external rotation on such trees corresponds to the clockwise rotation by~$2\pi/k$ of the dissection.  Hence the case $\de=1$, $n_2=0$  of Theorem~\ref{thm:delta} gives a cyclic sieving phenomenon for dissections of the $k$-gon with prescribed face degrees,
which we believed to be new\dots\
until we found it in a very recent independent preprint~\cite[Thm.~1]{adams-dissections}.
In particular, we obtain a cyclic sieving phenomenon for $m$-angulations of the $k$-gon, for $m$ fixed.

  Conversely, the known cyclic sieving phenomenon for dissections of the $k$-gon with $d$~diagonals~\cite[Thm.~7.1]{ReSWAA} gives a cyclic sieving phenomenon for the external rotation on trees with $d+1$ inner nodes and $k$~leaves, in which vertices of degree~$2$ are forbidden.

 \begin{figure}[htb]
    \centering
    \includegraphics[width=60mm]{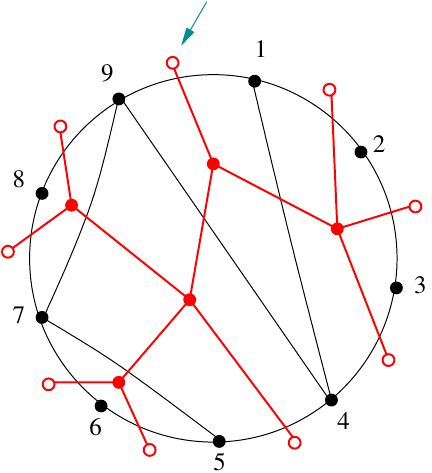}
    \caption{A dissection of the $k$-gon, with $k=9$, having $d=4$ diagonals, and the corresponding leaf-rooted tree with $k=9$ leaves and $d+1=5$ inner nodes. The degree of the faces of the dissection, or equivalently the degrees of the inner nodes of the tree, are $3,3, 3, 4$ and $4$.}
    \label{fig:dissection}
  \end{figure}

\subsection{Non-crossing partitions associated with reflection groups}
\label{sec:part}

The purpose of this subsection is to describe the setting of generalised
non-crossing partitions associated with well-generated complex
reflection groups as introduced by Armstrong~\cite{ArmsAA}, and
to explain how the general cyclic sieving results in~\cite{KratCG,KrMuAD}
cover the cyclic sieving results worked out in Subsections~\ref{sec:matchings}
and~\ref{sec:noncr}
for non-crossing matchings and non-crossing partitions.
Phrased differently, here we make the indications on results for
reflection groups made in these subsections explicit.

In order to do this, we need to recall the definition of {\it
  $m$-divisible non-crossing partitions associated with reflection groups.}
Let $W$ be an {\it irreducible well-generated complex reflection group} and $c$ a
{\it Coxeter element} in~$W$.
While, for our purposes, it is
entirely sufficient to assume that $W=S_n$, the
group of permutations of $\{1,2,\dots,n\}$, and $c=(1,2,\dots,n)$
(that is, the long cycle), we refer the interested
reader to~\cite{LeTaAA} for an excellent introduction into the
theory of complex reflection groups.
The {\it reflection length} $\ell_T(w)$
of an element~$w\in W$ is defined as the
smallest~$\ell$ such that $w=t_1t_2\cdots t_\ell$ with all~$t_i$'s
being reflections in~$W$.
For example, for the case of $W=S_n$ and $c=(1,2,\dots,n)
=(1,2)(2,3)\cdots(n-1,n)$, we have
$\ell_T(c)=n-1$ since the reflections in~$S_n$ are the
transpositions. Given these ingredients, 
Armstrong's~\cite[Sec.~3.3]{ArmsAA} definition of $m$-divisible non-crossing
partitions reads as follows.

\begin{definition}
Let $W$ be a well-generated complex reflection group and $c$ a fixed
Coxeter element in~$W$. For a positive integer~$m$,
the set of $m$-divisible non-crossing partitions for~$W$, denoted by
$NC^{(m)}(W)$, is defined by
\begin{multline*}
NC^{(m)}(W)=\big\{(w_0;w_1,\dots,w_m):w_0w_1\cdots w_m=c\text{ and }\\
\ell_T(w_0)+\ell_T(w_1)+\dots+\ell_T(w_m)=\ell_T(c)\big\}.
\end{multline*}
\end{definition}

\begin{remarks}
(1)
In the case where $m=1$, we write $NC(W)$ instead of $NC^{(1)}(W)$
for short and call this the set of (ordinary) non-crossing partitions
associated with the reflection group~$W$. Moreover, an element
$(w_0;w_1)$ in $NC(W)=NC^{(1)}(W)$ (whose product equals the Coxeter
element~$c$) is usually identified with~$w_0$. 

\medskip
(2)
We have suppressed the dependence on the Coxeter element~$c$
in the notation, since here
we are only interested in enumeration. Indeed,
it can be shown that any two Coxeter elements are related to each other
by conjugation and (possibly) an automorphism on the field of complex
numbers
(see \cite[Thm.~4.2]{SpriAA} or \cite[Cor.~11.25]{LeTaAA}).
It follows that any two sets $NC^{(m)}(W)$ corresponding to
different Coxeter elements
can be bijectively mapped to each other by these transformations. 
\end{remarks}

In the case of $W=S_n$, Armstrong \cite[Thm.~4.3.8]{ArmsAA} described
a translation of the above abstract definition of $m$-divisible
non-crossing partitions into the combinatorial language of
non-crossing partitions according to Kreweras
(those of Subsection~\ref{sec:noncr}).
The resulting
partitions are the $m$-divisible non-crossing partitions of
$\{1,2,\dots,mn\}$ of Edelman~\cite{EdelAA}, which are defined by the
property that all block sizes are divisible by~$m$.
Given an element
 $(w_0;w_1,\dots,w_m) \in NC^{(m)}(S_n)$,
the bijection, $\Na_{S_n}$ say, 
from \cite[Thm.~4.3.8]{ArmsAA} 
works as follows: for $i=1,2,\dots,m$, let $\ta_{m,i}$ be
the transformation that maps a permutation $w\in S_{n}$ to a
permutation $\ta_{m,i}(w)\in S_{mn}$ by letting
\[(\ta_{m,i}(w))(mk+i-m)=mw(k)+i-m,\quad k=1,2,\dots,n,\] 
and 
$(\ta_{m,i}(w))(l)=l$ for all $l\nequiv i$~(mod~$m$).
With the choice of Coxeter element $c=(1,2,\dots,n)$, 
the announced bijection maps 
$(w_0;w_1,\dots,w_m)\in NC^{(m)}(S_{n})$ to
\[ 
\Na_{S_n}(w_0;w_1,\dots,w_m)=
(1,2,\dots,mn)\,(\ta_{m,1}(w_1))^{-1}\,(\ta_{m,2}(w_2))^{-1}\,\cdots\,
(\ta_{m,m}(w_m))^{-1},
\]
where the cycles in the disjoint cycle decomposition
of the above permutation correspond to
the blocks in the non-crossing partition of
$\{1,2,\dots,mn\}$.
We refer the reader to \cite[Sec.~4.3.2]{ArmsAA} for the details.

\begin{example} \label{ex:9}
Let $n=7$, $m=3$, $w_0=(4,5,6)$,
$w_1=(3,6)$,
$w_2=(1,7)$, and
$w_3=(1,2,6)$. Then $(w_0;w_1,w_2,w_3)$ is mapped to
\begin{align*}
\Na_{S_{7}}(w_0;w_1,w_2,w_3)&=(1,2,\dots,21)\,(7,16)\,
(2,20)\,(18,6,3)\\
&=(1,2,21)\,(3,19,20)\,(4,5,6)\,(7,17,18)\,(8,9,\dots,16),
\end{align*}
which gives a partition of $\{1, \ldots, 21\}$ with one block of size~$9$ and four blocks of size~$3$. \qee
\end{example}

If $m=1$,
then $\Na_{S_n}$ takes a pair
$(w_0;w_1)\in NC^{(1)}(S_n)=NC(S_n)$, ignores~$w_1$, and interprets~$w_0$
as a partition of $\{1,2,\dots,n\}$ by converting each cycle
in the disjoint cycle decomposition of~$w_0$ into a block containing
the numbers involved in that cycle. It was already shown by
Biane~\cite{BianAA} that this produces indeed a non-crossing partition.
Conversely, one recovers $w_0$ from the partition by associating to each  block 
 a cycle in which the elements of the block are ordered clockwise; in other words, each cycle of~$w_0$ is \emm increasing,, meaning that it can be written $(i_1, i_2, \ldots, i_k)$ for an increasing sequence $i_1<i_2 <\cdots < i_k$.

\begin{example}
Let $m=1$, $W=S_9$,
\[ 
w_0=(1,5,8)(2,4)(6,7)
\quad \text{and}\quad
w_1=(1,4)(2,3)(5,7)(8,9).
\]
Then we have $w_0w_1=(1,2,\dots,9)=c$,
and thus the above described identification produces the non-crossing
partition $\{\{1,5,8\},\{2,4\},\{3\},\{6,7\},\{9\}\}$, that is,
the non-crossing partition shown on the left of Figure~\ref{fig:NCP-tree}. \qee
\end{example}

Armstrong \cite[Thm.~4.3.13]{ArmsAA} has shown that the cycle
structure of the first
component of an $m$-divisible non-crossing partition
determines the block structure of its image under his bijection.

\begin{proposition} \label{prop:block}
Let $(w_0;w_1,\dots,w_m)\in NC^{(m)}(S_{n})$. The non-crossing partition\break
$\Na_{S_{n}}(w_0;w_1,\dots,w_m)$ of $\{1,2,\dots,mn\}$ has $b_i$ blocks of
size~$mi$ if and only if $w_0$ has $b_i$~cycles of length~$i$
in its disjoint cycle decomposition. 
\end{proposition}

In Example~\ref{ex:9}, we have $w_0=(4,5,6)=(1)(2)(3)(7)(4,5,6)$ which has one cycle
of length~3 and four cycles of length~1. Indeed, the image of
$(w_0;w_1,w_2,w_3)$ in the example has one block of size $3\cdot 3=9$
and four blocks of size $3\cdot 1=3$.

\medskip
We introduced the concept of Kreweras complement in
Subsection~\ref{sec:noncr} (by way of an example; see Figure~\ref{fig:NCP-tree}). 
As shown by Kreweras \cite[p.~339]{KrewAC}, this may be equivalently
realised
in the language of permutations.
Namely, given a non-crossing partition~$\pi$,
and the associated permutation $w_0$ constructed as above,
the Kreweras complement of~$\pi$ is the partition associated with
the permutation $c w_0^{-1}$,
where
$c=(1,2,\dots,n)$ is again the long cycle.
It is useful to note that the cycles of $c w_0^{-1}$  are also increasing;
see Example~\ref{ex:complement} below.

More generally,
the \emm Kreweras complement for well-generated
complex reflection\break groups~$W$, is defined as the map
\beq\label{def:complement}
w\mapsto cw^{-1}, \quad \text{for }w\in NC(W),
\eeq
under the identification of $NC(W)$ with the first component
of the tuples $(w_0;w_1)$ lying in $NC(W)=NC^{(1)}(W)$.

\begin{example}\label{ex:complement}
Consider the non-crossing partition
\[
\{\{1,5,8\},\{2,4\},\{3\},\{6,7\},\{9\}\}
\]
on the left of Figure~\ref{fig:NCP-tree}.
We have
\begin{align*}
(1,2,\dots,9) \big(  (1,5,8)(2,4)(6,7)\big)^{-1}
&=(1,2,\dots,9)   (1,8,5)(2,4)(6,7)\\
&= (1,9)(2,5)(3,4)(6,8).
\end{align*}
We see that this is indeed the Kreweras complement indicated by
the light-grey blocks on the right of Figure~\ref{fig:NCP-tree}. \qee
\end{example}

Algebraically, the two-fold application of the Kreweras complement is
\[
w\mapsto c\left(cw^{-1}\right)^{-1}=cwc^{-1},
\]
that is, conjugation by the Coxeter element~$c$.
Hence, in the special case where $W=S_n$, the two-fold application of
the Kreweras complement acts by mapping each number~$i$ to~$i+1$
modulo~$n$, as we already observed in Subsection~\ref{sec:noncr}.

\medskip
We are now able to discuss the known cyclic sieving phenomena for
Armstrong's\break $m$-divisible non-crossing partitions for reflection
groups. Namely,
Armstrong defined an action on
$m$-divisible non-crossing partitions in~\cite[Def.~3.4.15]{ArmsAA} by 
\[
A:    (w_0;w_1,\dots,w_m) \mapsto (v;c  w_m  c^{-1}
,w_1,\dots,w_{m-1}),
\]
where $v =
(c  w_m  c^{-1})  w_0 (c  w_m  c^{-1})^{-1}$.
Armstrong shows that this action generates a cyclic group of
order~$mh$, where $h$ is the {\it Coxeter number} of~$W$, that is,
the largest of the {\it invariant degrees} of~$W$.
For $W=S_n$ the Coxeter
number is~$n$. Therefore, in the case of the symmetric group~$S_n$,
the above action has order~$mn$. Using the bijection~$\Na_{S_n}$, this
can also be seen from the fact that the action~$A$ in terms of
Edelman's $m$-divisible non-crossing
partitions is given by {\it rotation}, i.e.,
every element~$i$ in a partition is replaced by~$i+1$, taken
modulo~$mn$.
In particular, when $m=1$ we recover the rotation on non-crossing partitions that was identified with $R^2$ in Subsection~\ref{sec:noncr}.  

Putting results from~\cite{KratCG} and~\cite{KrMuAD} together, we
have the following cyclic sieving phenomenon.

\begin{theorem} \label{thm:TW1}
Let $W$ be an irreducible well-generated complex reflection group. The triple $(NC^{(m)}(W),\langle A\rangle,\Cat^{(m)}(W;q))$
exhibits the cyclic sieving phenomenon, where
\[
\Cat^{(m)}(W;q):=\prod_ {i=1}^n \frac {[mh+d_i]_q} {[d_i]_q},
\]
$n$ is the rank of $W$,  $(d_1,d_2,\dots,d_n)$  the list of invariant degrees of~$W$, and $h=d_n$  the Coxeter number of~$W$.
\end{theorem}

\begin{remark}
We refer the reader again to~\cite{LeTaAA} for the definitions of
the rank and the
invariant degrees. For our discussion, the only fact that is relevant
is that 
\begin{equation} \label{eq:CatW} 
\Cat^{(m)}(S_n;q)=\frac {1}
{[n]_q}\begin{bmatrix}(m+1)n\\n-1 \end{bmatrix}_q.
\end{equation}
For $m=1$, this is 
MacMahon's
$q$-Catalan polynomial, and one recovers the cyclic sieving phenomenon for non-crossing partitions under rotation mentioned in Subsection~\ref{sec:noncr}.
\end{remark}

In the case where $w_0 = \text{id}$, the action~$A$ reduces to
\begin{equation} \label{eq:action0}
A:    (\text{id};w_1,\dots,w_m) \mapsto (\text {id};c  w_m  c^{-1}
,w_1,\dots,w_{m-1}). 
\end{equation}

Combining again results from~\cite{KratCG} and~\cite{KrMuAD}, we
obtain a second cyclic sieving phenomenon.

\begin{theorem} \label{thm:TW2}
Let $W$ be an irreducible well-generated complex reflection group. Let $NC^{(m;0)}(W)$ denote the subset of $NC^{(m)}(W)$ consisting
of those elements for which $w_0=\text{\em id}$.
Then the triple $(NC^{(m;0)}(W),\langle A\rangle,\Cat^{(m-1)}(W;q))$
exhibits the cyclic sieving phenomenon.
\end{theorem}

Let us consider the case where $m=2$ and $W=S_n$, that is, the
$2$-divisible non-crossing partitions $(\text{id};w_1,w_2)$ for~$S_n$.
By Proposition~\ref{prop:block}, the map $\Na_{S_{n}}$ sends them to partitions of $\{1, 2, \ldots, 2n\}$ in which all blocks have cardinality $2$, hence to non-crossing \emm matchings,. We will come back to this at the end of this section. For the moment, observe that
 there is an alternative point of view on these objects:
forgetting $w_0=\id$,
these are the pairs $(w_1,w_2)$ with
$w_1   w_2 = c$ and $\ell_T(w_1) + \ell_T(w_2) = \ell_T(c) =
n-1$.
So we simply recover non-crossing partitions of $NC(S_n)$.

Obviously, $w_1$ determines $w_2$. The action~\eqref{eq:action0}
reduces to
\[
    (w_1,w_2) \mapsto (c  w_2  c^{-1}  ,w_1),
\]
or, if we focus on the first component,
\[
 w_1 \mapsto c  w_2  c^{-1}
 = c (w_1^{-1}    c)  c^{-1}   = c  w_1^{-1}   .
\]
As we can see, this is exactly the Kreweras complement
defined in~\eqref{def:complement}.
Hence, the special case of Theorem~\ref{thm:TW2} where $m=2$ and
$W=S_n$ is equivalent to the cyclic sieving result for
non-crossing partitions under the action of the Kreweras
complement mentioned in Subsection~\ref{sec:noncr}
(and, thus, also to Theorem~\ref{thm:ord}).

\medskip
Finally,
in order to understand the connection with the non-crossing (perfect)
matchings of Subsection~\ref{sec:matchings},
we still consider the $2$-divisible non-crossing partitions\break $(\id; w_1, w_2)$ for~$S_n$, but now we map them using $\Na_{S_n}$ to non-crossing partitions in which all blocks have size $2$; in other words, to non-crossing
matchings of $2n$ points.
More precisely, the blocks of the matching associated with $(\id; w_1, w_2)$ can be seen to be the pairs $\{2k, 2 w_1(k)-1\}$, for $1\le k \le n$, or equivalently the pairs $\{2 w_2(k), 2k+1\}$.
Now apply Armstrong's action:
\[
     (\text{id};w_1,w_2) \mapsto (\text{id};cw_2c^{-1}  ,w_1).
\]
The  matching corresponding to the new $2$-divisible partition  has blocks $\{2w_1(k), 2k+1\}$,
which means that the action $A$ translates into standard rotation on matchings (that is, shifting everything by one unit). 
Hence, we see that the special case of 
Theorem~\ref{thm:TW2} where $m=2$ and $W=S_n$ is also equivalent to
Corollary~\ref{cor:matchings}
(and, thus again, to Theorem~\ref{thm:ord}).

  \begin{example}
  Let $n=9$,
$w_1=(1,5,8)(2,4)(6,7)$ and 
$w_2=(1,4)(2,3)(5,7)(8,9)$.
Then the matching associated with $(\id;w_1, w_2)$ by $\Na_{S_9}$ has blocks
\[
  \{ 1,16 \},   \{ 2,9 \},   \{ 3,8 \},   \{4 ,7 \},   \{ 5,6 \},   \{10 ,15 \},   \{11 ,14 \},   \{12 ,13 \},   \{17 ,18 \}.
\]
One can check that the blocks of the matching associated with $ (\text{id};cw_2c^{-1}  ,w_1)$ are
\medskip
\hbox to \hsize{\hfill
  $\{ 2,17 \},   \{ 3,10 \},   \{ 4,9 \},   \{5 ,8 \},   \{ 6,7 \},   \{11 ,16 \},   \{12 ,15 \},   \{13 ,14 \},   \{18 ,1 \}$.\hfil\qee}
\end{example}

\section{Tree-rooted planar maps and their cyclic sieving}
\label{sec:maps}

The aim of this section is to establish cyclic sieving phenomena for \emm tree-rooted maps,, that is, planar maps equipped with a distinguished spanning tree. A rotation that generalises the ordinary rotation~$R$ acts on these objects.

We begin with some definitions, and then prove three cyclic sieving results;
see Theorems~\ref{thm:TMij}, \ref{thm:TMn}, and~\ref{thm:TMd} below.
The existence of such phenomena was actually the original motivation of this paper. To finish, we translate these results  in terms of cubic maps equipped with a Hamiltonian cycle, and in terms of lattice walks confined to a quadrant.

\subsection{Definitions}
\label{sec:defs-maps}

A (\emph{planar}) \emph{map} is a proper embedding of a connected planar graph in the
oriented sphere, considered up to orientation preserving
homeomorphism. Loops and multiple edges are allowed
(see Figure~\ref{fig:example-map}, left). The \emph{faces} of a map are the
connected components of its complement. The numbers of
vertices, edges, and faces of a planar map~$\m$, denoted by~$\vv(\m)$,
$\ee(\m)$, and~$\ff(\m)$, are related by Euler's relation
\[
  v(\m)+\ff(\m)=\ee(\m)+2.
\]
The \emph{degree} of a vertex
is the number
of half-edges incident to it (a loop thus 
contributes~$2$ to the degree of the vertex where it is attached).
A \emph{corner} is
a sector delimited by two consecutive half-edges around a vertex;
hence  a vertex
of degree~$k$ is incident to~$k$~corners.
A \emm rooted, map is a map with a distinguished corner.
It is equivalent to distinguish and orient one edge (the one that follows the root corner in counterclockwise order, oriented away from the root vertex).
All our maps are planar and rooted. We often draw them in the plane by putting the root corner in the infinite face.

\begin{figure}[h]  \centering 
  \includegraphics{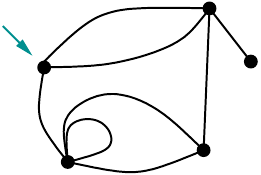}
\hskip 20mm
  \includegraphics{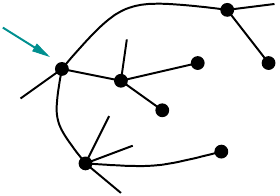}
  \caption{\emm Left,: a rooted planar map.
    \emm Right,: a rooted b-tree with $8$ vertices and $6$ buds.}
  \label{fig:example-map}
\end{figure}

Observe that a rooted tree, as defined in Section~\ref{sec:enum},  is just a rooted one-face map. A \emm b-tree,  is a tree with additional half-edges, called \emm buds,,  attached to the vertices; see 
the right of Figure~\ref{fig:example-map} for an example.
A bud contributes $1$ to the degree of the vertex to which it is attached.
Adding a bud to a vertex also increases
the number of corners by~$1$.
Note that there is no corner at the open (vertex-less) endpoint of a bud. A \emm rooted, b-tree is a b-tree with a distinguished corner.

A \emm tree-rooted map,, in Mullin's terminology~\cite{mullin-boisees}, is a rooted map with a distinguished spanning tree. A tree-rooted map with $i+1$ vertices and $j+1$ faces (and, thus, $n=i+j$ edges) can be decomposed into a rooted b-tree having $i+1$ vertices and $2j$ buds, coupled with a non-crossing matching of the $2j$ buds (see Figure~\ref{fig:tree-rooted}). This correspondence is bijective. Note that all corners of the map are incident to the spanning tree.

  \begin{figure}[htb]
    \centering
    \includegraphics[width=130mm]{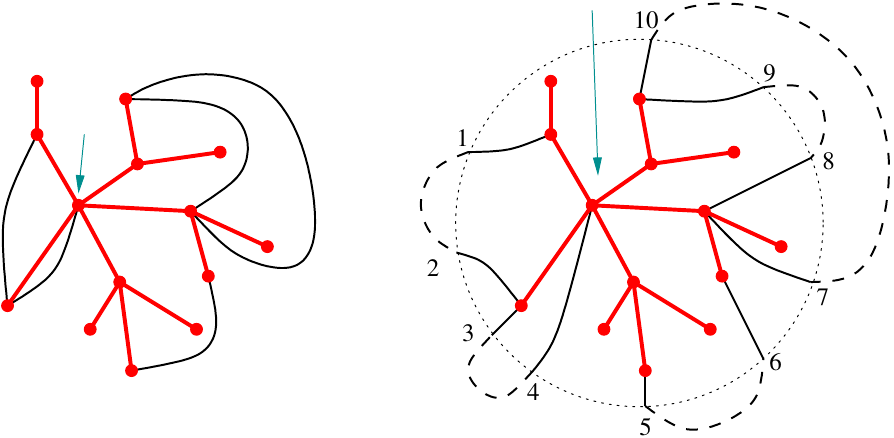}
    \caption{A tree-rooted map with $i+1=14$ vertices and $j+1=6$ faces (left). It can be seen as a rooted b-tree with $14$ vertices and $2j=10$ buds, coupled with a non-crossing matching of the $10$ buds (shown in dashed lines on the right).}
    \label{fig:tree-rooted}
  \end{figure}

\medskip
On tree-rooted maps, we introduce
an extension of the ordinary rotation~$R$ that we defined on rooted trees in Section~\ref{sec:rotations}.
By definition, it acts by moving the root corner to the next corner
of the map
around the spanning tree, in counterclockwise order; see an example in Figure~\ref{fig:rotation-tree-rooted}.
(In order to understand why the third orbit has cardinality $2$ only,
we have to remember
that maps are drawn \emm on the sphere,.) On tree-rooted maps with $n$ edges,
this extended rotation, which we still denote by~$R$,
realises an action of the cyclic group of order~$2n$.
On the other hand, the rotation~$R$ acts similarly on b-trees. Given that b-trees  with $n+1$ nodes and $b$ buds have $2n+b$~corners, this rotation gives an action of the cyclic group of order~$2n+b$ on such trees.
We shall make use of this separate consideration as a step towards
the proofs of Theorems~\ref{thm:TMij} and~\ref{thm:TMd}, see
Propositions~\ref{prop:BTij} and~\ref{thm:BTd}.
  
  \begin{figure}[htb]
    \centering
    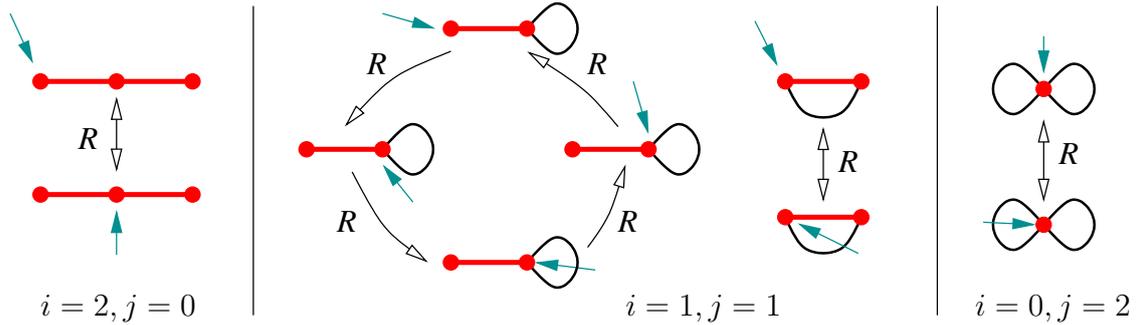
    \caption{The $10$ tree-rooted maps with $n=2$~edges ($i+1$~vertices, $j+1$~faces), and the action of the rotation~$R$ on these maps. The spanning trees are shown in thick red lines.}
    \label{fig:rotation-tree-rooted}
  \end{figure}

\subsection{Tree-rooted maps with prescribed number of vertices and faces}

For non-negative integers $i$ and $j$,
we denote by $\TM(i,j)$ the set of tree-rooted maps having $i+1$ vertices and $j+1$ faces. (Note the shift in the variables $i$ and $j$; this is analogous to $\T(n)$ being the set of rooted trees with $n$ edges, hence $n+1$ vertices.)

\begin{theorem} \label{thm:TMij}
Let $i$ and $j$ be non-negative integers, not both zero. 
Then the cyclic group of order $2i+2j$ acts on $\TM(i,j)$ by the rotation~$R$, and the triple
\[
  \left(\TM(i,j),\langle R\rangle,
   \frac 1 {[i+1]_q[j+1]_q} \qbinom{2i+2j}{i,i,j,j}_q
  \right)
\]
exhibits the cyclic sieving phenomenon.
\end{theorem}

The formula for the cardinality of $\TM(i,j)$
(the specialisation $q=1$ of the above cyclic sieving polynomial)
is due to Mullin~\cite[Eq.~(4.6)]{mullin-boisees}.  The above theorem extends Theorem~\ref{thm:ord}, which corresponds to the case $j=0$.

\begin{example}
  For $i=2, j=0$, one finds the polynomial $1+q^2$, in agreement with the first orbit in Figure~\ref{fig:rotation-tree-rooted}. We obtain the same polynomial (which is expected, by a duality argument) for $i=0,j=2$. In the more interesting case where $i=j=1$, we obtain $P(q)=(1+q^2)(1+q+q^2)$, and one can check that cyclic sieving holds indeed with the help of Figure~\ref{fig:rotation-tree-rooted}. \qee
  \end{example}

Recall that a tree-rooted map is obtained by coupling a b-tree with a perfect matching (cf.\ Figure~\ref{fig:tree-rooted}). Accordingly, we will show that the above cyclic sieving phenomenon results from the combination of two cyclic sieving results, one for b-trees (established in the proposition below), and one for non-crossing matchings (Corollary~\ref{cor:matchings}).

\begin{proposition}\label{prop:BTij}
 Let $n$ and $b$ be non-negative integers, not both zero. Let $\BT(b,n)$ be the set of b-trees with $b$~buds and $n+1$~nodes. The cyclic  group of order $m:=2n+b$ acts on $\BT(b,n)$ by the rotation~$R$, and the triple 
  \[
    \left(\BT(b,n),\langle R\rangle,
    \frac 1 {[n+1]_q}  \qbinom{2n+b}{b,n,n}_q
    \right)
\]
exhibits the cyclic sieving phenomenon.
\end{proposition}
Note that this result specialises to Theorem~\ref{thm:ord} when $b=0$.

\begin{proof}
If $n=0$ then the assertion of the theorem is trivial.
Hence, we assume $n>0$ from now on.

\medskip
  Polynomiality and non-negativity of coefficients of the cyclic
sieving polynomial follow from the rewriting
\[
  \frac 1 {[n+1]_q}  \qbinom{2n+b}{b,n,n}_q
=
\begin{bmatrix} 2n+b\\b\end{bmatrix}_q
\frac {1} {[n+1]_q}\begin{bmatrix} 2n\\n\end{bmatrix}_q
\]
and the well-known facts that $q$-binomial coefficients and
$q$-Catalan numbers are polynomials in~$q$ with non-negative
coefficients; see \cite[Thm.~3.1]{AndrAF} and
the first paragraph of the proof of Theorem~\ref{thm:ord}.

\medskip
Let $\om$ be a primitive $m$-th root of unity, with $m=2n+b$.
We have to prove 
\beq\label{q-BT(b,i)}
\Fix_{R^e}(\BT(b,n))=
  \frac 1 {[n+1]_q}  \qbinom{2n+b}{b,n,n}_q
\Bigg\vert_{q=\om^ e}
\end{equation}
for all integers $e$. It suffices to prove this for values of $e\in \llbracket 0, m-1\rrbracket$ equal to $0$ or dividing $m$. Lemma~\ref{lem:13} below gives the corresponding values of the right-hand side of~\eqref{q-BT(b,i)}.

We begin with $e=0$. To construct a b-tree of $\BT(b,n)$, we start from a rooted tree with $n+1$ nodes (counted by the Catalan number $C_{n}$) and then spread $b$ buds around this tree. When walking around the b-tree, starting from the root corner, one sees a sequence of $2n$~edge-sides and $b$~buds. This sequence is counted by the binomial coefficient $\binom{2n+b}{b}$. Hence
\beq\label{BT(b,i)}
|\BT(b,n)| =\binom{2n+b}{b} \frac 1{n+1} \binom{2n}{n}
= \frac 1{n+1} \binom{2n+b}{b,n,n},
\eeq
which is indeed the value at $q=1$ of~\eqref{q-BT(b,i)}.

\medskip
Now let $e$ be a positive divisor of $m=2n+b$, and let $d=m/e$. Let $\tau^+$ be  a tree  of $\BT(b,n)$, and let $\tau$ be the rooted tree obtained by removing all buds of $\tau^+$.  Then $\tau$ has~$n$ edges. By adapting the proof of Proposition~\ref{prop:RRR}, we obtain that
\[
  R^e(\tau^+)= \tau^+ \quad \text{implies}\quad  R^f(\tau)=\tau,
\]
with $f=2n/d$. In particular, there can exist b-trees fixed by~$R^e$ only if $e=n+\frac b2$ with $n$ odd and $b$ even (that is,
in particular, $f=n$ and $d=2$), or if $m/e$ divides~$n$ (that is, $f$ is even).

\medskip
So let us first take $e=n+\frac b2$, with $b$ and $n+1$ even, and let $\tau^+\in \BT(b,n)$ be fixed by~$R^e$.  The associated tree $\tau$ is fixed by~$R^n$, and thus has a central edge~$c$ by Proposition~\ref{prop:structure}. We now proceed as in Proposition~\ref{prop:structure}.  We return to $\tau^+$, cut  the edge $c$ in the middle, retain the component containing the root, and mark the bud that comes from the edge $c$; this gives a tree $\tau_1^+$ in $\BT\left(1+\frac b2, \frac {n-1}2\right)$, with a marked bud, and the correspondence  $\tau^+ \mapsto \tau_1^+$ is one-to-one. Hence in this case,
\[
  \Fix_{R^e}(\BT(b,n))=\left(1+\frac b 2\right)\left|\BT\left(1+\frac b 2, \frac {n-1} 2 \right)\right|,
\]
which, according to~\eqref{BT(b,i)}, is equivalent to the expression in the second alternative of~\eqref{eq:17}. This proves~\eqref{q-BT(b,i)} in this case.

\medskip Now assume that $d=m/e$ divides~$n$, and let $f=2n/d$.   Let $\tau^+\in \BT(b,n)$ be fixed by~$R^e$.  The associated tree~$\tau$ is fixed by~$R^f$, where $f$ is even, and thus has a central node~$c$ by Proposition~\ref{prop:structure}. We now proceed as in Proposition~\ref{prop:structure}.   Let $\ell$ be
the degree of~$c$ in~$\tau^+$. As in Figure~\ref{fig:Psidef}, we denote
the b-trees attached at the centre by $\tau_1, \ldots, \tau_\ell$
--- some of them may consist of a single bud. Our convention here is that $\tau_1$ contains the edge \emm or the bud, that follows the root corner, in counterclockwise order. Then  $d:=m/e$ divides~$\ell$, and the trees
$\tau_a$ and $\tau_{a+\ell/d}$
are copies of one another. The map~$\Psi_d$, now extended to trees with buds, then sends
trees of  $\BT(b,n)$ fixed by~$R^e$ bijectively 
to trees of $\BT(b/d, n/d)$ having a marked node. Hence
\[
  \Fix_{R^e}(\BT(b,n))= \left( 1+ \frac {n} d\right)  \left|\BT\left(\frac b d, \frac {n} d\right)\right|
  ,
\]
which,  according to~\eqref{BT(b,i)}, is equivalent to the expression in the third alternative of~\eqref{eq:17}. (Note that the right-hand side is zero if $b$ is not a multiple of~$d$.) This proves~\eqref{q-BT(b,i)} in this case, and concludes the proof of the proposition. 
\end{proof}

The following lemma extends Lemma~\ref{lem:1}, which corresponds to the case $b=0$.

\begin{lemma} \label{lem:13}
  Let $n$ and $b$ be non-negative integers, not both zero, and write $m=2n+b$. Furthermore,
let~$\om$ be a primitive $m$-th root of unity, and let $e$ be a
non-negative integer
$<m$ equal to~$0$ or dividing~$m$. In the latter case, write $d:=m/e$.
Then we have
\begin{equation} \label{eq:17} 
  \frac 1 {[n+1]_q}  \qbinom{2n+b}{b,n,n}_q
\Bigg\vert_{q=\om^e}
=
\begin{cases}
  \frac 1{n+1} \binom{2n+b}{b,n,n},
  &\text{if }e=0,\\[.5em]
\left(  \begin{smallmatrix}
    {n+\frac b 2}\\{\frac b 2,\frac {n-1}2,\frac{n+1}2}
  \end{smallmatrix}
  \right),
&\text{if } d=2 \text{ with  $n+1$ and $b$ even},\\[.5em]
\left(\begin{smallmatrix}
  {e}\\
  {\frac b d, \frac n d, \frac n d}
\end{smallmatrix}\right),
&\text{if\/ $d$ divides $n$ and $b$},\\
0,&\text{otherwise.}
\end{cases}
\end{equation}
\end{lemma}

\begin{proof}
The case $e=0$ being obvious, we assume that
$e$ is a positive divisor of $m=2n+b$ with $e<m$.
According to~\eqref{qint0}, in order to obtain a
non-zero value of the specialisation on the left-hand side
of~\eqref{eq:17}, there must be  an equal number of factors of the form
$1-q^\al$ with $d \mid \al$ in the numerator and in the
denominator, with $d\ge2$. In other words, we must have
\begin{equation*} 
\fl{\frac{2n+b}d}
-\fl{\frac b d}-\fl{\frac {n} d}-\fl{\frac {n+1} d}=0.
\end{equation*}
Introducing the fractional parts  $N=\{\frac {n} {d}\}$
and $B=\{\frac {b} d\}$,  this can be rewritten as
\begin{equation*}
2N+B= \fl{N+\frac 1d}.
\end{equation*}
Here, the left-hand side does not need taking an integer
part since, by assumption, we have $d\mid (2n+b)$.
Clearly, the only possible values of the right-hand side in the above equation
are $0$ or~$1$.
In the former case, we infer $N=B=0$, that is,
we have that $d$ divides~$n$ and~$b$, and we are in the third case of~\eqref{eq:17}. If $2N+B=1= \fl{N+\frac 1d}$, we have $N=1-\frac 1d$, so that $B=\frac 2d-1$, which forces $d=2$, with $b$ and $n+1$ even. We are then in the second case of~\eqref{eq:17}. Consequently, in the  fourth case of~\eqref{eq:17}, the value is $0$.

It is then an easy task to evaluate the left-hand side of~\eqref{eq:17} in the second and third case, using the simple fact~\eqref{eq:SM}.
\end{proof}

We can now prove the cyclic sieving phenomenon for tree-rooted maps with $i+1$ vertices and $j+1$ faces.

\begin{proof}[Proof of Theorem~\ref{thm:TMij}.]
  We first note that the expression of the theorem can be written as a product:
  \[ 
    \frac 1{[i+1]_q [j+1]_q} \qbinom{2i+2j}{i,i,j,j}_q
                    =      \qbinom{2i+2j}{2i,2j}_q
                    \times       \frac 1 {[i+1]_q} \qbinom{2i}{i}_q
                    \times       \frac 1 {[j+1]_q} \qbinom{2j}{j}_q.
 \] 
 This is a polynomial with non-negative coefficients because this is the case for each of the three factors.

  Let $n=i+j$ be the number of edges in a
  tree-rooted map of~$\TM(i,j)$.  We have to prove that
  \[
  \Fix_{R^e}(\TM(i,j))=
   \frac 1 {[i+1]_q[j+1]_q} \qbinom{2i+2j}{i,i,j,j}_q
\Bigg\vert_{q=\om^ e}
\]
for all integers $e$. It suffices to prove this for values of $e\in \llbracket 0, 2n-1\rrbracket$ equal to $0$ or dividing $2n$.
  
  For $e=0$, all elements of $\TM(i,j)$ are fixed by~$R^e$. As recalled in Subsection~\ref{sec:defs-maps}, such tree-rooted maps are obtained by combining a non-crossing perfect matching of~$2j$~points, counted by the Catalan number $C_j$, with a b-tree having $i+1$ vertices and $b=2j$ buds; see Figure~\ref{fig:tree-rooted}. According to~\eqref{BT(b,i)}, this gives  
  \[
    |\TM(i,j)|=  \frac 1 {j+1} \binom{2j}{j} \cdot
    \frac 1 {i+1} \binom{2i+2j}{i,i,2j}
    =
    \frac 1{(i+1)(j+1)} \binom{2i+2j}{i,i,j,j},
  \]
  as desired.

  Now let $e\in \llbracket 1 , 2n-1\rrbracket$ be  a divisor of~$2n$, and write $d:=2n/e$. Consider a tree-rooted map $(\m, \tau)$ in $\TM(i,j)$
  that is fixed by~$R^e$, and let $\tau^+$ be the associated b-tree.  Since~$\tau^+$ is simply obtained by opening the edges of~$\m$ that are not in $\tau$, then $\tau^+$ itself is fixed by~$R^e$. Conversely, take a b-tree~$\tau^+$ with $i+1$~vertices and $b=2j$ buds,
  invariant by $R^e$,
  and match its buds with a non-crossing perfect matching $\mathfrak p$; then the resulting tree-rooted map is invariant by $R^e$ if and only if $\mathfrak p$ is invariant by a rotation of $(2\pi)/d$. Hence, using the notation of Corollary~\ref{cor:matchings},
  \begin{align}
     \Fix_{R^e}(\TM(i,j)) &=   \Fix_{R^e}(\BT(2j, i))\cdot \Fix_{{2\pi/d}}(\NCM(j)) \nonumber\\
                          & =
\frac 1{[i+1]_q} \qbinom{2i+2j}{i,i,2j}_q
                            \Bigg\vert_{q=\e^{2 \pi i/d} }\,
                      \times       \frac 1 {[j+1]_q} \qbinom{2j}{j}_q
     \Bigg\vert_{q=\e^{2 \pi i/d} }  \label{CSP-TMij0} \\
                          & =
                            \frac 1{[i+1]_q [j+1]_q} \qbinom{2i+2j}{i,i,j,j}_q
                            \Bigg\vert_{q=\e^{2 \pi i/d} },
                            \label{CSP-TMij}
  \end{align}
where we have used
Proposition~\ref{prop:BTij} (with $b$ replaced by $2j$ and $n$ replaced by $i$) and Corollary~\ref{cor:matchings} (with $n$ replaced by $j$)
to obtain the second line.
By Lemma~\ref{lem:divisors},   this concludes the proof.
\end{proof}

\subsection{Tree-rooted maps with prescribed number of edges}

We now derive
a cyclic sieving result for tree-rooted maps with $n$ edges
from the previous theorem on $\TM(i,j)$,
using simply summation, in the spirit of~\cite{Alexandersson}. We denote
the set of tree-rooted maps with $n$ edges by $\TM(n)$.

\begin{theorem} \label{thm:TMn}
Let $n$ be a positive integer.
Then the cyclic group of order~$2n$ acts on~$\TM(n)$ by the rotation~$R$, and the triple
\[\left(\TM(n),\langle R\rangle,
  \frac 1{[n+1]_q[n+2]_q} \qbinom{2n}{n}_q \qbinom{2n+2}{n+1}_q\right)
\]
exhibits the cyclic sieving phenomenon.
\end{theorem}
The formula for the cardinality of $\TM(n)$
(the specialisation $q=1$ of the above cyclic sieving polynomial)
is due to Mullin~\cite[Eq.~(4.7)]{mullin-boisees}.

\begin{example}
For $n=2$, as in Figure~\ref{fig:rotation-tree-rooted}, we obtain
the polynomial
\[
  P(q)=\left(q^{2}+1\right) \left(q^{4}+q^{3}+q^{2}+q +1\right) \left(q^{2}-q +1\right)= q^{8}+2 q^{6}+q^{5}+2 q^{4}+q^{3}+2 q^{2}+1,
\]
and we can check, by evaluating it at $1, i$ and $ -1$, that cyclic sieving indeed holds. \qee
\end{example}

\begin{remark}
  Below we prove this theorem by summation of (a slight modification of) the cyclic sieving polynomials of Theorem~\ref{thm:TMij}. In the same fashion, we could have derived our first theorem, Theorem~\ref{thm:ord}, from its refinement, Theorem~\ref{thm:ord_leaves}, but we gave instead a direct, sum-free proof of Theorem~\ref{thm:ord}. In contrast, we have no direct proof of the above theorem.
Bernardi~\cite{bernardi-boisees} gave a direct combinatorial proof of the case $q=1$, by constructing  a bijection between tree-rooted maps with $n$ edges and pairs consisting of a tree with $n$ edges and a non-crossing partition of $n+1$ elements.
However, this bijection does not seem to behave nicely when applying the rotation~$R$ to tree-rooted maps.
\end{remark}

\begin{proof}[Proof of Theorem~\ref{thm:TMn}]
  A tree-rooted map with $n$ edges has $i+1$ vertices and $j+1$ faces, for some $i$ and~$j$ that sum to $n$.   Theorem~\ref{thm:TMij} gives a cyclic sieving polynomial for these tree-rooted maps under the rotation~$R$.
Recall  that we have constructed it  as the product of  the cyclic sieving polynomials obtained for b-trees of $\BT(2j,i)$ and for non-crossing perfect matchings of $2j$ elements; cf.~\eqref{CSP-TMij0}
and~\eqref{CSP-TMij}, and the paragraph above these lines.

Now let $e\in\llbracket 0, 2n-1\rrbracket$ be either $0$ or a divisor of~$2n$. Let $d=2n/e$. Then $d\in \{2, 3, \ldots, 2n\} \cup \{\infty\}$. For $\om$ a primitive $2n$-th root of unity, $\om^e$ is a primitive $d$-th root of unity (or $1$, if $e=0$).  By Lemma~\ref{lem:13} (with $b$ replaced by $2j$ and  $n$ replaced by $i$) and~Corollary~\ref{cor:matchings} (with $n$ replaced by $j$), the cyclic sieving polynomial
in the statement of Theorem~\ref{thm:TMij} vanishes at $\om^e$ unless
  \begin{itemize} 
  \item   $e=0$,
  \item or $d=2$,
    \item or $d$ divides both $i$ and $j$.
    \end{itemize}
    This implies that
     \[
                           \frac 1{[i+1]_q [j+1]_q} \qbinom{2i+2j}{i,i,j,j}_q       \cdot  q^{(n+1-i)j}
     \]
     is another cyclic sieving polynomial for $\TM(i,j)$, since the power of~$q$ that we have multiplied with takes the value $1$ at $\om^e$ in all the above cases. (Note that $(n+1-i)j=(j+1)j$ is always even.)

     Hence a cyclic sieving polynomial for $\TM(n)$ is
     \begin{align*}
       \sum_{i+j=n}                           \frac 1{[i+1]_q [j+1]_q} \qbinom{2i+2j}{i,i,j,j}_q
       q^{(n+1-i)j}
       &=  \frac{[2n]_q!}{[n+1]_q!^2}\sum_{i+j=n} \qbinom{n+1}{i} \qbinom{n+1}{j} q^{(n+1-i)j}\\
       &= \frac{[2n]_q!}{[n+1]_q!^2} \qbinom{2n+2}{n},
     \end{align*}
     by the $q$-analogue of the Chu--Vandermonde identity~\cite[Eq.~(3.3.10)]{AndrAF}. This gives the polynomial announced in Theorem~\ref{thm:TMn}. Since it is obtained as a sum of cyclic sieving polynomials, its coefficients are clearly non-negative.
     \end{proof}

     \subsection{Tree-rooted maps with prescribed vertex degrees}

     Let $\bn=(n_1, n_2, \ldots)$ be a sequence of non-negative numbers containing finitely many non-zero entries. Let $\TM(j, \bn)$ be the set of tree-rooted maps having $n_k$ vertices of degree~$k$ for all~$k$, and $j+1$~faces. The b-trees associated with such maps have  $n_k$~vertices of degree~$k$ for all~$k$, and $b=2j$ buds. Adding a node of degree~$1$ at the end of each bud gives a tree with $b+n_1$~leaves  and $n_k$~nodes of degree~$k$ for $k\ge 2$. By~\eqref{tree-cond}, such a tree can only exist if
     \beq\label{tree-cond-bud}
     -b+ \sum_{k\ge 1} (k-2) n_k =-2.
     \eeq
     Assuming this holds (with $b=2j$), a tree-rooted map of $\TM(j, \bn)$ has $i+1$ vertices and~$n$ edges, with
     \[
       i+1:=\sum_k n_k \qquad \text{and} \qquad n:=i+j=j-1+\sum_k n_k.
     \]
     
     The following theorem extends Theorem~\ref{thm:ord_deg}, which corresponds to the case $j=0$.

\begin{theorem} \label{thm:TMd}
Let $j\ge 0$ and $\bn=(n_1, n_2, \ldots)$ satisfy~\eqref{tree-cond-bud} with $b=2j$. Define $n$ as above, and assume $n\ge 1$. The cyclic group of order~$2n$ acts on $\TM(j, \bn)$ by the rotation~$R$, and the triple
\[
  \left(\TM(j,\bn),\langle R\rangle,
\frac {[2n]_q} {[j+1]_q[n+j+1]_q\,[n+j]_q}
\begin{bmatrix}
    n+j+1\\j,j,n_1,n_2,\dots\end{bmatrix}_q\right)
\]
exhibits the cyclic sieving phenomenon.
\end{theorem}
\begin{example}
  Let us take $j=1$, $n_1=n_3=1$, the other $n_k$'s being zero. Then $i=1$, $n=2$, and we obtain the polynomial $P(q)=[4]_q=1+q+q^2+q^3$. We can check that cyclic sieving holds, using the second orbit in Figure~\ref{fig:rotation-tree-rooted}. \qee
\end{example}

Again, this theorem decouples into the cyclic sieving phenomenon for non-crossing matchings of $2j$~points stated
in Corollary~\ref{cor:matchings} (which we apply with $n$ replaced by $j$) and the following cyclic sieving result dealing with b-trees with fixed vertex degrees (which we apply with
 $n$ replaced by $i$ and
$b$ replaced by $2j$). The proof of Theorem~\ref{thm:TMd} then mimics that of Theorem~\ref{thm:TMij}, and we do not repeat the argument.

\begin{proposition} \label{thm:BTd} 
Let $b\ge 0$ and $\bn=(n_1, n_2, \ldots)$ satisfy~\eqref{tree-cond-bud}.
Let $\BT(b,\bn)$ denote the set of all b-trees with $b$~buds and $n_k$~nodes
of degree~$k$, for $k\ge1$. Furthermore, write $n+1:=\sum_i n_i$, and assume $n\ge 1$.
Then the cyclic group of order $m:=2n+b$ acts on $\BT(b,\bn)$ by the rotation~$R$, and the triple
\[
  \left(\BT(b,\bn),\langle R\rangle,
 \frac {[2n+b]_q} {[n+b+1]_q\,[n+b]_q}
\begin{bmatrix}
    n+b+1\\b,n_1,n_2,\dots\end{bmatrix}_q\right)
\]
exhibits the cyclic sieving phenomenon.
\end{proposition}

\begin{proof}
Polynomiality and non-negativity of coefficients of the cyclic
sieving polynomial are proved in Lemma~\ref{lem:14} below.

\medskip
Let $\om$ be a primitive $m$-th root of unity.
We have to prove 
\begin{equation} \label{eq:20}
\Fix_{R^e}(\BT(b,\bn))=
 \frac {[2n+b]_q} {[n+b+1]_q\,[n+b]_q}
\begin{bmatrix}
    n+b+1\\b,n_1,n_2,\dots\end{bmatrix}_q
\Bigg\vert_{q=\om^ e}
\end{equation}
for all integers $e$. It suffices to prove this for values of $e\in \llbracket 0, m-1\rrbracket$ equal to $0$ or dividing $m$. Lemma~\ref{lem:15} below gives the corresponding values of the right-hand side of~\eqref{eq:20}.

\medskip
We begin with $e=0$. To construct a b-tree of $\BT(b,\bn)$, we start from a rooted tree of $\T(b+n_1, n_2, n_3, \ldots)$. We choose $b$ of its leaves, which we convert into buds. This yields a tree with buds, rooted either at a corner (and there are $2n+b$ of them) or at one of the $b$  buds. A double rooting argument then gives
\begin{align}
 \label{BT_deg}
  |\BT(b, \bn)|&=\frac{2n+b}{2n+2b}\binom{b+n_1}{b}|\T(b+n_1, n_2,  \ldots)| \nonumber
                 \\
&  =\frac{2n+b}{(n+b)(n+b+1)}\binom{b+n+1}{b,n_1, n_2, \ldots}
\end{align}
by~\eqref{eq:T(n,d)}. This proves the case $e=0$ of~\eqref{eq:20}.

\medskip
Next we assume that $e>0$, and we refine the argument of the proof of Proposition~\ref{prop:BTij},
taking into account the number of vertices of each degree. We write as usual $d:=m/e=(2n+b)/e$.

If $d=2$ and the number $n+1$ of vertices is even, then we need to count b-trees with $1+\frac b2$ buds, $n_i/2$ nodes of degree~$i$ for $i\ge 1$, and a marked bud. In particular, $b$ must be even. This gives
\[
  \Fix_{R^e} (\BT(b, \bn))= \left(1+ \frac b 2\right) \left|\BT \left(1+ \frac b 2, \frac \bn 2\right)\right|
  = \frac{2n+b}{n+b+1} \binom{\frac{b+n+1}2}{\frac b 2, \frac {n_1} 2, \ldots}
\]
by~\eqref{BT_deg}. This coincides with the expression in the second alternative of~\eqref{eq:19}. This proves~\eqref{eq:20} in this case.

\medskip
Now we assume that $d=m/e$ divides $n$.
Then the b-trees of $\BT(b,\bn)$ fixed by~$R^e$ have a central node, of degree~$\ell$ divisible by~$d$. Moreover, $\ell$~must be the (necessarily unique) integer such that $n_\ell\equiv 1$ (mod $d$) while $d$ divides~$n_i$ for $i\neq \ell$.
By extending the map $\Psi_d$ to b-trees, we obtain
a bijection between b-trees of $\BT(b,\bn)$ fixed by~$R^e$
and
b-trees having 
\begin{itemize}
  \item $\frac bd$ buds,
\item  $\bar n_i:=\frac {n_i} d$ nodes of degree~$i$ for $i\not \in \{\ell/d, \ell\}$,
\item $\bar n_\ell:=\frac{n_\ell-1}d$ nodes of degree~$\ell$,
\item $\bar n_{\ell/d}:= \frac {n_{\ell/d}}d+1$ nodes of degree~$\ell/d$, one of which is marked.
\end{itemize}
This gives
\begin{align*}
  \Fix_{R^e} (\BT(b, \bn))&=\bar n_{\ell/d} \left|\BT\left( \frac b d, \bar n_1, \bar n_2, \ldots\right)\right|\\
  &=\frac{2n+b}{n+b} \binom{\frac{n+b}d}{\frac b d, \frac{n_1} d, \ldots, \frac{n_{\ell-1}}d, \frac{n_\ell-1}d, \frac{n_{\ell+1}}d, \ldots},
\end{align*}
by~\eqref{BT_deg}. This coincides with the expression in the third alternative of~\eqref{eq:19}, and concludes the proof of the proposition.
\end{proof}

\begin{lemma}\label{lem:14}
Let $b\ge 0$ and $\bn=(n_1, n_2, \ldots)$ satisfy~\eqref{tree-cond-bud}. Let $n+1:=\sum_i n_i$.  Then the expression
  \beq\label{eq:Pol5}
  \frac {[2n+b]_q} {[n+b+1]_q\,[n+b]_q}
\begin{bmatrix}
  n+b+1\\b,n_1,n_2,\dots\end{bmatrix}_q
\eeq
is a polynomial in $q$ with non-negative coefficients.
\end{lemma}

\begin{proof}
  Let us introduce numbers $\tilde n_i$ defined by
  \[
    \tilde n_1=b \quad \text{and} \quad  \tilde n_{i+1}=n_i \text{ for } i\ge 1.
  \]
  Define $\tilde n=-1 + \sum _{i\ge1} \tilde n_i$.  Then
  \[
    n= -1+\sum_{i\ge 1} n_i = -1-b +\sum_{i\ge1} \tilde n_i= \tilde n-\tilde n_1.
  \]
  This allows us to rewrite~\eqref{eq:Pol5} as
  \[
    \frac{[2\tilde n-\tilde n_1]_q}{[\tilde n +1]_q [\tilde n]_q}
    \qbinom{\tilde n +1}{\tilde n_1, \tilde n_2, \ldots}.
  \]
  We recognise the expression of Lemma~\ref{lem:6}, with $n_i$
  replaced by $\tilde n_i$. However, we cannot directly apply Lemma~\ref{lem:6}, because the $\tilde n_i$'s do not satisfy~\eqref{tree-cond}. Indeed, in
view of~\eqref{tree-cond-bud},
  \[
    \sum_{i \ge 1} (i-2) \tilde n_i= -b +\sum_{i\ge 1} (i-1) n_i=
    -b +\sum_{i\ge 1} (i-2) n_i +\sum_{i\ge 1} n_i=-2 + (n+1)\neq -2.
  \]
  However, we can reproduce the first steps of the proof of Lemma~\ref{lem:6}, as long as we do not use~\eqref{tree-cond}.
To begin with, we write~\eqref{eq:Pol5} as
    \begin{align}
\notag
      \frac {1} {[\tilde n]_q}&
\begin{bmatrix}
  \tilde n\\ \tilde n_1-1,\tilde n_2,\dots
\end{bmatrix}_q
+q^{\tilde n_1}
\frac {1+q^{\tilde n-\tilde n_1}} {[\tilde n+1-\tilde n_1]_q}\begin{bmatrix} \tilde n-1\\\tilde n_1\end{bmatrix}_q
\begin{bmatrix}
  \tilde n+1-\tilde n_1\\\tilde n_2,\tilde n_3,\dots
\end{bmatrix}_q  \\
\notag
&=\frac {1} {[ n+b]_q}
\begin{bmatrix}
   n+b\\ b-1, n_1, n_2,\dots
\end{bmatrix}_q
+q^{b}
\frac {1+q^{n}} {[ n+1]_q}\begin{bmatrix} n+b-1\\ b\end{bmatrix}_q
\begin{bmatrix}
   n+1\\  n_1, n_2,\dots
 \end{bmatrix}_q
 \\
&= \frac {1} {[ n+b]_q}
\begin{bmatrix}
   n+b\\ b-1, n_1, n_2,\dots
\end{bmatrix}_q
+q^{b}
 [2n]_q\frac {[n+b-1]_q!}{[b]! \prod _{i\ge 1} [n_i]_q!}, \label{split2}
\end{align}
and study the two summands separately.

For the first summand, we cannot directly apply the case $\de=1$ of
Lemma~\ref{lem:11}, again because the $\tilde n_i$ do not
satisfy~\eqref{tree-cond} (cf.\ the first line of~\eqref{split2}).
Thus we revisit the proof of
Lemma~\ref{lem:11}. When doing this,
in order to establish polynomiality of the first
summand, we end up with the following condition (now considering the
last line of~\eqref{split2}):
\[
  j'_d
: =\fl{B+
  {\, \sum_{i\ge1}}
  N_i-\frac {1} {d}} \ge 0 \qquad \text{ for } d\ge 2,
\]
 where $B= \{(b-1)/d\}$ and $N_i=\{n_i/d\}$. This is non-negative
 except if $B=N_i=0$ for all $i$, which means that $b-1$ and all
 $n_i$'s are divisible by $d$. By~\eqref{tree-cond-bud}, this forces
 $d=1$, which is impossible. This establishes polynomiality of the
 first summand in~\eqref{split2}. For non-negativity, we use once
 again Lemma~\ref{lem:unimodal}, here with $n$ replaced by $n+b$,
 $P_0(q)=1$, and $P(q)=\left[\begin{smallmatrix}
   n+b\\ b-1, n_1, n_2,\dots
\end{smallmatrix}\right]_q$.

For the second summand in~\eqref{split2}, we revisit the proof of
Lemma~\ref{lem:6}. This leads us to the following point:
this summand is a polynomial if and only if for all $d\ge 2$, the following quantity is non-negative:
\[ 
    k'_d:= \chi\left(d \mid (2n)\right) + \fl{B + \sum_{i\ge 1} N_i -\frac 2 d},
\] 
  where we have used the fractional parts $B=\{b/d\}$ and $N_i=\{n_i/d\}$. This is the counterpart of~\eqref{eq:k_d2}. The above quantity is non-negative except possibly if $B+\sum N_i$ equals $0$ or $\frac 1 d$. If this holds, the final term above equals $-1$.

  If $B+\sum N_i=0$, that is, if $d$ divides $b$ and each $n_i$, then it follows from~\eqref{tree-cond-bud} that $d=2$, so that
$d\mid (2n)$ and $k'_d=0$.

If $B+\sum N_i=\frac 1 d$, then either $B=\frac 1 d$ and 
all $N_i$'s 
vanish, or $N_\ell=\frac 1 d$ for some $\ell$, while $B=N_i=0$ for $i\neq \ell$. In the first case, $d$ divides~$b-1$ and all~$n_i$'s, and it follows from~\eqref{tree-cond-bud} that $d=1$, which is impossible. In the second case, $d$ divides~$b$, $n_\ell-1$ and all~$n_i$'s for $i\neq \ell$. Since $n=-1+\sum n_i$, we conclude that $d$ divides~$n$, and thus $2n$. Hence $k'_d$ vanishes in this case.

This proves the polynomiality of the second summand
in~\eqref{split2}. Non-negativity follows once more from
Lemma~\ref{lem:unimodal}.

This concludes the proof of the lemma.
\end{proof}

\begin{lemma} \label{lem:15}
Let $b\ge 0$ and $\bn=(n_1, n_2, \ldots)$ satisfy~\eqref{tree-cond-bud}. Let $n+1:=\sum_i n_i$. Write $m:=2n+b$.
Furthermore,
let $\om$ be a primitive $m$-th root of unity, and let $e$ be a
non-negative integer $<m$ being equal to~$0$ or dividing~$m$. In this case, write $d:=m/e$.
Then we have
\begin{multline} \label{eq:19} 
 \frac {[2n+b]_q} {[n+b+1]_q\,[n+b]_q}
\begin{bmatrix}
    n+b+1\\b,n_1,n_2,\dots\end{bmatrix}_q
\Bigg\vert_{q=\om^e}\\
=
\begin{cases}
  \frac{2n+b}{(n+b)(n+b+1)}\binom{b+n+1}{b,n_1, n_2, \ldots},&
  \text{if }e=0,
  \\[.5em]
  \frac{2n+b}{n+b+1}
  \left( \begin{smallmatrix}
    \frac{b+n+1}2\\
  \frac b 2, \frac {n_1} 2, \ldots
  \end{smallmatrix}\right),
  &\text{if $d=2$ with $b$ and each $n_i$ even},\\[.5em]
  \frac{2n+b}{n+b}
  \left(\begin{smallmatrix}
    {\frac{n+b}d}\\
    {\frac b d, \frac{n_1} d, \ldots, \frac{n_{\ell-1}}d, \frac{n_\ell-1}d, \frac{n_{\ell+1}}d, \ldots}  
  \end{smallmatrix}\right),
  &\text{if $d \mid b$,
$d\mid (n_\ell-1)$
for some $\ell$} \\[-.5em]
  &\text{and $d\mid n_i$ for $i\neq \ell$,}
  \\[.3em]
  0, &\text{otherwise.}
\end{cases}
\end{multline}
\end{lemma}

\begin{proof}
The case where $e=0$ is trivial. Therefore we assume from now on that $e$
is a positive divisor of~$m$ with $e<m$. Then $\om^e$ is a primitive $d$-th root of unity. According to~\eqref{qint0}, in order to obtain a
non-zero value of the specialisation on the left-hand side
of~\eqref{eq:19}, there must be  an equal number of factors of the form
$1-q^\al$ with $d \mid \al$ in the numerator and in the
denominator. In other words, we must have
\[
  1+ \fl{\frac{n+b-1}d} - \fl{\frac b d} -\sum_i \fl{\frac {n_i}d}=0.
\]
Introducing the fractional parts $B:=\{ b/d\}$ and $N_i=\{n_i/d\}$, 
this may be rewritten as
\[
 1+ \fl{B + \sum N_i-\frac  2 d} =0.
\]
This forces $B+\sum N_i$ to be contained in $\{0, \frac 1 d\}$.

If $B+\sum N_i =0$, that is, $B=N_i=0$ for all $i$, then $d$ divides~$b$ and each $n_i$.
However, by assumption $d$ divides~$2n+b$, hence $d$ divides $2n= 2\sum n_i -2$. This forces $d=2$, and we are in the second case of~\eqref{eq:19}.

If $B+\sum N_i=1/d$ with $B=1/d$ and $N_i=0$ for all $i$, then $d$ divides~$b-1$ and each $n_i$, so it divides $n+1=\sum n_i$ as well.
However, by assumption $d$ divides $2n+b=2(n+1)+(b-1)-1$, which means that $d=1$, which is impossible.

Finally, if $B+\sum N_i=1/d$ with $N_\ell=1/d$, $B=0$ and $N_i=0$ for $i\neq \ell$, then $d$ divides $b$, $n_\ell-1$ and each $n_i$ for $i\neq \ell$. Hence, we are in the third case of~\eqref{eq:19}.

It is then an easy task to evaluate the left-hand side of~\eqref{eq:19} in the second and third case, using the simple fact~\eqref{eq:SM}.
\end{proof}

\subsection{Two reformulations}
\subsubsection{Cubic maps with a Hamiltonian cycle}

There exists a simple bijection between tree-rooted maps
and \emm cubic, maps
(maps in which all vertices have degree~$3$)
equipped with a \emm Hamiltonian cycle, containing the root edge.
(By definition, a Hamiltonian cycle visits exactly once each vertex.) This bijection is illustrated in Figure~\ref{fig:hamilton} and works as follows: starting from a tree-rooted map $(\m,\tau)$, say with $i+1$ vertices and $j+1$ faces, and drawn with a circle surrounding $\tau$ as we did in Figure~\ref{fig:tree-rooted} (right), one  forms a non-crossing matching dual to the $i$ edges of~$\tau$, using the classical construction of Figure~\ref{fig:NCM-tree}.  This matching is shown in dashed red lines in Figure~\ref{fig:hamilton}. Note that its edges must intersect only one edge of~$\m$ (and this edge is in $\tau$). Now one has $2i+2j=2n$ points on the surrounding circle. The cubic map then consists of these~$2n$~vertices, together with the edges of the surrounding circle (which forms the distinguished Hamiltonian cycle), plus the $i$ edges of the matching, and finally the $j$ edges that lie outside the circle. The root edge of the cubic map is chosen to be the unique edge of the cycle that faces the root corner of~$\m$, oriented in counterclockwise direction.

  \begin{figure}[htb]
    \centering
    \includegraphics[width=130mm]{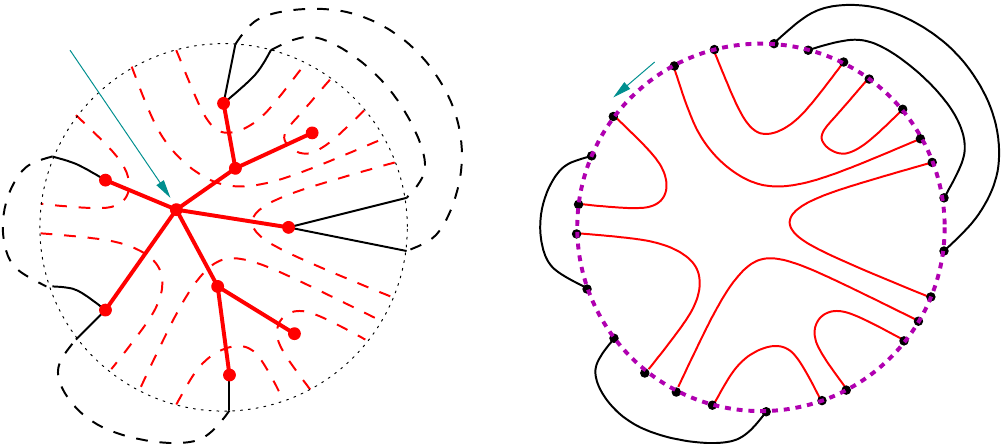}
    \caption{A tree-rooted map with $i+1=10$ vertices and $j+1=5$ faces (left), and the corresponding cubic map with a distinguished Hamiltonian cycle (containing the root edge). The rotation~$R$ on the tree-rooted map moves the root edge of the cubic map along the distinguished cycle.}
    \label{fig:hamilton}
  \end{figure}

  Note that a tree-rooted map with $i+1$ vertices, $j+1$ faces, $n$ edges, $n_k$ nodes of degree~$k$ is sent
to a cubic map (with a distinguished Hamiltonian cycle) having~$2n$~vertices, $i$~edges inside the cycle, $j$~edges outside the cycle, and $n_k$~\emm inner, faces (those lying inside the distinguished cycle) that contain
$k$~edges of the cycle.

It is not hard to see that the rotation~$R$ on tree-rooted maps just moves the root-edge of the cubic map one step in counterclockwise direction around the cycle. Our three cyclic sieving theorems then translate naturally
into cyclic sieving phenomena for cubic maps with a distinguished Hamiltonian cycle.

\subsubsection{Square lattice walks in the first quadrant}

Let us finally recall the classical bijection between tree-rooted maps with $n$ edges and walks with north, east, south and west steps ($N, E, S, W$) that start and end at $(0,0)$,
and never leave the non-negative quadrant $\N^2$, with $\N=\{0, 1, 2, \ldots\}$. This bijection works as follows: starting from the root corner of the tree-rooted map, and walking around the tree in counterclockwise direction, one encodes each edge of the tree by $E$ (respectively~$W$) when it is visited for the first (respectively second) time, and each edge not in the tree by $N$ (respectively~$S$) when it is visited for the first (respectively second) time, until one is back to the root corner. For instance, the encoding of the tree-rooted map of Figure~\ref{fig:hamilton} starts $ENWESNWE\ldots$ If the map has $i+1$ vertices and $j+1$ faces, then the walk has $i$ east steps and $j$ north steps.

The effect of the rotation~$R$ on walks is not hard to describe: if a tree-rooted map is encoded by the walk (or word) $w=a w_1 \bar a w_2$, where $a\in \{N, E\}$ and $\bar a \in \{S, W\}$ is the step that \emm matches, the first letter of~$w$, then applying $R$ gives the walk $w_1 a w_2 \bar a$. To be more precise: if $a=N$, then $a  w_1 \bar a$ is the shortest non-empty prefix of~$w$ that ends on the $x$-axis; if $a=E$, then $a  w_1 \bar a$ is the shortest non-empty prefix of~$w$ that ends on the $y$-axis. Theorem~\ref{thm:TMn} (respectively Theorem~\ref{thm:TMij}) thus translates into a cyclic sieving phenomenon for quadrant walks with $2n$~steps (respectively with $i$~east steps and $j$~north steps).

Note that this cyclic action on $NSEW$ quadrant walks is similar to another cyclic action defined on quadrant walks with $S$, $W$, and $N\!E$~steps, known as \emm Kreweras walks,, for which a cyclic sieving phenomenon has been conjectured; see~\cite[Conj.~6.2]{hopkins-rubey} and~\cite{GaKr}.

\medskip

\bibliographystyle{abbrv}
\bibliography{sieving.bib}

@Unpublished{adams-dissections,
  author = 	 {Ashleigh Adams and Esther Banaian},
  title = 	 {The Cyclic Sieving Phenomenon and frieze patterns},
  note = 	 {\href{https://arxiv.org/abs/2509.17258}{arXiv:2509.17258}. To appear in \emph{Journal of Algebra and its Applications}},
  OPTkey = 	 {},
  OPTmonth = 	 {},
  OPTyear = 	 {},
  OPTannote = 	 {}
}

@article {BianAA,
    AUTHOR = {Biane, Philippe},
     TITLE = {Some properties of crossings and partitions},
   JOURNAL = {Discrete Math.},
  FJOURNAL = {Discrete Mathematics},
    VOLUME = {175},
      YEAR = {1997},
    NUMBER = {1-3},
     PAGES = {41--53},
      ISSN = {0012-365X,1872-681X},
   MRCLASS = {05A18 (05C25 06A07)},
  MRNUMBER = {1475837},
MRREVIEWER = {Konrad\ Engel},
       DOI = {10.1016/S0012-365X(96)00139-2},
       URL = {https://doi.org/10.1016/S0012-365X(96)00139-2},
}

@article {GaKr,
    AUTHOR = {Gatzweiler, M. and Krattenthaler, C.},
     TITLE = {A positivity conjecture for a quotient of {$q$}-binomial
              coefficients},
   JOURNAL = {Ramanujan J.},
  FJOURNAL = {Ramanujan Journal. An International Journal Devoted to the
              Areas of Mathematics Influenced by Ramanujan},
    VOLUME = {69},
      YEAR = {2026},
    NUMBER = {1},
     PAGES = {13},
      ISSN = {1382-4090,1572-9303},
   MRCLASS = {05A30 (11B65 33D05)},
  MRNUMBER = {5004693},
       DOI = {10.1007/s11139-025-01285-2},
       URL = {https://doi.org/10.1007/s11139-025-01285-2},
       note = {\href{https://arxiv.org/abs/2502.06032}{arXiv:2502.06032}},
}

@book {GaRaAF,
    AUTHOR = {Gasper, George and Rahman, Mizan},
     TITLE = {Basic hypergeometric series},
    SERIES = {Encyclopedia of Mathematics and its Applications},
    VOLUME = {96},
   EDITION = {Second},
      NOTE = {With a foreword by Richard Askey},
 PUBLISHER = {Cambridge University Press, Cambridge},
      YEAR = {2004},
     PAGES = {xxvi+428},
      ISBN = {0-521-83357-4},
   MRCLASS = {33Dxx (05A30 05E35 33-01 33-02)},
  MRNUMBER = {2128719},
MRREVIEWER = {Shaun\ Cooper},
       DOI = {10.1017/CBO9780511526251},
       URL = {https://doi.org/10.1017/CBO9780511526251},
}

@article {hopkins-rubey,
    AUTHOR = {Hopkins, Sam and Rubey, Martin},
     TITLE = {Promotion of {K}reweras words},
   JOURNAL = {Selecta Math. (N.S.)},
  FJOURNAL = {Selecta Mathematica. New Series},
    VOLUME = {28},
      YEAR = {2022},
    NUMBER = {1},
     PAGES = {Paper No. 10, 38},
      ISSN = {1022-1824,1420-9020},
   MRCLASS = {05E18 (05E10 06A07)},
  MRNUMBER = {4346507},
MRREVIEWER = {Vedrana\ Mikuli\'c{} Crnkovi\'c},
       DOI = {10.1007/s00029-021-00714-6},
       URL = {https://doi.org/10.1007/s00029-021-00714-6},
       note = {\href{https://arxiv.org/abs/2005.14031}{arXiv:2005.14031} \href{https://doi.org/10.1007/s00029-021-00714-6}{[doi]}},
}

@book {flajolet-sedgewick,
    AUTHOR = {Flajolet, P. and Sedgewick, R.},
     TITLE = {Analytic combinatorics},
 PUBLISHER = {Cambridge University Press},
   ADDRESS = {Cambridge},
      YEAR = {2009},
     PAGES = {xiv+810},
      ISBN = {978-0-521-89806-5},
   MRCLASS = {05-02 (05A15 05A16 60C05 60E10 82-01)},
  MRNUMBER = {MR2483235},
}

@book {AndrAF,
    AUTHOR = {Andrews, George E.},
     TITLE = {The theory of partitions},
    SERIES = {Cambridge Mathematical Library},
      NOTE = {Reprint of the 1976 original},
 PUBLISHER = {Cambridge University Press, Cambridge},
      YEAR = {1998},
     PAGES = {xvi+255},
      ISBN = {0-521-63766-X},
   MRCLASS = {11P81 (05A17 11P82 11P83)},
  MRNUMBER = {1634067},
}

@article {Alexandersson,
    AUTHOR = {Alexandersson, Per and Linusson, Svante and Potka, Samu and
              Uhlin, Joakim},
     TITLE = {Refined {C}atalan and {N}arayana cyclic sieving},
   JOURNAL = {Comb. Theory},
  FJOURNAL = {Combinatorial Theory},
    VOLUME = {1},
      YEAR = {2021},
     PAGES = {Paper No. 7, 53},
      ISSN = {2766-1334},
   MRCLASS = {05E18 (05A19 05A30)},
  MRNUMBER = {4396212},
MRREVIEWER = {Ivica\ Martinjak},
       DOI = {10.5070/C61055513},
       URL = {https://doi.org/10.5070/C61055513},
       note ={\href{https://arxiv.org/abs/2010.11157}{arXiv:2010.11157} \href{https://doi.org/10.5070/C61055513}{[doi]}},
}

@incollection {AndrCB,
    AUTHOR = {Andrews, George E.},
     TITLE = {The {F}riedman--{J}oichi--{S}tanton monotonicity conjecture at
              primes},
 BOOKTITLE = {Unusual applications of number theory},
    SERIES = {DIMACS Ser. Discrete Math. Theoret. Comput. Sci.},
    VOLUME = {64},
     PAGES = {9--15},
 PUBLISHER = {Amer. Math. Soc., Providence, RI},
      YEAR = {2004},
      ISBN = {0-8218-2703-0},
   MRCLASS = {11P81 (05A17 33D15)},
  MRNUMBER = {2063197},
MRREVIEWER = {David\ M.\ Bressoud},
       DOI = {10.1090/dimacs/064/02},
       URL = {https://doi.org/10.1090/dimacs/064/02},
       note = {\href{https://doi.org/10.1090/dimacs/064/02}{[doi]}},
}

@article {ArmsAA,
    AUTHOR = {Armstrong, Drew},
     TITLE = {Generalized noncrossing partitions and combinatorics of
              {C}oxeter groups},
   JOURNAL = {Mem. Amer. Math. Soc.},
  FJOURNAL = {Memoirs of the American Mathematical Society},
    VOLUME = {202},
      YEAR = {2009},
    NUMBER = {949},
     PAGES = {x+159},
      ISSN = {0065-9266,1947-6221},
      ISBN = {978-0-8218-4490-8},
   MRCLASS = {05-02 (05A17 05E15 05E18 06A06 20F55)},
  MRNUMBER = {2561274},
       DOI = {10.1090/S0065-9266-09-00565-1},
       URL = {https://doi.org/10.1090/S0065-9266-09-00565-1},
       note = {\href{https://arxiv.org/abs/math/0611106}{arXiv:math/0611106} \href{https://doi.org/10.1090/S0065-9266-09-00565-1}{[doi]}},
}

@article {armstrong-uniform,
    AUTHOR = {Armstrong, Drew and Stump, Christian and Thomas, Hugh},
     TITLE = {A uniform bijection between nonnesting and noncrossing
              partitions},
   JOURNAL = {Trans. Amer. Math. Soc.},
  FJOURNAL = {Transactions of the American Mathematical Society},
    VOLUME = {365},
      YEAR = {2013},
    NUMBER = {8},
     PAGES = {4121--4151},
      ISSN = {0002-9947,1088-6850},
   MRCLASS = {05A05 (05A19 20F55)},
  MRNUMBER = {3055691},
MRREVIEWER = {Alessandro\ Conflitti},
       DOI = {10.1090/S0002-9947-2013-05729-7},
       URL = {https://doi.org/10.1090/S0002-9947-2013-05729-7},
       note = {\href{https://arxiv.org/abs/1101.1277}{arXiv:1101.1277} \href{https://doi.org/10.1090/S0002-9947-2013-05729-7}{[doi]}},
}

@article {bernardi-boisees,
    AUTHOR = {Bernardi, O.},
     TITLE = {Bijective counting of tree-rooted maps and shuffles of
              parenthesis systems},
   JOURNAL = {Electron. J. Combin.},
  FJOURNAL = {Electronic Journal of Combinatorics},
    VOLUME = {14},
      YEAR = {2007},
    NUMBER = {1},
     PAGES = {Research Paper 9, 36 pp. (elect.)},
      ISSN = {1077-8926},
   MRCLASS = {05C30 (05A15)},
  MRNUMBER = {MR2285813 (2007m:05125)},
MRREVIEWER = {Linfan Mao},
note = {\href{https://arxiv.org/abs/math/0601684}{arXiv:math/0601684} \href{https://doi.org/10.37236/928}{[doi]}},
}

@article {bessis-reiner,
    AUTHOR = {Bessis, David and Reiner, Victor},
     TITLE = {Cyclic sieving of noncrossing partitions for complex
              reflection groups},
   JOURNAL = {Ann. Comb.},
  FJOURNAL = {Annals of Combinatorics},
    VOLUME = {15},
      YEAR = {2011},
    NUMBER = {2},
     PAGES = {197--222},
      ISSN = {0218-0006,0219-3094},
   MRCLASS = {20F55 (05A18 05E15)},
  MRNUMBER = {2813511},
MRREVIEWER = {T.\ Kyle\ Petersen},
       DOI = {10.1007/s00026-011-0090-9},
       URL = {https://doi.org/10.1007/s00026-011-0090-9},
       note = {\href{https://arxiv.org/abs/math/0701792}{arXiv:math/0701792} \href{https://doi.org/10.1007/s00026-011-0090-9}{[doi]}},
}

@article {CaChAA,
    AUTHOR = {Carrell, Sean R. and Chapuy, Guillaume},
     TITLE = {Simple recurrence formulas to count maps on orientable
              surfaces},
   JOURNAL = {J. Combin. Theory Ser.~A},
  FJOURNAL = {Journal of Combinatorial Theory. Series~A},
    VOLUME = {133},
      YEAR = {2015},
     PAGES = {58--75},
      ISSN = {0097-3165,1096-0899},
   MRCLASS = {05C30 (05C10)},
  MRNUMBER = {3325628},
MRREVIEWER = {Jun-Liang\ Cai},
       DOI = {10.1016/j.jcta.2015.01.005},
       URL = {https://doi.org/10.1016/j.jcta.2015.01.005},
       note = {\href{https://arxiv.org/abs/1402.6300}{arXiv:1402.6300} \href{https://doi.org/10.1016/j.jcta.2015.01.005}{[doi]}},
}

@article {EdelAA,
    AUTHOR = {Edelman, Paul H.},
     TITLE = {Chain enumeration and non-crossing partitions},
   JOURNAL = {Discrete Math.},
  FJOURNAL = {Discrete Mathematics},
    VOLUME = {31},
      YEAR = {1980},
    NUMBER = {2},
     PAGES = {171--180},
      ISSN = {0012-365X,1872-681X},
   MRCLASS = {05A17},
  MRNUMBER = {583216},
MRREVIEWER = {Daniel\ I. A. Cohen},
       DOI = {10.1016/0012-365X(80)90033-3},
       URL = {https://doi.org/10.1016/0012-365X(80)90033-3},
       note = {\href{https://doi.org/10.1016/0012-365X(80)90033-3}{[doi]}},
}

@article {FuHoAA,
    AUTHOR = {F\"urlinger, J. and Hofbauer, J.},
     TITLE = {{$q$}-{C}atalan numbers},
   JOURNAL = {J. Combin. Theory Ser.~A},
  FJOURNAL = {Journal of Combinatorial Theory. Series~A},
    VOLUME = {40},
      YEAR = {1985},
    NUMBER = {2},
     PAGES = {248--264},
      ISSN = {0097-3165,1096-0899},
   MRCLASS = {05A30 (05A15)},
  MRNUMBER = {814413},
MRREVIEWER = {Ira\ Gessel},
       DOI = {10.1016/0097-3165(85)90089-5},
       URL = {https://doi.org/10.1016/0097-3165(85)90089-5},
       note = {\href{https://doi.org/10.1016/0097-3165(85)90089-5}{[doi]}},
}

@Unpublished{heitsch,
  author = 	 {C. Heitsch},
  title = 	 {Counting orbits under {K}reweras complementation},
  note = 	 {Preprint; \href{arXiv:2303.12240}{arXiv:2303.12240}},
  OPTkey = 	 {},
  OPTmonth = 	 {},
  OPTyear = 	 {},
  OPTannote = 	 {}
}

@article {KratAF,
    AUTHOR = {Krattenthaler, Christian},
     TITLE = {Counting lattice paths with a linear boundary. {II}.\
              {$q$}-ballot and {$q$}-{C}atalan numbers},
   JOURNAL = {\"Osterreich. Akad. Wiss. Math.-Natur. Kl. Sitzungsber. II},
  FJOURNAL = {\"Osterreichische Akademie der Wissenschaften
              Mathematisch-Naturwissenschaftliche Klasse. Sitzungsberichte.
              Abteilung II. Mathematische, Physikalische und Technische
              Wissenschaften},
    VOLUME = {198},
      YEAR = {1989},
    NUMBER = {4--7},
     PAGES = {171--199},
      ISSN = {0723-9319,1728-0540},
   MRCLASS = {05A30 (05A15)},
  MRNUMBER = {1064033},
}

@Unpublished{KratCG,
  author = 	 {C. Krattenthaler},
  title = 	 {Non-crossing partitions on an annulus},
  note = 	 {In preparation},
  OPTkey = 	 {},
  OPTmonth = 	 {},
  OPTyear = 	 {},
  OPTannote = 	 {}
}

@incollection {KrMuAD,
    AUTHOR = {Krattenthaler, Christian and M\"uller, Thomas W.},
     TITLE = {Cyclic sieving for generalised non-crossing partitions
              associated with complex reflection groups of exceptional type},
 BOOKTITLE = {Advances in combinatorics},
     PAGES = {209--247},
 PUBLISHER = {Springer, Heidelberg},
      YEAR = {2013},
      ISBN = {978-3-642-30978-6; 978-3-642-30979-3},
   MRCLASS = {11P81 (20F55)},
  MRNUMBER = {3363972},
MRREVIEWER = {Ping\ Ding},
note = {\href{https://arxiv.org/abs/1001.0028}{arXiv:1001.0028}},
}

@Unpublished{KrStAA,
  author = 	 {C. Krattenthaler and C. Stump},
  title = 	 {Positive $m$-divisible
  non-crossing partitions and their {K}reweras maps},
  note = 	 {Preprint; \href{https://arxiv.org/abs/2506.14996}{arXiv:2506.14996}},
  OPTkey = 	 {},
  OPTmonth = 	 {},
  OPTyear = 	 {},
  OPTannote = 	 {}
}

@article {KrewAC,
    AUTHOR = {Kreweras, G.},
     TITLE = {Sur les partitions non crois\'ees d'un cycle},
   JOURNAL = {Discrete Math.},
  FJOURNAL = {Discrete Mathematics},
    VOLUME = {1},
      YEAR = {1972},
    NUMBER = {4},
     PAGES = {333--350},
      ISSN = {0012-365X,1872-681X},
   MRCLASS = {05A17 (06A20)},
  MRNUMBER = {309747},
MRREVIEWER = {Robin\ J.\ Wilson},
       DOI = {10.1016/0012-365X(72)90041-6},
       URL = {https://doi.org/10.1016/0012-365X(72)90041-6},
       note = {\href{https://doi.org/10.1016/0012-365X(72)90041-6}{[doi]}},
}

@book {LeTaAA,
    AUTHOR = {Lehrer, Gustav I. and Taylor, Donald E.},
     TITLE = {Unitary reflection groups},
    SERIES = {Australian Mathematical Society Lecture Series},
    VOLUME = {20},
 PUBLISHER = {Cambridge University Press, Cambridge},
      YEAR = {2009},
     PAGES = {viii+294},
      ISBN = {978-0-521-74989-3},
   MRCLASS = {20F55},
  MRNUMBER = {2542964},
MRREVIEWER = {O.\ V.\ Shvartsman},
}

@article {mullin-boisees,
    AUTHOR = {Mullin, R. C.},
     TITLE = {On the enumeration of tree-rooted maps},
   JOURNAL = {Canad. J. Math.},
  FJOURNAL = {Canadian Journal of Mathematics. Journal Canadien de
              Math\'ematiques},
    VOLUME = {19},
      YEAR = {1967},
     PAGES = {174--183},
      ISSN = {0008-414X},
   MRCLASS = {05.65},
  MRNUMBER = {MR0205882 (34 \#5708)},
MRREVIEWER = {W. T. Tutte},
   DOI = {10.4153/CJM-1967-010-x},
       URL = {https://doi.org/10.4153/CJM-1967-010-x},
       note = {\href{https://doi.org/10.4153/CJM-1967-010-x}{[doi]}},
}

@article {rhoades,
    AUTHOR = {Rhoades, Brendon},
     TITLE = {Cyclic sieving, promotion, and representation theory},
   JOURNAL = {J. Combin. Theory Ser. A},
  FJOURNAL = {Journal of Combinatorial Theory. Series A},
    VOLUME = {117},
      YEAR = {2010},
    NUMBER = {1},
     PAGES = {38--76},
      ISSN = {0097-3165,1096-0899},
   MRCLASS = {05E10},
  MRNUMBER = {2557880},
MRREVIEWER = {Anna\ Stokke},
       DOI = {10.1016/j.jcta.2009.03.017},
       URL = {https://doi.org/10.1016/j.jcta.2009.03.017},
       note={\href{https://arxiv.org/abs/1005.2568}{arXiv:1005.2568} \href{https://doi.org/10.1016/j.jcta.2009.03.017}{[doi]}},
}

@article {Petersen-web,
    AUTHOR = {Petersen, T. Kyle and Pylyavskyy, Pavlo and Rhoades, Brendon},
     TITLE = {Promotion and cyclic sieving via webs},
   JOURNAL = {J. Algebraic Combin.},
  FJOURNAL = {Journal of Algebraic Combinatorics. An International Journal},
    VOLUME = {30},
      YEAR = {2009},
    NUMBER = {1},
     PAGES = {19--41},
      ISSN = {0925-9899,1572-9192},
   MRCLASS = {05E10},
  MRNUMBER = {2519848},
MRREVIEWER = {Gregory\ S.\ Warrington},
       DOI = {10.1007/s10801-008-0150-3},
       URL = {https://doi.org/10.1007/s10801-008-0150-3},
       note = {\href{https://arxiv.org/abs/0804.3375}{arXiv:0804.3375} \href{https://doi.org/10.1007/s10801-008-0150-3}{[doi]}},
}

@article {ReSWAA,
    AUTHOR = {Reiner, V. and Stanton, D. and White, D.},
     TITLE = {The cyclic sieving phenomenon},
   JOURNAL = {J. Combin. Theory Ser.~A},
  FJOURNAL = {Journal of Combinatorial Theory. Series~A},
    VOLUME = {108},
      YEAR = {2004},
    NUMBER = {1},
     PAGES = {17--50},
      ISSN = {0097-3165,1096-0899},
   MRCLASS = {05A15 (05A30 05E99 20B05 20F55)},
  MRNUMBER = {2087303},
MRREVIEWER = {Timothy\ Y.\ Chow},
       DOI = {10.1016/j.jcta.2004.04.009},
       URL = {https://doi.org/10.1016/j.jcta.2004.04.009},
       note = {\href{https://doi.org/10.1016/j.jcta.2004.04.009}{[doi]}},
}

@article {sagan-congruences,
    AUTHOR = {Sagan, Bruce E.},
     TITLE = {Congruence properties of {$q$}-analogs},
   JOURNAL = {Adv. Math.},
  FJOURNAL = {Advances in Mathematics},
    VOLUME = {95},
      YEAR = {1992},
    NUMBER = {1},
     PAGES = {127--143},
      ISSN = {0001-8708,1090-2082},
   MRCLASS = {11B65 (05A30 11B73)},
  MRNUMBER = {1176155},
MRREVIEWER = {F.\ T.\ Howard},
       DOI = {10.1016/0001-8708(92)90046-N},
       URL = {https://doi.org/10.1016/0001-8708(92)90046-N},
}

@incollection {SagaCS,
    AUTHOR = {Sagan, Bruce E.},
     TITLE = {The cyclic sieving phenomenon: a survey},
 BOOKTITLE = {Surveys in combinatorics 2011},
    SERIES = {London Math. Soc. Lecture Note Ser.},
    VOLUME = {392},
     PAGES = {183--233},
 PUBLISHER = {Cambridge Univ. Press, Cambridge},
      YEAR = {2011},
      ISBN = {978-1-107-60109-3},
   MRCLASS = {05-02 (05A05 05E10 05E15)},
  MRNUMBER = {2866734},
MRREVIEWER = {T.\ Kyle\ Petersen},
note = {\href{https://arxiv.org/abs/1008.0790}{arXiv:1008.0790}},
}

@article {SpriAA,
    AUTHOR = {Springer, T. A.},
     TITLE = {Regular elements of finite reflection groups},
   JOURNAL = {Invent. Math.},
  FJOURNAL = {Inventiones Mathematicae},
    VOLUME = {25},
      YEAR = {1974},
     PAGES = {159--198},
      ISSN = {0020-9910,1432-1297},
   MRCLASS = {20H15},
  MRNUMBER = {354894},
MRREVIEWER = {R.\ Steinberg},
       DOI = {10.1007/BF01390173},
       URL = {https://doi.org/10.1007/BF01390173},
       note = {\href{https://doi.org/10.1007/BF01390173}{[doi]}},
}

@book {StanBI,
    AUTHOR = {Stanley, R. P.},
     TITLE = {Enumerative combinatorics. {V}ol. 2},
    SERIES = {Cambridge Studies in Advanced Mathematics},
    VOLUME = {62},
      OPTNOTE = {With a foreword by Gian-Carlo Rota and appendix 1 by Sergey
              Fomin},
 PUBLISHER = {Cambridge University Press},
   ADDRESS = {Cambridge},
      YEAR = {1999},
     PAGES = {xii+581},
      ISBN = {0-521-56069-1; 0-521-78987-7},
   MRCLASS = {05A15 (05-02 05E05 05E10 68R05)},
  MRNUMBER = {MR1676282 (2000k:05026)},
MRREVIEWER = {Ira Gessel},
}

@book {StanBZ,
    AUTHOR = {Stanley, Richard P.},
     TITLE = {Catalan numbers},
 PUBLISHER = {Cambridge University Press, New York},
      YEAR = {2015},
     PAGES = {viii+215},
      ISBN = {978-1-107-42774-7; 978-1-107-07509-2},
   MRCLASS = {05-01 (01A05 11B75 11B83)},
  MRNUMBER = {3467982},
MRREVIEWER = {David\ Callan},
       DOI = {10.1017/CBO9781139871495},
       URL = {https://doi.org/10.1017/CBO9781139871495},
       note = {\href{https://doi.org/10.1017/CBO9781139871495}{[doi]}},
}

@article {TutteTrees,
    AUTHOR = {Tutte, W. T.},
     TITLE = {The number of planted plane trees with a given partition},
   JOURNAL = {Amer. Math. Monthly},
  FJOURNAL = {American Mathematical Monthly},
    VOLUME = {71},
      YEAR = {1964},
     PAGES = {272--277},
      ISSN = {0002-9890,1930-0972},
   MRCLASS = {05.65 (05.45)},
  MRNUMBER = {186585},
MRREVIEWER = {F.\ Harary},
       DOI = {10.2307/2312183},
       URL = {https://doi.org/10.2307/2312183},
       note = {\href{https://doi.org/10.2307/2312183}{[doi]}},
}

\end{document}